\documentclass[11pt,a4paper,draft]{amsart}
 \usepackage[]{amsmath, amsthm, amsfonts, amssymb, amscd,mathtools}
\usepackage[mathscr]{eucal}
\usepackage{hyperref}
\usepackage{xypic}
\usepackage[usenames,dvipsnames]{color}
\usepackage{todonotes}

\newcommand{\A}{\mathbb{A}}

\newcommand{\C}{\mathbb{C}}

\newcommand{\G}{\mathbb{G}}
\renewcommand{\H}{\mathbb{H}}

\renewcommand{\P}{\mathbb{P}}
\newcommand{\Q}{\mathbb{Q}}
\newcommand{\R}{\mathbb{R}}

\newcommand{\Z}{\mathbb{Z}}

\newcommand{\caA}{\mathcal{A}}

\newcommand{\caC}{\mathcal{C}}
\newcommand{\caD}{\mathcal{D}}

\newcommand{\caH}{\mathcal{H}}

\newcommand{\caL}{\mathcal{L}}

\newcommand{\caP}{\mathcal{P}}

\newcommand{\caZ}{\mathcal{Z}}

\newcommand{\fA}{\mathfrak{A}}
\newcommand{\fB}{\mathfrak{B}}

\newcommand{\fD}{\mathfrak{D}}

\newcommand{\bfj}{{\boldsymbol{j}}}
\newcommand{\fg}{{\mathfrak{g}}}

  \newcommand {\p} {{\prime}}

 \DeclareMathOperator {\coker}{coker}
 \DeclareMathOperator {\Div}{div}
 \DeclareMathOperator{\codim}{codim}
 \DeclareMathOperator{\im}{Im}
 \DeclareMathOperator {\cl}{cl}
 \DeclareMathOperator {\CH}{CH} 
  
\DeclareMathOperator{\Spec}{Spec}

\DeclareMathOperator{\Id}{Id}

\DeclareMathOperator{\tfA}{\tau \fA}
\DeclareMathOperator{\tfD}{\tau \fD}

\DeclareMathOperator{\fDTW}{\fD_{TW}}
\DeclareMathOperator{\tfDTW}{\tau \fD_{TW}}

\DeclareMathOperator{\DB}{\fD}
\DeclareMathOperator{\vol}{Vol}

\DeclareMathOperator{\amap}{\boldsymbol{a}}
\DeclareMathOperator{\bmap}{\boldsymbol{b}}

\newcommand{\Go}{\text{\rm Go}}
\newcommand{\BF}{\text{\rm BF}}
\newcommand{\Be}{\text{\rm Be}}

\newcommand{\TW}{\text{\rm TW}}
\newcommand{\an}{\text{\rm an}}
\newcommand{\sm}{\text{\rm sm}}
\newcommand{\dg}{\text{\rm dg}}

\newcommand{\rat}{\text{\rm rat}}

\newcommand{\Top}{\text{\rm top}}
\newcommand{\sing}{\text{\rm sing}}
\newcommand{\Zar}{\text{\rm Zar}}

\numberwithin{equation}{section}

\theoremstyle{plain}
\newtheorem{prop}{Proposition}[section]

\newtheorem{cor}[prop]{Corollary}
\newtheorem{lem}[prop]{Lemma}
\newtheorem{thm}[prop]{Theorem}

\theoremstyle{definition}
\newtheorem{df}[prop]{Definition}

\newtheorem{notation}[prop]{Notation}

\theoremstyle{remark}
\newtheorem{rmk}[prop]{Remark}
\newtheorem{ex}[prop]{Example}

\begin{document}

\title{Higher Arithmetic Intersection Theory}

\author{Jos\'e Ignacio Burgos}

\address{Instituto de Ciencias Matem\'aticas (CSIC-UAM-UCM-UCM3).
  Calle Nicol\'as Ca\-bre\-ra~15, Campus UAM, Cantoblanco, 28049 Madrid,
  Spain} 
\email{burgos@icmat.es}

\author{Souvik Goswami}

\address{Office 601D, Department of Mathematics\\
Texas A\&M University\\
College Station, TX, USA}
\email{souvikjoy@math.tamu.edu}
\date{\today}
%\thanks{The second author would like to thank in particular the Instituto de Ciencias Matem\'aticas (ICMAT) for generously hosting him while much of the research was carried out.}
\date{\today}
\newif\ifprivate
\privatetrue
\thanks{Burgos Gil was partially supported by the
  MINECO research projects MTM2016-79400-P and SEV-2015-0554 (ICMAT
  Severo Ochoa). Burgos Gil and Goswami were partially supported by
  the FP7-MC-IRSES project num. 612534 (Moduli).}

\begin{abstract}
We give a new definition of higher
arithmetic Chow groups for smooth projective varieties defined over an
arithmetic field, which is similar to Gillet and Soul\'e's definition of
arithmetic Chow groups. We also give a compact description of the
intersection theory of such groups. A consequence of this theory is the
definition of a heigh pairing between two higher algebraic cycles, of
complementary dimensions, whose real regulator
class is zero. This description agrees  with Beilinson's 
height pairing for the classical arithmetic Chow groups. We also give
examples of the higher arithmetric intersection pairing in dimension
zero that, assuming a conjecture by 
Milnor on the independence of the values of the dilogarithm, are
non-zero.  
\end{abstract}

\maketitle

\tableofcontents{}

\section{Introduction}
\label{sec:introduction}

%\newpage
\subsection{Higher arithmetic Chow groups}
Let $X$ be a smooth and projective variety of dimension $d$, defined
over an arithmetic field $F$. In \cite{GilletSoule:ait}, Gillet and Soul\'e
defined the \textit{arithmetic Chow groups} of $X$, denoted by
$\widehat{\CH}^p(X)$. In fact, in \emph{loc. cit.} the arithmetic Chow
groups are defined in much greater generality for regular
quasi-projective schemes over an 
arithmetic ring, but in this paper we will only treat the case of
smooth and projective varieties defined 
over an arithmetic field $F$. The elements of $\widehat{\CH}^p(X)$ are classes
of pairs $(Z, g_Z)$, with $Z$ a codimension $p$ subvariety of $X$ and
$g_Z$ a Green current for $Z$. The groups $\widehat{\CH}^p(X)$
satisfy the following properties (see \cite{Burgos:CDB} for
details): 
\begin{itemize}
\item
They fit into an exact sequence
\begin{equation}\label{eq:1001}
  \CH^p(X,1)\xrightarrow{\rho_{\Be}}\widetilde{\fD}^{2p-1}(X,p)
  \xrightarrow{a}\widehat{\CH}^p(X)\xrightarrow{\zeta }
  \CH^p(X)\rightarrow 0,  
\end{equation}
where the group $\CH^p(X,1)$ is Bloch's higher Chow group \cite{Bloch:achK},
the map $\rho_{\text{Be}}$ is Beilinson's 
regulator, $\fD^\ast(X,p)$ is the Deligne complex computing the
real Deligne cohomology $H^{\ast}_{\fD}(X,\R(p))$ of $X$ and
\begin{displaymath}
  \widetilde{\fD}^{2p-1}(X,p)=\fD^{2p-1}(X,p)/\im d_{\fD}.
\end{displaymath}
\vspace{0.2cm}
\item
There is an intersection pairing
\begin{displaymath}
\widehat{\CH}^p(X)\otimes \widehat{\CH}^q(X)\xrightarrow{\cdot}
\widehat{\CH}^{p+q}(X)
\end{displaymath}
turning $\oplus_{p\geq 0}\widehat{\CH}^p(X)$ into a commutative
graded unitary algebra. Note that, since we are working with varieties
over a field, there is no need to tensor with $\Q$ to have a well
defined product.
\vspace{0.2cm}
\item
If $f\colon X\rightarrow Y$ is a morphism of smooth projective
varieties, then there exists a pullback
$$f^\ast\colon \widehat{\CH}^p(Y)\rightarrow \widehat{\CH}^p(X),$$
and if $f$ is smooth, since it is also proper, there exists a pushforward map
$$f_{\ast}\colon \widehat{\CH}^p(X)\rightarrow \widehat{\CH}^{p-l}(Y),$$
where $l=\dim(X)-\dim(Y)$. The inverse image is an algebra
homomorphism and, together with the direct images, it satisfies the
projection formula.
\vspace{0.2cm}
\item The intersection product, the direct image and the inverse image
  are compatible with the corresponding operations in the classical
  Chow groups. 
\end{itemize}
In parallel with Deligne and Soul\'e's proposal for a higher arithmetic
$K$-theory, Goncharov \cite{Goncharov:prAmc} introduced the higher
arithmetic Chow groups. Temporarily, in this introduction, we denoted them as
$\widehat{\CH}^p(X,n)_{\Go}$. These groups satisfy
$\widehat{\CH}^p(X,0)_{\Go}=\widehat{\CH}^p(X)$ and were 
constructed to extend the exact sequence \ref{eq:1001} to a long exact 
sequence 
\begin{multline*}
\cdots \xrightarrow{\amap} \widehat{\CH}^{p}(X,n)_{\Go} \xrightarrow{\zeta}
  \CH^{p}(X,n) \xrightarrow{\rho}  H_{\fD}^{2p-n}(X,\R(p))
 \xrightarrow{\amap} \cdots \\
  \rightarrow \CH^p(X,1) 
   \xrightarrow{\rho_{\Be}}
\widetilde{\fD}^{2p-1}(X,p)  \xrightarrow{\amap}
\widehat{\CH}^p(X) \xrightarrow{\zeta} \CH^p(X)  \rightarrow 0.
\end{multline*}
The construction of Goncharov uses a morphism of complexes
$$Z^p(X,\ast)_{0}\xrightarrow{\mathcal{P}}\fD^{2p-\ast}_D(X,p),$$
where $Z^p(X,\ast)_{0}$ is the normalized cubical Bloch complex, whose
homology computes the higher Chow groups $\CH^{p}(X,\ast)$, and
$\fD^\ast_D(X,p)$ is the Deligne 
complex of currents, which computes the real Deligne cohomology of
$X$. Due to the use of the Deligne complex of currents, this construction
left open the following questions: 
\begin{enumerate}
\item
Does the composition of the isomorphism
\begin{displaymath}
  K_n(X)_{\Q}\cong
\oplus_{p\geq 0}\CH^p(X,n)_{\Q}
\end{displaymath}
with the morphism induced by
$\mathcal{P}$ agree with the Beilinson regulator?
\item
Can one define a product structure on
$\oplus_{p,n}\widehat{\CH}^p(X,n)_{\Go}$? 
\item
Are there well-defined pull-back morphisms? 
\end{enumerate}
\vspace{0.5cm}
All these questions were answered positively by the first author
together with Feliu \cite{BurgosFeliu:hacg} and with Feliu and Takeda
\cite{BurgosFeliuTakeda}. Namely, in \cite{BurgosFeliu:hacg}, there is
a new construction of higher arithmetic Chow groups, that we
temporarily denote $\widehat {\CH}^{p}(X,n)_{\BF}$. These groups
satisfy $\widehat {\CH}^{p}(X,0)_{\BF}=\widehat {\CH}^{p}(X)$ and fit
in
the same long exact sequence as the groups $\widehat
{\CH}^{p}(X,n)_{\Go}$, but they are defined using only differential
forms with logarithmic singularities. This allows to define
inverse images and products. Latter in \cite{BurgosFeliuTakeda} they
proved the existence of a natural isomorphism
\begin{displaymath}
  \widehat {\CH}^{p}(X,n)_{\BF}\xrightarrow{\simeq}\widehat
  {\CH}^{p}(X,n)_{\Go} 
\end{displaymath}
and the fact that the morphism induced by $\caP$ agrees with
Beilinson's regulator.

Following Takeda's alternative definition of higher arithmetic
$K$-theory (\cite{Takeda:haKt}), one can envision a different
kind of higher arithmetic Chow groups, that we denote by
$\widehat{\CH}^p(X,n)$. These new higher arithmetic Chow groups should
fit in exact sequences of the  form  
\begin{equation}\label{eq:1002}
\CH^p(X,n+1)\xrightarrow{\rho_{\Be}}\widetilde{\fD}^{2p-n-1}(X,p)
\xrightarrow{\amap} \widehat{\CH}^p(X,n)\xrightarrow{\zeta}\CH^p(X,n)\rightarrow
0.
\end{equation}
In particular, for this groups the map
$\widehat{\CH}^p(X,n)\xrightarrow{\zeta}\CH^p(X,n)$ is surjective. 
Such a definition was already proposed in Elisenda
Feliu's thesis (\cite{Feliu:Thesis}), but was not developed
further. The definition given there is based on homological
algebra constructions and is not very well suited for concrete 
computations.

The relationship between the groups $\widehat{\CH}^p(X,n)$ and the
previously defined $\widehat{\CH}^p(X,n)_{\BF}$ or $\widehat{\CH}^p(X,n)_{\Go}$, is as follows. There is a map
\begin{displaymath}
  \omega \colon \widehat{\CH}^p(X,n) \to \fD^{2p-n}(X).
\end{displaymath}
Writing $\widehat{\CH}^p(X,n)^{0}=\ker(\omega )$, we get
\begin{displaymath}
  \widehat {\CH}^{p}(X,n)_{\Go}=\widehat {\CH}^{p}(X,n)_{\BF}=
  \begin{cases}
    \widehat{\CH}^p(X,0),&\text{ for }n=0,\\
    \widehat{\CH}^p(X,n)^{0},&\text{ for }n>0.
  \end{cases}
\end{displaymath}

The aim of this paper is to give a new presentation of higher
arithmetic Chow groups that is closer to the original definition of
arithmetic Chow groups, and to develop an intersection
theory which can be seen as a natural generalization of the
intersection theory of arithmetic Chow groups.

We emphasize that the use of Bloch's cubical complex for higher Chow
groups restricts us to varieties over an arithmetic field. The
question whether one can develop a theory for higher arithmetic Chow
groups for arithmetic varieties over general arithmetic rings, is
still open.

\subsection{The main definitions}

We now give a brief description of the construction of the higher
arithmetic Chow groups, as presented in this paper. For more details
the reader is referred to the main text.

There are several complexes that compute Deligne cohomology of $X$.
Using differential forms, we will consider the
complexes $\fD$, $\fD_{t}$, $\fD_{\TW}$ described in Section
\ref{sec:absol-hodge-deligne}. The first is the simplest 
one and was introduced in \cite{Deligne:dc}, the second is the
closest to the original definition of Deligne cohomology, while the
third, called the Thom-Whitney version of the Deligne complex, is
the more involved, but has the advantage of having a 
product that is graded commutative and associative at the level of
complexes. In order to define the groups $\widehat {\CH}^{p}(X,n)^{0}$
it does not matter which complex one uses, but the groups
$\widehat {\CH}^{p}(X,n)$ depend on the complex of forms used to
define them. Since we want an intersection product that is associative
and graded commutative we will mainly use  the complex $\fD_{\TW}$.

There are
variants of the three complexes computing Deligne cohomology that are
defined using currents. They are denoted denoted $\fD_{D}$,
$\fD_{t,D}$ and $\fD_{\TW,D}$. In particular, $\fD_{\TW,D}$ is a
module over the algebra $\fD_{\TW}$.

The cubical version of Goncharov regulator can be modified easily to
give a morphism of complexes 
\begin{displaymath}
  \caP\colon Z^{p}(X,\ast)_0 \to  \fD_{TW,D}^{2p-\ast}(X,p).
\end{displaymath}
Using this map we can extend the notion of Green current for higher
cycles.

\begin{df}
  Let $Z\in Z^{p}(X,n)_0$ be a cycle, that is $\delta Z=0$, then a
  \emph{$\fD_{\TW}$-Green 
  current} for $Z$ is a current $g_{Z}\in \fD^{2p-n-1}_{\TW,D}(X,p)$
  such that
  \begin{displaymath}
    \caP(Z)+d g_{Z}=[\omega ],
  \end{displaymath}
  with $\omega \in \fD^{2p-n}_{\TW}(X,n)$ a smooth form. The class
  of $g_{Z}$ in
  \begin{displaymath}
    \widetilde{\fD}^{2p-n-1}_{\TW,D}(X,p)=\fD^{2p-n-1}_{\TW,D}(X,p)/\im d_{\fD}
  \end{displaymath}
  is denoted by $\widetilde g_{Z}$ and is called a class of Green
  currents. Given a Green current $g_{Z}$ we write $\omega (g_{Z})$
  for the only form such that 
  \begin{displaymath}
    \caP(Z)+d g_{Z}=[\omega (g_{Z})].
  \end{displaymath}
\end{df}

\begin{df}
  A \emph{codimension $p$ higher arithmetic cycle} is a pair
  $(Z,\widetilde g_{Z})$ 
  with $Z\in Z^{p}(X,n)_{0}$ a cycle and $\widetilde g_{Z}$ a class of
  Green currents for $Z$. The group of codimension $p$ higher
  arithmetic cycles will be denoted by $\widehat {Z}^{p}(X,n)_0$.\\

  A higher arithmetic cycle is called rationally equivalent to zero if
  it is of the form $(\delta T,-\widetilde{\caP}(T))$ for $T\in
  Z^{p}(X,n+1)_{0}$. The subgroup of cycles rationally equivalent to
  zero is denoted $\widehat {Z}^{p}_{\rat}(X,n)_0$. Finally the \emph{higher
  arithmetic Chow groups} are defined as
\begin{displaymath}
  \widehat{\CH}^{p}(X,n)=\widehat {Z}^{p}(X,n)_0/\widehat {Z}^{p}_{\rat}(X,n)_0.
\end{displaymath}
\end{df}
As in the theory of arithmetic Chow groups $\widehat{\CH}^p(X)$, we
show that these groups  are functorial and have a product
structure. The definition above is very well suited to define direct
images and one
can also define inverse images for smooth maps. But in order to define
products and inverse images we need to define the analogue of a Green
form with logarithmic singularities. The main difficulty here is that
the cycle $Z$ and the current $g_{Z}$ live in different spaces.

Let $\square = \P^{1}\setminus \{1\}$ and write $\square^{n}=\square
\times \underset{i}{\cdots}\times \square$. The varieties
$\square^{\cdot}$ form a cocubical scheme (see Section
\ref{cubab}). In particular we have face maps $\delta ^{i}_{j}\colon
\square^{n}\to \square ^{n+1}$, for $j=0,1$ and $i=1,\dots,n+1$.  

An element $Z\in
Z^{p}(X,n)_{0}$ is an algebraic cycle in $X\times \square^{n}$ that
intersects properly all the faces of $X\times \square ^{n}$ with the
extra condition that $Z\in \cap^{n+1}_{i=1}\ker (\delta^i_1)^\ast$, see \S \ref{sec:higher-chow-groups} for
more details. We write $|Z|\subset X\times \square^{n}$ for its
support and let $|Z|_{k}\subset X\times \square^{k}$ be the
codimension $p$ 
subset given as
\begin{displaymath}
  |Z|_{k}=\bigcup_{i_{1},\dots,i_{n-k}} (\delta
  ^{i_{1}}_{0})^{-1}\dots (\delta ^{i_{n-k}}_{0})^{-1}|Z|.
\end{displaymath}

For any quasi-projective variety $Y$ we denote as
$\fD_{\TW,\log}^{\ast}(Y,p)$ the Thom-Whitney version of the Deligne
complex using differential forms with logarithmic singularities at
infinity.

A Green form with logarithmic singularities for a cycle $Z\in
Z^{p}(X,n)_0$ is a staircase that allows 
us to go from the current $\delta _{Z}$ of integration along the cycle $Z$ in $X\times \square^{n}$ to the
smooth form $\omega _{Z}$ on $X$.

\begin{df}
  Given a cycle $Z \in Z^p(X,n)_0$, a  \emph{$\fD_{\TW}$-Green form
    (with logarithmic singularities)} 
for $Z$ is an $n$-tuple
\begin{displaymath}
  \fg_{Z}\coloneqq (g_n, g_{n-1},\cdots , g_0)\in \oplus^0_{k=n}\fD_{\TW,\log}
  ^{2p-n+k-1}(X\times \square ^k\setminus |Z|_{k},p),
\end{displaymath}
with
\begin{displaymath}
  (\delta ^{i}_{1})^{\ast}g_{k}=0,\quad i=1,\dots,k,
\end{displaymath}
and such that, if $n>0$,
\begin{enumerate}
\item $\delta _{Z}+d[g_n]=0$, where $\delta _{Z}$ is viewed as an  element in $\fD_{\TW,D}$
(see Example \ref{exm:4}). 
\item $(-1)^{n-k+1}\delta g_k+dg_{k-1}=0,\quad k=2,\cdots ,n$.
\item $(-1)^{n}\delta g_1+dg_0\in
  \fD^{2p-n}_{\TW}(X,p)$. 
\end{enumerate}
where in the previous equations  $\delta $ is the differential in the
cocubical direction 
\begin{displaymath}
  \delta g_k= \sum _{i=1}^{k}(-1)^{i}(\delta _{0}^{i})^{\ast}g_{k}.
\end{displaymath}
In this case we write $\omega (\fg_{Z})=(-1)^{n}\delta g_1+dg_0$.

While, if $n=0$ the previous conditions collapse to the classical condition
\begin{enumerate}
\item [(1)] $\delta _{Z}+d[g_n]\in [\fD^{2p}_{\TW}(X,p)],$
\end{enumerate}
and write $[\omega(\fg _{Z})]=\delta _{Z}+d[g_n]$.\\

Further, we call $\fg_Z$ a \textit{basic} Green form if $g_n$ is a basic Green form for the algebraic cycle $Z$ \cite[4.6]{Burgos:Gftp}.
\end{df}
Every Green form $\fg_{Z}$ for a cycle $Z\in
Z^{p}(X,n)_0$ gives rise to a Green current for $Z$, denoted by
$[\fg_{Z}]$, see Proposition \ref{GreenCurrent} . Moreover, each class
of Green currents $\widetilde g_{Z}$ 
 contains a representative of the type $[\fg_{Z}]$ for a basic Green form
$\fg_{Z}$, see Proposition \ref{prop:8}. 

If $Z\in Z^{p}(X,n)_{0}$ and $W\in  Z^{q}(X,m)_{0}$ are two cycles
that intersect properly (Definition  \ref{def:19}), $g_{Z}$
is a Green current for $Z$ and $\fg_{W}$ a basic Green form with
logarithmic singularities for $W$, one can define in a formal way the
current $\caP(Z)\cdot [\fg_{W}]$, see Definition \ref{def:16}, and the $\ast$-product between the
Green current and the Green form is defined as
\begin{displaymath}
  g_{Z}\ast \fg_{W}=\caP(Z)\cdot [\fg_{W}]+g_{Z}\cdot \omega (\fg_{W}). 
\end{displaymath}
This product is well defined on the class of Green currents, see Definition \ref{def:6.21}. Finally, using a moving lemma suited for our purpose (Lemma \ref{lemm:10}), the $\ast$-product above, and the product of higher Chow groups, in Theorem \ref{thm:int} we define a product for higher arithmetic Chow groups which is graded commutative and associative.\\
A fallout of this theory of higher arithmetic Chow groups, is a definition of height pairing between higher cycles with trivial real regulators. This has been developed in \S \ref{HigherChow:Product} (Definition \ref{def:17}). It is a very interesting direction, and will be the subject of a future project.

\subsection{Examples of intersection pairing in dimension zero.}

In order to produce the first examples of non-zero higher arithmetic
products, we study the case of dimension zero. That is, we consider
$X=\Spec (F)$ as a smooth projective variety over $F$.

Higher Chow groups of number fields have been extensively
studied. They are the subject of  Goncharov's programme
\cite{Goncharov:prAmc}.  For instance in \cite{Petras:dilog} there
are many concrete examples of computations of higher Chow groups for
particular fields.

The case of dimension zero is special because, in dimension zero, there
is no difference between forms and currents. Therefore for cycle $Z\in
Z^{p}(F,n)_{0}$, the current $\caP(Z)$ is already a form and $0$ is a
Green current for the cycle $Z$. This gives us many examples of higher
arithmetic cycles. One has to be careful that, if $Z_{1}$ and $Z_{2}$
are rationally  equivalent, that is they define the same class in
$\CH^{p}(F,n)_{0}$, then $(Z_{1},0)$ and $(Z_{2},0)$ do not need to be
rationally equivalent in the arithmetic setting.

In consequence we can define a pairing
\begin{displaymath}
  (\ ,\ )_{2p-1,2q-1}\colon Z^p(F, 2p-1)_{0}\times Z^q(F, 2q-1)_{
    0}\rightarrow \frac{H^1_{\fD}(F, \R(p+q))}{\im(\rho_{\text{Be}})}\
  p,q\geq1,
\end{displaymath}
given, for two cycles $\alpha\in Z^p(F, 2p-1)_{0}$ and $\beta \in
Z^q(F, 2q-1)_{0}$, by
\begin{displaymath}
  (\alpha,\beta)_{2p-1,2q-1}\coloneqq \pi _{\ast}([(\alpha,0)]\cdot[(\beta,0)]).
\end{displaymath}
We note that this pairing is defined at the level of cycles but does
not descend to the level of Chow groups. 

A crucial observation needed to compute this pairing is the fact that
$[(\alpha,0)]\cdot[(\beta,0)]=[(\alpha\cdot\beta,0)]$ and that the
cycle $\alpha\cdot\beta$ is torsion. For
instance, if $p=q=1$ and $p=1,q=2$, such intersection pairing is given
by the Bloch-Wigner polylogarithm functions. This is the subject of
our study in the last section.

\subsection{Layout of the paper}
\label{sec:layout-paper}

Now we briefly describe the layout of this paper. Sections two to five
are largely preliminary in nature. Here we fix notations
and state results that we need for the rest of the paper. In section 2
we recall the zig-zag diagrams of complexes. Such diagrams are useful
in two different ways. First, to define Deligne
cohomology and second, to define higher arithmetic Chow groups. We
also explain the Thom
Whitney simple of such diagrams, that will be used to give a graded
commutative and associative differential graded algebra that computes
Deligne cohomology. In this section we also recall the theory of
cubical abelian groups and complexes.    

In Section 3 we recall the cubical version of Bloch's Higher
Chow groups as well as the moving lemmas that we will need in the
paper.

Section 4 is devoted to recalling Deligne-Beilinson cohomology and the different
complexes that we can use to compute it. Of particular importance for
us is the Thom-Whitney complex because the product is graded
commutative and associative at the level of complexes and not just at
the level of cohomology. We include a key lemma (Lemma \ref{lemm:13}),
which is used in later sections to prove different properties of Green
currents.

In Section 5 we recall the cycle class map from higher cycles to
Deligne cohomology through the cubical version of Goncharov regulator.

In Section 6 we develop the theory of higher Green current and form in
detail. We show the functorial properties of Green currents, and
develop a product of Green currents for two cycles intersecting
properly. 

Section 7 is devoted towards developing the theory of
higher arithmetic Chow groups, and an intersection theory which
generalizes the one developed by Gillet and Soul\'e for arithmetic
Chow groups. It also includes a definition of a higher height
pairing which can be seen as a generalization of Beilinson's height
pairing, for higher Chow groups.

Finally, in section 8, we give examples of
intersection product in case of dimension zero. This section gives the
reader a recipe to compute the higher arithmetic intersection pairing
for a particular choice of Green current, and it is given using 
Goncharov's regulator at the level of complexes.\\

\textbf{Acknowledgements.} This paper originated during the second
author's stay at ICMAT, Madrid. He is extremely grateful for the
wonderful hospitality provided by ICMAT. During further elaboration
and writing of the paper, the second author was at Texas
A\&M University as a postdoctoral fellow. He wants to thank Greg
Pearlstein for being a wonderful host, a great mentor, and for his suggestions on improving the paper. We would also like to
thank James Lewis for his constant encouragement and willingness
to participate in this project, Herbert Gangl for his help in the
concrete computations, Vincent Maillot for the initial discussions about this
project, and Matt Kerr for the numerous email discussions they had
during the shaping of this project.

\section{Preliminaries on homological algebra}
\label{sec:prelim}
\subsection{Complexes and simples}

In this paper we use the standard conventions on (co)-chain
complexes and on differential graded
algebras ($\dg$-algebras for short). If $\caA$ is an abelian category,
a \emph{cochain complex} in $\caA$ is a pair $A=(A^{\ast},d_{A})$,
where $A^{\ast}=\bigoplus_{n\in \Z}A^{n}$ is a graded object and
$d_{A}\colon A^{\ast}\to A^{\ast}$ is a homogeneous map of degree 1
such that $d_{A}^{2}=0$. The cochain complex $(A^{\ast},d_{A})$ will
usually be denoted by
the letter $A$ without decoration unless we want to stress the fact that it is a complex
or we want to emphasize the degree. 

Given a cochain complex 
$A=(A^{\ast},d_{A})$  in an
abelian category, the \emph{shift} by an integer $m$, 
denoted by $A[m]=(A[m]^{\ast},d_{A}[m])$, is the shifted graded object
$A[m]^n=A^{m+n}$ with differential $d_{A}[m]=(-1)^{m}d_A$.

The \emph{simple} of a cochain map $f\colon A\rightarrow B$ is the
cochain complex $s(f)=(s(f)^{\ast}, d_s)$ with
\begin{equation}
  \label{eq:71}
 s(f)^n=A^n\oplus B^{n-1},\qquad d_s(a,b)=(d_Aa, f(a)- d_Bb). 
\end{equation}
Note that this is same as
the cone of $-f$ shifted by $-1$. There are maps
\begin{displaymath}
  \begin{matrix}
    \bmap\colon & B&\to &s(f)[1]\\
    & b&\mapsto&(0,-b)
  \end{matrix},\qquad
  \begin{matrix}
    \amap\colon & s(f)&\to & A\\
    &(a,b)&\mapsto & a
  \end{matrix}
\end{displaymath}
such that
\begin{displaymath}
  s(f)\xrightarrow{\amap}A\xrightarrow{f}B
  \xrightarrow{\bmap} s(f)[1]
\end{displaymath}
is a distinguished triangle in the derived category. Therefore there
is an associated long exact sequence
\begin{displaymath}
  \cdots \rightarrow H^n(s(f))\xrightarrow{\amap}
  H^n(A)\xrightarrow{f}H^n(B)\xrightarrow{\bmap}
  H^{n+1}(s(f))\rightarrow \cdots
\end{displaymath}
If $f$ is injective, there is a quasi-isomorphism
\begin{displaymath}
  s(f)[1]\xrightarrow {\pi }\coker(f)=B/A,\hspace{0.2cm}(a,b)\mapsto [-b],
\end{displaymath}
while if $f$ is surjective, there is a quasi-isomorphism
\begin{displaymath}
  \ker(f)\xrightarrow {\iota }s(f),\hspace{0.2cm}a\mapsto (a,0).
\end{displaymath}

Let $n\in \Z$. The \emph{canonical truncation} of $A$ at degree $n$ is
the cochain complex $\tau
_{\leq n}A=(\tau
_{\leq n}A^{\ast},d_{\tau
_{\leq n}A})$ given by 
\begin{equation}\label{eq:30}
(\tau _{\leq n}A)^r =
\begin{cases}
  A^r & r<n, \\
  \ker(d_A\colon A^n\rightarrow A^{n+1}) & r=n,\\
  0 & r>n,
\end{cases}
\end{equation}
and differential induced by $d_A$. It follows from the definition that
\begin{displaymath}
H^r(\tau _{\leq n}A) =
\begin{cases}
H^r(A), & r\leq n, \\
0, & r>n.  
\end{cases}
\end{displaymath}

\begin{df}\label{Itcochain}(\cite{BurgosFeliu:hacg}, definition 1.2.2)
A $k$-iterated cochain complex $A^*=(A^*, d_1,\cdots , d_k)$ of abelian groups is a $k$-graded object together with $k$ endomorphisms $d_1,\cdots, d_k$ of multidegrees $l_1,\cdots ,l_k$, such that, for all $i,j$, $d_id_i=0$ and $d_id_j=d_jd_i$. The endomorphism $d_i$ is called the $i$-th differential of $A^*$. 
\end{df}
A cochain morphism between two $k$-iterated cochain complexes is a
morphism of $k$-graded objects that respects the degrees and that
commutes with all the differentials.\\ 

When dealing with $k$-iterated complexes we will use the multi-index
notation
\begin{displaymath}
  \bfj=(j_{1},\dots,j_{k}),\ |\bfj|=\sum_{i=1}^{k}j_{i},\ ,\
  |\bfj|_{l}=\sum_{i=1}^{l}j_{i},  
\end{displaymath}
in particular $|\bfj|_{0}=0$.

Generalizing the case of a double complex, from a $k$-iterated complex
we can construct a simple complex.

\begin{df}\label{SIt}(definition 1.2.4 of \cite{BurgosFeliu:hacg})
Let $(A^*, d_1,\cdots ,d_k)$ be a $k$-iterated cochain complex. The
\textit{simple complex} of $A^*$ is a cochain complex $s(A)^*$ whose
graded groups are 
$$s(A)^n=\bigoplus _{|\bfj|=n}A^\bfj,$$
and whose differential $s(A)^n\xrightarrow {d_s}s(A)^{n+1}$ is defined
by, for every $a_\bfj\in A^\bfj$ with $|\bfj|=n$ 
$$d_s(a_{\bfj})=\sum ^k_{l=1}(-1)^{|\bfj|_{l-1}}d_l(a_\bfj)\in s(A)^{n+1}.$$

\end{df}
If $\mathcal{I}^k$ denote the category of all $k$-iterated cochain
complexes, then the simple associated as above describes a functor
$s(-)\colon \mathcal{I}^k\rightarrow \mathcal{I}^1$. Likewise, one can
also define the category $\mathcal{I}_k$ of $k$-iterated chain
complexes using a similar methodology. 

\begin{ex}
  Let $f\colon (A^{\ast},d)\to (B^{\ast},d)$ be a morphism of
  complexes, then $f$ determines a
  $2$-iterated complex, also denoted by $f$, that is given by
  \begin{displaymath}
    f^{0,n}=A^{n},\ f^{1,n}=B^{n},\  d_{1}=f,\  d_{2}=d.
  \end{displaymath}
  with this convention the simple of $f$ as a cochain morphism
  of complexes described in \eqref{eq:71}, agrees with its simple as a
  2-iterated complex. 
\end{ex}

\subsection{Cubical abelian groups}\label{cubab}
Cubical objects on a category are modeled on the cube as simplicial
objects are modeled on a simplex (see \cite{GNPP} for details). We
describe the holomogical 
version (with chain complexes) because the higher Chow groups 
form a homology theory.  

Let $C_{\cdot}=\{C_n\}_{n\geq 0}$ be a cubical abelian group with face maps
$\delta_i^j: C_n\rightarrow C_{n-1}$, for $i=1,\dots,n$ and $j=0,1$,
and degeneracy maps $\sigma_i:
C_n\rightarrow C_{n+1}$, for $i=1,\dots,n+1$. Let $D_n\subset C_n$ be
the subgroup of \emph{degenerate elements} of $C_n$, and let
$\widetilde C_{n}=C_{n}/D_{n}$.

Let $C_{\ast}$ denote the \emph{associated chain complex}, that is, the
chain complex whose $n$-th graded
piece is $C_n$ and whose differential is given by
$$\delta= \sum_{i=1}^n \sum_{j=0,1}(-1)^{i+j}\delta_i^j.$$
Thus
$D_{\ast}$ is a subcomplex and $\widetilde C_{\ast}$ is a quotient
complex.
We fix the \emph{normalized chain complex} associated to $C_{\cdot}$,
$NC_{\ast}$, to be the chain
complex whose $n$-th graded group is
\begin{displaymath}
 NC_n \coloneqq \bigcap_{i=1}^{n} \ker \delta_i^1, 
\end{displaymath}
and whose
differential is $ \delta=\sum_{i=1}^n (-1)^{i}\delta_i^{0}.$
It is well-known that there is a decomposition of chain complexes
\begin{equation}
  \label{eq:25}
  C_{\ast} \cong NC_{\ast} \oplus D_{\ast}
\end{equation}
giving an isomorphism
$NC_{\ast}\cong \widetilde C_{\ast}$. In general the complex
$D_{\ast}$ is not acyclic and the homology of the complex $C_{\ast}$
does not agree with the homology of $NC_{\ast}$. But it turns out
that the homology of $NC_{\ast}$ is the ``right one''. See for
instance \cite[Chapter VII]{Massey:top}.

For certain cubical abelian groups, the normalized chain complex
can be further simplified, up to homotopy equivalence, by considering the
elements which belong to the kernel of all faces but $\delta_1^0$.

\begin{df}\label{normalized2}
Let $C_{\cdot}$ be a cubical abelian  group. The \emph{refined
  normalized complex}, denoted $N_0C_{\ast}$, is the
complex defined by
\begin{equation}
N_0C_n= \bigcap_{i=1}^{n} \ker \delta_i^1 \cap \bigcap_{i=2}^{n}
\ker \delta_i^0,\quad \textrm{and differential }\
\delta=-\delta_1^0.\end{equation}
\end{df}

The proof of the next proposition is analogous to the proof of
Theorem 4.4.2 in \cite{Bloch:nephcg}. The result is proved there only
for the cubical abelian group defining  the higher Chow complex
(see $\S$\ref{cubBl} below). We give here the abstract version of
the statement whose proof can be found in \cite[Prop.1.20]{BurgosFeliu:hacg}. 

\begin{prop}\label{normalized3}
Let $C_{\cdot}$ be a cubical abelian group. Assume that it comes
equipped with a collection of maps
$$h_j: C_n \rightarrow C_{n+1},\qquad j=1,\dots,n, $$
such that, for any $l=0,1$, the following identities are satisfied:
\begin{eqnarray}
\delta_j^1h_j &=& \delta_{j+1}^1h_j =s_j\delta_j^1, \notag \\
\delta_j^0
h_j &=& \delta_{j+1}^0 h_j=\Id, \label{eq:3}\\
\notag \delta_i^lh_j &=&
\left\{\begin{array}{ll} h_{j-1}\delta_i^l & i<j, \\
h_j \delta_{i-1}^l & i>j+1.
\end{array} \right.
\end{eqnarray}
Then, the inclusion of complexes
$$i\colon N_0C_{\ast} \hookrightarrow NC_{\ast} $$
is a homotopy equivalence. More precisely, there is a map $\pi \colon
NC_{\ast} \hookrightarrow N_{0}C_{\ast}$ such that $\pi\circ i=\Id$
and $i\circ \pi $ is homotopically equivalent to the identity. In
particular $N_0C_{\ast}$ is a direct summand of $N C_{\ast}$. 
\end{prop}

\begin{rmk} To every cubical abelian group $C_{\cdot}$
there are associated four chain complexes: $C_*,NC_*,N_0C_*$ and
$\widetilde{C}_*$. In some situations it will be necessary to
consider the cochain complexes associated to these chain
complexes. In this case we will denote them, respectively, by
$C^*,NC^*,N_0C^*$ and $\widetilde{C}^*.$
\end{rmk}

\subsection{Cubical cochain complexes}
\label{sec:cubic-coch-compl}

Let $X^*_{\cdot}$ be a cubical cochain complex, where $\ast$ denotes
the cochain degree and $\cdot$ the cubical degree. Then, for every
$m$, we have defined 
the cochain complexes $NX_m^*, N_0X_m^*$ and $\widetilde{X}_m^*$.
\begin{prop}\label{cubchain}
Let $X^*_{\cdot},Y^*_{\cdot}$ be two cubical cochain complexes and
let $f:X^*_{\cdot}\rightarrow Y^*_{\cdot}$ be a morphism. Assume
that for every $m$, the cochain morphism
$$X^*_m\xrightarrow{f_m} Y^{*}_m $$ is a quasi-isomorphism.
Then, the induced morphisms
$$ NX^{*}_m \xrightarrow{f_m}  NY^{*}_m \quad \textrm{and}\quad
\widetilde{X}^{*}_m
\xrightarrow{f_m}  \widetilde{Y}^{*}_m $$
are quasi-isomorphisms.
\end{prop}
\begin{proof}
  See \cite[Prop. 1.24]{BurgosFeliu:hacg}.
\end{proof}

A similar result holds true for the refined normalized complex.

\begin{prop}\label{prop:7}  Let $X^*_{\cdot},Y^*_{\cdot}$ be two
  cubical cochain complexes and $f:X^*_{\cdot}\rightarrow Y^*_{\cdot}$
  a morphism of complexes such that for every $m$, the cochain
  morphism
  \begin{displaymath}
   X^*_m\xrightarrow{f_m} Y^{*}_m  
 \end{displaymath}
 is a quasi-isomorphism. Assume furthermore that, for each $n$, there are morphisms
 of complexes
 \begin{displaymath}
   h_{j}\colon X^{\ast}_{n}\to X^{\ast}_{n+1},\
   h'_{j}\colon Y^{\ast}_{n}\to Y^{\ast}_{n+1},\
   j=1,\dots,n
 \end{displaymath}
 satisfying the identities \eqref{eq:3} and the commutativity
 $f_{n}\circ h_{j}=h'_{j}\circ f_{n}$. Then the induced morphisms
 \begin{displaymath}
   N_{0}X^{\ast}_{n}\longrightarrow N_{0}Y^{\ast}_{n}
 \end{displaymath}
are quasi-isomorphisms.
\end{prop}
\begin{proof}
  Since the maps $h_{j}$ and $h_{j}'$ are morphisms of complexes,
  then $N_{0}X^{\ast}_{n}$ is a direct summand as a complex of
  $NX^{\ast}_{n}$. Similarly for $N_{0}Y^{\ast}_{n}$. By the
  commutation of the $h$ with $f$, the induced morphism
  $NX^{\ast}_{n}\longrightarrow NY^{\ast}_{n}$ respects these direct
  factors. Thus this proposition follows from Proposition
  \ref{cubchain}.  
\end{proof}

% \begin{proof}
% The proposition follows from the decompositions
% \begin{eqnarray*}
%  H^r(X^*_{m}) & = &  H^r(NX^*_{m}) \oplus  H^r(DX^*_{m}), \\
%  H^r(Y^*_{m}) & = & H^r(NY^*_{m}) \oplus  H^r(DY^*_{m}),
%  \end{eqnarray*}
% and the fact that $f_m$ induces cochain maps
% $$NX^*_m \xrightarrow{f_m} NY^*_m,\qquad DX_m^*\xrightarrow{f_m} DY^*_m.$$
% \end{proof}

Let $X^*_{\cdot}$ be a cubical cochain complex, then normalization and
cohomology commute with each other. More precisely, if we normalize
with respect to the cubical direction,  for each $m$, we obtain a
cochain complex
$NX^{\ast}_{m}$ with $r$-th cohomology $H^{r}(NX^{\ast}_{m})$. On the
other hand, if we first take $r$-th cohomology in the cochain
direction we obtain a cubical object $H^{r}(X^{\ast}_{\cdot})$, whose
normalization is given, in degree $m$ by $NH^{r}(X^{\ast}_{\cdot})_{m}$.   

\begin{prop}\label{cubchain2}
With the above setting, the natural
morphism
$$ H^r(NX^*_{m}) \xrightarrow{f} NH^r(X^*_{\cdot})_{m}$$
is an isomorphism for all $m\geq 0$.
\end{prop}
\begin{proof}
  See  \cite[Prop. 1.25]{BurgosFeliu:hacg}
\end{proof}

% \begin{proof} The  cohomology groups $ H^r(X^*_{\cdot})$ have a cubical abelian group structure.
% Hence there is a decomposition
% $$ H^r(X^*_{\cdot})= NH^r(X^*_{\cdot}) \oplus  DH^r(X^*_{\cdot}).$$
% In addition, there is a decomposition $X^*_n=NX^*_n\oplus DX^*_n.$
% Therefore
% $$ H^r(X^*_{\cdot})= H^r(NX^*_{\cdot}) \oplus  H^r(DX^*_{\cdot}).$$
% The lemma follows from the fact that the identity morphism in
% $H^r(X^*_{\cdot})$ maps $NH^r(X^*_{\cdot})$ to $H^r(NX^*_{\cdot})$
% and $DH^r(X^*_{\cdot})$ to $ H^r(DX^*_{\cdot}).$
% \end{proof}

Raising the cubical degree and taking normalization we obtain a
2-iterated cochain complex
\begin{displaymath}
  NX^{r,n}=NX^{r}_{-n}.
\end{displaymath}
and an associated simple cochain complex $sNX^{\ast}$. In this situation
we will denote by by $d$ the
cochain (or chain differential) and by $\delta $ the cubical
differential. We will always put the cochain  
differential as the first one, therefore the differential in the
simple complex, denoted $d_{s}$ is given, for $\omega \in NX^{r,n}$ by
\begin{equation}
  \label{eq:23}
  d_{s}\omega =d \omega +(-1)^{r}\delta \omega.
\end{equation}

\section{ Higher Chow groups}
\label{sec:higher-chow-groups}

We recall here the definition and main properties of the
\emph{higher Chow groups} defined by Bloch in \cite{Bloch:achK}.
Initially, they were defined using the chain complex associated to
a simplicial abelian group. However, since we are interested in the
product structure,
it is more convenient to use the
cubical presentation, as given by Levine in \cite{Levine:BhCgr} .

\subsection{The cubical Bloch complex for higher Chow groups
}\label{cubBl}

Fix a base field $k$ and let $\P^1$ be the projective line over
$k$. Let $ \square = \P^1\setminus \{1\}(\cong \A^1).$
The cartesian
product $(\P^1)^{\cdot}$ has a cocubical scheme structure.
For $i=1,\dots,n$, we denote by $t_{i}\in (k\cup\{\infty\})\setminus
\{1\}$  the absolute coordinate of the $i$-th factor. Then the
coface and codegeneracy maps
are defined as
\begin{eqnarray*}
  \delta_0^i(t_1,\dots,t_n) &=& (t_1,\dots,t_{i-1},0,t_{i},\dots,t_n), \\
\delta_1^i(t_1,\dots,t_n) &=& (t_1,\dots,t_{i-1},\infty,t_{i},\dots,t_n), \\
\sigma^i(t_1,\dots,t_n) &=& (t_1,\dots,t_{i-1},t_{i+1},\dots,t_n).
\end{eqnarray*}
Then, $\square^{\cdot}$ inherits a cocubical scheme structure from that of $(\P^1)^{\cdot}$.
An $r$-dimensional \emph{face} of $\square^n$ is any subscheme of the form
$\delta^{i_1}_{j_1}\cdots \delta^{i_{n-r}}_{j_{n-r}}(\square^{r})$. According to the convention used in \cite{AmalenduLevine:ahcg}, we regard $\square^n$ as a face.

In the definition of higher Chow groups, it is customary to represent
$\A^{1}$ as $\P^{1}\setminus \{1\}$, so 
that the face maps are represented by the inclusion at zero and the
inclusion at infinity as in \cite{Levine:BhCgr}. In this way the
cubical structure of 
$\square^{\cdot}$ is compatible with the cubical structure of
$(\P^{1})^{\cdot}$ in \cite{BurgosWang:hBC}. In the literature, the
usual representation $\A^{1}=\P^{1}\setminus \{\infty\}$ is
used sometimes. One
can translate from one definition to the other by using the involution
\begin{equation}
  \label{eq:5}
  x\longmapsto \frac{x}{x-1}.
\end{equation}
This involution has the fixed points $\{0,2\}$ and interchanges the
points $1$ and $\infty$.

Let $X$ be an equidimensional quasi-projective scheme of
dimension $d$ over the field $k$. Let $Z^{p}(X,n)$ be the free
abelian group generated by the codimension $p$ closed irreducible
subvarieties of $X\times \square^{n}$, which intersect properly
all the faces of $\square^n$.
The pull-back by the coface and codegeneracy maps of $\square^{\cdot}$
endow $Z^{p}(X,\cdot)$ with a cubical abelian group structure, given by
\begin{displaymath}
  \delta ^{j}_{i}=(\delta ^{i}_{j})^{\ast},\ \sigma_{i}=(\sigma ^{i})^{\ast}.
\end{displaymath}
Note that the indexes are raised or lowered to reflect the change from
cocubical to cubical structures.

Let
$(Z^{p}(X,*),\delta)$ be the associated chain complex (see $\S$\ref{cubab}) 
 and consider the \emph{normalized chain complex} associated to $Z^p(X,*)$,
$$Z^{p}(X,n)_{0}\coloneqq NZ^p(X,n)=\bigcap_{i=1}^{n} \ker \delta^1_i.$$
An element $Z\in Z^p(X,n)_0$ will be called a \textit{pre-cycle}, and
a (higher) \textit{cycle} if it also satisfies $\delta(Z)=0$. We have
also the identification
\begin{displaymath}
  Z^p(X,n)_0=\widetilde Z^{p}(X,n)\coloneqq Z^{p}(X,n)/DZ^{p}(X,n),
\end{displaymath}
where $DZ^{p}(X,n)$ is the group of degenerate pre-cycles.

\begin{df} Let $X$ be a quasi-projective equidimensional scheme  over a field
$k$. The \emph{higher Chow groups} defined by Bloch are
$$\CH^p(X,n)\coloneqq H_n(Z^{p}(X,*)_{0}).$$
\end{df}

Let $N_0$ be the refined normalized
complex of Definition \eqref{normalized2}, and let $Z^p(X,*)_{00}$ be
the complex with
\begin{equation}
  \label{eq:38}
  Z^p(X,n)_{00} \coloneqq N_0Z^p(X,n)=\bigcap_{i=1}^{n} \ker \delta^1_i\cap
\bigcap_{i=2}^{n} \ker \delta^0_i.
\end{equation}

Fix $n\geq 0$. For every $j=1,\dots,n$, we define a map
\begin{eqnarray}\label{mapsh}
 \square^{n+1} & \xrightarrow{h^j} & \square^n  \\
(t_1,\dots,t_{n+1}) &\mapsto & (t_1,\dots,t_{j-1},1-(t_j-1)(t_{j+1}-1),t_{j+2},\dots,t_{n+1}). \nonumber
\end{eqnarray}
The maps $h^{j}$ are smooth, hence
flat, so they induce pull-back maps
\begin{equation}
  \label{eq:4}
  h_{j}: Z^p(X,n) \longrightarrow Z^p(X,n+1),\  j=1,\dots,n+1,
\end{equation}
that satisfy the conditions of Proposition
\ref{normalized3}. Therefore we obtain
\begin{lem}\label{lemm:12}
The inclusion
\begin{displaymath}
  Z^p(X,n)_{00} \coloneqq  N_0 Z^p(X,n)\rightarrow Z^p(X,n)_0 
\end{displaymath}
is a homotopy equivalence.   
\end{lem}
For more details see \cite[\S
4.4]{Bloch:nephcg}. Note that in \emph{loc. cit.} the refined
normalized complex is 
given by considering the elements in the kernel of all faces but
$\delta^1_1$, instead of
$\delta^0_1$ like here and $\square$ is defined as $\A^{1}$. After
applying  the involution 
\eqref{eq:5}, the map $h^{j}$ agrees with the map denoted $h^{n-j}$ in 
\emph{loc. cit.}

\subsection{Moving lemmas}
\label{movingchow}
We gather some moving lemmas that we will need for further discussions.
The master moving lemma is Theorem \ref{thm:9} below. 

\begin{df}
Let $X$ be an equi-dimensional $k$-scheme of finite type and denote by
$X_{\sm}$ the subscheme of smooth points. Let $\caC$ be a
finite set of locally closed irreducible subsets of $X_{\sm}$ and let
$e \colon \caC \to \Z_{\ge 0}$ be a function.
We define $Z^{p}(X, n)_{\caC,e} \subset  Z^{p} (X, n)$ to be the
subgroup  generated by integral closed subschemes $Z\subset X \times
\square^{n}$ such that, for 
each $C\in\caC$ and each face $F$ of $\square^{n}$ (including the
case $F=\square^{n}$),
\begin{displaymath}
  \codim_{C\times F} \left(Z \cap (C \times F ) \right)\ge  p - e(C).
\end{displaymath}
\end{df}

\begin{lem}\label{lemm:8} For each $p\ge 0$, the groups
  $Z_{\caC,e}^{p}(X,\ast)$ form a sub-cubical group of
  $Z^{p}(X,\ast)$. Moreover, this cubical subgroup is stable under the
  pull-back morphisms $h_{j}$   of \eqref{eq:4}.
\end{lem}
\begin{proof}
Let $Z\in Z^p_{\caC, e}(X,n)$. Then for a face $F$ of $\square^{n-1}$,
each $F_i\coloneqq \delta^i_{\epsilon}(F)$ is a face of $\square^n$, where
$\epsilon=0,1$. Now the first part follows from the fact
\begin{displaymath}
  \codim_{C\times F}(\delta^{\epsilon }_{i}(Z)\cap (C\times F))=
\codim_{C\times F_i}\left(Z\cap (C\times F_i\right)),  
\end{displaymath}
for any $C\in \caC$. For the second part, we use a similar argument,
now noticing the fact that for each face $F$ of $\square^{n+1}$,
$h_j(F)$ gives a face in $\square^n$ (including possibly the total
space $\square^n$ if either $t_j$ or $t_{j+1}$ equals zero) and that
the induced map $F\to h_{j}(F)$ is flat.
\end{proof}

As a consequence of Lemma \ref{lemm:8} we can define the complexes
$Z^{p}_{\caC,e}(X,\ast)_{0}$ and $Z^{p}_{\caC,e}(X,\ast)_{00}$.

\begin{thm}\label{thm:9}  Let $X$ be a quasi-projective scheme over $k$,
  $\caC$ a finite subset of locally closed irreducible
  subsets of $X_{\sm}$ and $e\colon \caC\to \Z_{\ge 0}$ a function. Then the
  injective morphism of  complexes
\begin{displaymath}
  Z^{p}_{\caC,e}(X,\ast)_{0}\to Z^{p}(X,\ast)_{0}
\end{displaymath}
is a quasi-isomorphism.
\end{thm}
\begin{proof}
  This is \cite[Theorem 1.10]{AmalenduLevine:ahcg}.
\end{proof}

\begin{cor}\label{cor:1} With the hypothesis of Theorem \ref{thm:9},
  the morphism of complexes
  \begin{displaymath}
     Z^{p}_{\caC,e}(X,\ast)_{00}\to Z^{p}(X,\ast)_{00}
  \end{displaymath}
is a quasi-isomorphism.
\end{cor}
\begin{proof}
  Since the morphisms $h_{j}$ respect the subcomplex
  $Z^{p}_{\caC,e}(X,\ast)$, the argument that shows that
  $Z^{p}(X,\ast)_{00}$ and  $Z^{p}(X,\ast)_{0}$ are quasi-isomorphic, also
  shows that  $Z^{p}_{\caC,e}(X,\ast)_{00}$ and
  $Z^{p}_{\caC,e}(X,\ast)_{0}$ are quasi-isomorphic. Therefore the
  corollary is a consequence of Theorem \ref{thm:9}.
\end{proof}

We next want to specialize Theorem \ref{thm:9} to the particular cases
we will use in the sequel.

\begin{df}\label{def:18}
  Let $f:X\rightarrow Y$ be an arbitrary map between two smooth
  quasi-projective varieties $X,Y$ over $k$. Let $Z_{f}^{p}(Y,n)\subset
  Z^{p}(Y,n)$ be the 
  subgroup generated by the codimension $p$ irreducible
  subvarieties $Z\subset Y\times \square^{n}$,
  such that, for every face $F$ of $\square^{n}$ 
  \begin{displaymath}
    \codim_{X\times F}\left((f\times \Id)^{-1}(Z)\cap (X\times F)\right)\ge p.
  \end{displaymath}
\end{df}

\begin{lem}\label{lemm:9} Let $f\colon X\to Y$ be a morphism of smooth
  quasi-projective varieties. Then there is a finite subset $\caC$ of locally
  closed irreducible subsets of $Y$ and a function $e\colon \caC \to
  \Z_{\ge 0}$ such that $Z_{f}^{p}(Y,\ast)=Z_{\caC,e}^{p}(Y,\ast)$.  
\end{lem}
\begin{proof}
  This is proved in \cite[Part I, Chapter II, Lemma 3.5.2]{Levine:mm}.
\end{proof}

\begin{cor} \label{cor:3}
  The inclusions $Z^{p}_{f}(Y,\ast)_{0}\to Z^{p}(Y,\ast)_{0}$ and
  $Z^{p}_{f}(Y,\ast)_{00}\to Z^{p}(Y,\ast)_{00}$ are
  quasi-isomorphisms. 
\end{cor}

\begin{df}\label{def:19}
  Let $X$ be a smooth quasi-projective scheme over $k$, and let
  $p,q,n,m\ge 0$ be positive
integers. If  $Z\in Z^{p}(X,n)$, 
$W\in Z^{q}(X,m)$ are pre-cycles. We say that $Z$ and $W$
\emph{intersect properly} if, for any
face $F$ of $\square ^{n+m}$,
\begin{displaymath}
  \codim_{X\times F}\left( p_{12}^{-1}|Z| \, \cap \, p_{13}^{-1}|W|\,
    \cap \, (X\times
  F)\right)\ge p+q,
\end{displaymath}
where
\begin{displaymath}
  p_{12}\colon X\times \square ^{n}\times \square^{m}\to X\times
  \square ^{n},\quad 
  p_{13}\colon X\times \square ^{n}\times \square^{m}\to X\times
  \square ^{m}
\end{displaymath}
are the projections.

\begin{rmk}\label{casep0}
Suppose $q=0$. For $m\ge 0$ all the elemets of  $Z^{0}(X,m)$ are
degenerate. Therefore $Z^{0}(X,m)_{0}=0$  and the case $m=0$ is the
only non-trivial situation. If $W\in Z^{0}(X,0)_{0}$, then it is a
component of $X$ and if $Z$ intersects properly all the faces of
$X\times \square^{n}$, then $Z$ necessarily
intersects $W$ properly.
\end{rmk}

Let $W\in Z^{q}(X,m)$ be a pre-cycle. We denote by
  $Z_{W}^{p}(X,n)\subset Z^{p}(X,n)$ be the 
  subgroup generated by the codimension $p$ irreducible
  subvarieties $Z\subset X\times \square^{n}$, such that $Z$ and $W$
  intersect properly.
\end{df}

\begin{lem}\label{lemm:10} Let $X$ be a smooth
  quasi-projective scheme over $k$ and $W\in Z^{q}(X,m)$ a
  pre-cycle. Then there is a finite subset $\caC$ of locally 
  closed irreducible subsets of $X$ and a function $e\colon \caC \to
  \Z_{\ge 0}$ such that $Z_{W}^{p}(X,\ast)=Z_{\caC,e}^{p}(X,\ast)$.  
\end{lem}
\begin{proof} For simplicity of the discussion we assume $W$ to be
  irreducible. Being a pre-cycle, it intersects all the faces of
  $X\times \square^m$ properly.  For every face $H$ of $\square^m$
  (including the total face $\square^m$), we define
  \begin{displaymath}
   C_{H,i}\coloneqq \left \{ x\in X\, \middle|\, \dim
     \left(\pi_m^{-1}(x)\cap W\cap (X\times H)\right)=i\right\}, 
  \end{displaymath}
where $\pi_m\colon X\times \square^m\rightarrow X$ is the
projection. Then, $C_{H,i}$ is a constructible subset of $X$. Let
$C^\ell_{H,i}$ be locally closed, irreducible subsets of $X$, such that
$C_{H,i}=\bigcup _{\ell}C^\ell_{H,i}$. Observe that, for every face $F$ of
$\square^m$, $\bigcup_{i,\ell}C^\ell_{H,i}=X$. From the condition of
proper intersection, for every face $H$ of $\square^m$,
\begin{displaymath}
  \codim_{X\times H}\left(W\cap (X\times H)\right)\geq q\implies
  \dim\left(W\cap (X\times H)\right)\leq d+h-q, 
\end{displaymath}
where $h=\dim(H)$. Let $d^\ell_{H,i}=\dim(C^\ell_{H,i})$. Then
\begin{displaymath}
  d^\ell_{H,i}+i\leq \dim\left(W\cap (X\times H)\right)\leq d+h-q.
\end{displaymath}
Let $\caC=\{C^\ell_{H,i}\}$, and $e(C^\ell_{H,i})=d+h-q-d^\ell_{H,i}-i\geq 0$.

We claim that $Z_{\caC,e}(X,n)=Z_W(X,n)$. Notice that, any face $F$ of
$\square^{n+m}$ is of the form $G\times H$, where $G$ is a face of
$\square^n$ and $H$ is a face of $\square^m$. Write $g=\dim(G)$
and $f=\dim(F)$, so $f=g+h$. For any face $H$ of $\square^m$
we have the decomposition $X\times
\square^n=(\bigcup_{i,\ell}C^\ell_{H,i})\times \square^n$. Let now $Z\in
Z^{p}(X,n)$ be a pre-cycle. 
Then
\begin{multline}\label{eq:73}
  p_{12}^{-1}|Z| \, \cap \, p_{13}^{-1}|W|\,
  \cap \, (X\times F)\\
  = p_{12}^{-1}\left(|Z|\cap (X\times G)\right)\,\cap\,
  p_{13}^{-1}\left(|W|\cap (X\times H)\right)\\
   =\bigcup_{i,\ell}p^{-1}_{12}\left(|Z|\cap (C^\ell_{H,i}\times
     G)\right)\cap p^{-1}_{13}\left(|W|\cap (X\times H)\right). 
\end{multline}
Now
\begin{align*}
\codim_{C^\ell_{H,i}\times G}\left(|Z|\cap (C^\ell_{H,i}\times G)\right)\geq p-e(C^\ell_{H,i})\\
\Leftrightarrow \dim\left(|Z|\cap (C^\ell_{H,i}\times G)\right)\leq d^\ell_{H,i}+g-p+e(C^\ell_{H,i}).
\end{align*}
Finally, $Z$ belongs to $Z^{p}_{\caC,e}(X,n)$ if and only if, for
every face $F=G\times H$ and every $i,\ell$, the condition 
\begin{equation}\label{eq:72}
  \dim(|Z|\cap C^\ell_{H,i}\times G) \le d^{\ell}_{H,i}+g - p+e(C^\ell_{H,i})
\end{equation}
holds. By the definition of $C^\ell_{H,i}$, the condition \eqref{eq:72}
is equivalent to the condition
\begin{multline*}
  \dim p^{-1}_{12}(|Z|\cap C^\ell_{H,i}\times G)\cap p^{-1}_{13}(W\cap X\times H)
\\ \leq d^{\ell}_{H,i}+g-p+e(C^\ell_{H,i})+i=d+f-p-q,
\end{multline*}
where, in the last equality we have used the definition of
$e(C^\ell_{H,i})$. Therefore, using the decomposition \eqref{eq:73} we
deduce that condition  \eqref{eq:72} is satisfied for every face $F$
and indices $\ell,i$ if and only if
\begin{displaymath}
  \codim_{X\times F}\left(p^{-1}_{12}|Z|\cap p^{-1}_{13}|W|\cap
    (X\times F)\right)\geq p+q.
\end{displaymath}
In other words, $Z\in Z_{\caC,e}^{p}(X,n)$ if and only if $Z\in
Z_W^{p}(X,n)$ proving the result. 
\end{proof}

\begin{cor} \label{cor:2}
  The inclusions $Z^{p}_{W}(X,\ast)_{0}\to Z^{p}(X,\ast)_{0}$ and
  $Z^{p}_{W}(X,\ast)_{00}\to Z^{p}(X,\ast)_{00}$ are
  quasi-isomorphisms. 
\end{cor}

\begin{rmk}\label{rem:3}
  Lemma \ref{lemm:10} and Corollary \ref{cor:2} can easily be
  generalized to a finite family of pre-cycles.
\end{rmk}

\subsection{Functoriality}
\label{functchow}
It follows easily from the definition that the complex $Z^{p}(X,*)_0$
is covariant with respect to proper maps (with a shift in the
grading) and contravariant for flat maps.

Let $f\colon X\to Y$ be a morphism between smooth quasi-projective
schemes over $k$.  Then,
$Z_{f}^{p}(Y,*)_0$ is a chain complex and, by Corollary \ref{cor:3}
the inclusion  of complexes
\[Z^{p}_f(Y,*)_{0} \subseteq Z^{p}(Y,*)_0\]
is a quasi-isomorphism. Moreover, by the definition of $Z^{p}_f(Y,*)$,
the pull-back by  $f$ is defined for
algebraic cycles in $Z_{f}^{p}(Y,*)_0$ and hence there is a
well-defined pull-back morphism
\begin{displaymath}
  \CH^{p}(Y,n) \xrightarrow{f^*} \CH^p(X,n).
\end{displaymath}

\subsection{Product structure}\label{chowprod}
Let $X$ be a quasi-projective scheme over $K$ and $\alpha \in
\CH^{p}(X,n)$ and $\beta \in \CH^{q}(X,m)$. We can represent $\beta $
by an element $W\in Z^{q}(X,m)_{0}$ and, thanks to Corollary
\ref{cor:2}, $Z$ by an element $Z\in 
Z^{p}_{W}(X,n)_{0}$, which in turn implies that $W\in
Z^{q}_{Z}(X,m)$. Let $p_{12}$ and $p_{13}$ be as in Definition
\ref{def:19}.  
Since $Z$ and $W$ intersect properly, the intersection product
\begin{displaymath}
  p_{12}^{\ast}Z\cdot p_{23}^{\ast}W\in Z^{p+q}(X,n+m) 
\end{displaymath}
is well defined. This defines a pairing
\begin{equation}
  \label{eq:41}
  \alpha \cdot \beta =[p_{12}^{\ast}Z\cdot p_{23}^{\ast}W].
\end{equation}

 \begin{prop} Let $X$ be a smooth quasi-projective scheme over
   $k$. The pairing \eqref{eq:41}
 defines an associative product on
 \begin{displaymath}
   \CH^\ast(X,\ast):=\bigoplus_{p,n}\CH^p(X,n).
 \end{displaymath}
 This product
 is graded commutative with respect to the degree
 given by $n$. That is, if $\alpha \in \CH^p(X,n)$ and $\beta \in
 \CH^q(X,m)$, then
 \begin{displaymath}
   \alpha \cdot \beta =(-1)^{nm}\beta \cdot \alpha .
 \end{displaymath}
 \end{prop}
 \begin{proof}
   We first show that the intersection product is well defined. If
   $W'$ is another representative of $\beta $ with $\delta T=W-W'$, we
   can represent $\alpha $ by a cycle $Z'$ that intersects properly
   $T$, $W$ and $W'$ (see Remark \ref{rem:3}). Since
   \begin{displaymath}
     \delta (T\cdot Z')=W\cdot Z'-W\cdot Z',
   \end{displaymath}
   the product is independent on the choice of the representative
   $W$. A similar argument shows the independence on the choice of the
   representative $Z$.

   The associativity follows from the fact that the intersection
   product  of cycles intersecting properly is associative already at
   the level of cycles.

   The commutativity is more involved because writing
   \begin{displaymath}
     \xymatrix{ && X\times \square ^{m}\\
       X\times \square ^{n+m}\ar@{=}[r]&X \times \square ^{m}\times \square ^{n}
       \ar[ur]^{p'_{12}}\ar[dr]_{p'_{23}}&\\
       && X\times \square ^{n},}
   \end{displaymath}
   it is evident that in general,
   \begin{displaymath}
     p_{12}^{\ast}Z\cdot p_{23}^{\ast}W\not =
     (-1)^{nm}(p'_{12})^{\ast}W\cdot (p'_{23})^{\ast}Z.
   \end{displaymath}
   Thus, it is not graded commutative at the level of cycles but only
   at the level of Chow groups. The proof of the graded commutativity
   in \cite{Levine:BhCgr} uses a explicit homotopy. By technical
   reasons, this homotopy is only defined for cycles in the refined
   normalized complex $Z^p(X,\ast)_{00}$. This is harmless thanks to
   Lemma \ref{lemm:12}.
   
   We recall briefly the construction of the homotopy $H$. For
   details, the reader is encouraged to consult \S 5.4 of
   \cite{BurgosFeliu:hacg}.
   Let
   \begin{displaymath}
     h^\ast_{n+m}\colon Z^p(X, n+m)_{0}\rightarrow Z^p(X, n+m+1)_{0}
   \end{displaymath}
denote the morphism induced by
\begin{equation}\label{eq:46}
\begin{matrix}
h_{n+m}\colon\!\! & X\times \square^{n+m+1}&\rightarrow& X\times \square^{n+m}\\
&\!\!(p, x_1,\cdots, x_{n+m+1})\!\!\!\!&\mapsto &\!\!\!\!(p, x_2,\cdots ,x_{n+m}, x_1+x_{n+m+1}-x_1x_{n+m+1}).
\end{matrix}
\end{equation}
and let $\tau$  be the automorphism of $X\times \square^{n}$, given by
\begin{displaymath}
  (p, x_1,\cdots, x_n)\mapsto (p, x_2,\cdots, x_n, x_1).
\end{displaymath}
Finally, for each $n,m\geq 0$, the morphism
\begin{displaymath}
  H_{n,m}\colon Z^{p}(X,n+m)_{0}\rightarrow Z^p(X, n+m+1)_0
\end{displaymath}
is defined by (see Proposition 5.35 of \cite{BurgosFeliu:hacg})
\begin{equation}\label{eq:45}
H_{n,m}(Z)=\left\{ \begin{array} {ll}
\sum_{i=0}^{n-1}(-1)^{(m+i)(n+m-1)}
h^\ast_{n+m}((\tau^*)^{m+i}(Z)), & n\neq 0, \\ 0 & n=0,
\end{array} \right.
\end{equation}
for $Z\in Z^p(X, n+m)_{0}$. By \cite[Lemma 5.35]{BurgosFeliu:hacg}, if
$Z\in Z^{p}(X,n)_{00}$ and $W\in Z^{q}(X,m)_{00}$ are cycles
intersecting properly, then
\begin{equation}
  \label{eq:44}
  \delta H_{n,m}(Z\cdot W) = Z\cdot W - (-1)^{nm}W\cdot Z,
\end{equation}
showing the graded commutativity.
 \end{proof}

\section{Deligne-Beilinson cohomology} 
\label{sec:absol-hodge-deligne}

In this section we will recall the definition real Deligne-Beilinson
cohomology and describe several complexes to compute
it. Since we will work mainly with smooth projetive varieties,
real Deligne cohomology agrees with real Deligne-Beilinson cohomology
and with absolute Hodge cohomology in
the range of interest \cite{Beilinson:naHc}. Note that to extend the
constructions given here to quasi-projective varieties it may be
useful to work directly with real absolute Hodge cohomology instead of
real Deligne-Beilinson cohomology..

As a reference for mixed Hodge structures we will use 
\cite{MHS}.

\subsection{Conventions on differential forms and currents}
\label{sec:conv-diff-forms}

When dealing with differential forms, currents and cohomology
classes,  one can use the topologist convention, where the emphasis is
put on having real or integral valued classes. For instance, in this
convention the first Chern class of 
a line bundle will have integral coefficients. In algebraic geometry,
the fact that rational de Rham classes are not rational in the
topological, the ubiquitous appearance of the period $2\pi
i$, and the fact the choice of a particular square root of $-1$ is non
canonical makes it useful to use a different convention. In fact, let
$X$ be a 
projective smooth variety 
of dimension $d$ defined over $\Q$ and $X(\C)$ be the associated
complex manifold. Then
there is a canonical isomorphism in top degree
\begin{displaymath}
  \H_{\Zar}^{2d}(X,\Omega _{X}^{\ast})
  =H_{\sing}^{2d}(X(\C),(2\pi i)^{d}\Q)
\end{displaymath}
given by integration of differential forms. One has to take care that
integration of differential forms depends on the choice of a global
orientation and the standard choice of global orientation n a complex
manifold depends on the choice of the square root of $1$.

The algebro-geometric convention aims to control these obvious powers
of $2\pi i$ and make the above isomorphism canonical. For instance, in
this convention the first Chern
class of a line bundle has coefficients in $(2\pi i)\Z$.

Of course
using one convention or the other is a matter of taste and one can go
easily from one to the other by a normalization factor. In this paper
we will follow the algebraic geometry convention. Therefore, it is
useful to incorporate different powers of $2\pi i$
in the standard operations regarding forms and currents as in  \cite[\S
  5.4]{BurgosKramerKuehn:cacg}. We summarize here the conventions used
  because they differ from commonly used notations. 

Let $X$ be a complex manifold. We will denote by
$E_{X}^{\ast}$ the
differential graded algebra of complex valued differential forms on
$X$, by $E_{X,\R}^{\ast}$ the subalgebra of real valued forms and by
$E_{X,c}^{\ast}$ and $E_{X,\R,c}^{\ast}$ the subalgebras of
differential forms with compact support. The complexes of currents are
defined as the topological dual of the latter ones. Namely
$E^{\prime-n}_{X}$ and $E^{\prime-n}_{X,\R}$  are the topological dual
of $E_{X,c}^{n}$ and $E_{X,\R,c}^{n}$ respectively, with differential
given by
\begin{displaymath}
  dT(\eta)=(-1)^{n+1}T(d\eta).
\end{displaymath}
Assume that $X$ is equidimensional of dimension $d$.  
We then write
\begin{displaymath}
  D_{X}^{n}=E_{X,c}^{\prime n-2d},\qquad D_{X,\R}^{n}=(2\pi
  i)^{-d}E_{X,\R,c}^{\prime n-2d}. 
\end{displaymath}
In other words
\begin{displaymath}
  D_{X}=E_{X}'[-2d](-d),
\end{displaymath}
where the symbol $(-d)$ refers to the above change of real
structures. This implies that 
  \begin{displaymath}
    D^{n}_{X,\R}=\{T\in D^{n}_{X}\mid \forall \eta\in
    E^{2d-n}_{X,\R},\ T(\eta)\in (2\pi i)^{-d}\R\}.
  \end{displaymath}

  To be consistent with these choices we need to adjust the definition
  of the current associated to a locally integrable form and a
  cycle. Given a locally integrable differential form
  $\omega $ of degree $n$, we will denote by $[\omega ]\in D^{n}_{X}$
  the current defined by
  \begin{equation}\label{eq:7}
    [\omega ](\eta)=\frac{1}{(2\pi i)^{d}}\int_{X}\omega \land \eta.
  \end{equation}
  With this convention, the morphism of complexes 
  \begin{math}
    [\cdot]\colon E^{\ast}_{X}\to D^{\ast}_{X} 
  \end{math}
  sends $E^{\ast}_{X,\R}$ to $D^{\ast}_{X,\R}$.

  If $f\colon X\to Y$ is a proper map of complex manifolds, of
  dimensions $d,d'$ and relative
  dimension $e=d-d'$, then the
  push-forward of currents $f_{\ast}\colon D^{\ast}_{X}\to
  D^{\ast-2e}_{Y}$ is defined, for $T\in D^{n}_{X}$ and $\eta\in
  E^{2d-n}_{Y,c}$ by
  \begin{displaymath}
    f_{\ast}T(\eta)=T(f^{\ast}\eta).
  \end{displaymath}
  Then $f_{\ast}$ sends $D^{\ast}_{X,\R}$ to
  $(2\pi i)^{-e} D^{\ast}_{Y,\R}$.

  Finally, assume that $X$ is algebraic and $Y\subset X$ is a
  codimension $p$ subvariety of $X$. Let $\iota \colon \widetilde Y\to X$ be
  a resolution of singularities of $Y$. Then the current integration
  along $Y$ is defined as $\delta _{Y}=\iota_{\ast}[1_{\widetilde
    Y}]$, where $1_{\widetilde Y}$ is the constant function $1$ on
  $\widetilde Y$. Therefore
  \begin{equation}
    \label{eq:8}
    \delta _{Y}(\eta)=\frac{1}{(2\pi i)^{d-p}}\int_{\widetilde Y}\iota ^{\ast} \eta.
  \end{equation}
  Then $\delta _{Y}\in (2\pi i)^{p}D^{2p}_{X,\R}$. Given any cycle
  $\zeta\in Z^{p}(X)$ we define $\delta _{\zeta}$ by linearity.

  \begin{rmk}\label{rem:15} We stress the fact that the sign of the
    integral depends on the choice of an orientation. If
    $z_{1},\dots,z_{d}$ are local complex coordinates with
    $z_{j}=x_{j}+i y_{j}$, then the standard orientation is given by
    the volume form
    \begin{displaymath}
      \vol=dx_{1}\land d y_{1}\land \dots \land dx_{d}\land d y_{d}.
    \end{displaymath}
    If we change the choice of the square root of $-1$ from $i$ to
    $-i$ then $\vol $ is sent to $(-1)^{d}\vol$, which is the same
    change of sign suffered by $(2\pi i)^{d}$. Therefore the symbols
    $[\omega ]$ and $\delta _{Y}$ as used here do not depend on a
    particular choice of $\sqrt{-1}$. 
  \end{rmk}

  \begin{ex}\label{exm:5} To see how this conventions work in
    practice, we review the classical example of the
    logarithm. Consider the case $X=\P^{1}(\C)$ with absolute
    coordinate $t$. So $\Div(t)=[0]-[\infty]$. Then
    \begin{align}
      \partial\bar \partial [\log t\bar t]&=-\delta _{\Div t}=\delta
                                            _{\infty}-\delta _{0} \notag\\
      d \left[\frac{dt}{t}\right ]&=\delta _{\Div t}=\delta
                                    _{0}-\delta _{\infty},\label{eq:77}\\
      d \left[\frac{d\bar t}{\bar t}\right ]&=-\delta _{\Div t}=\delta
                                              _{\infty}-\delta _{0}. \notag
    \end{align}
    More generally, if $X$ is a complex manifold, $L$ is a line
    bundle provided with a smooth hermitian metric $\|\cdot\|$ and $s$
    is a nonzero rational section of $L$, then the Poincar\'e-Lelong
    formula reads
    \begin{equation}\label{eq:74}
      \partial\bar \partial [\log
      \|s\|^{2}]=[c_{1}(L,\|\cdot\|)]-\delta _{\Div s}, 
    \end{equation}
    where $c_{1}(L,\|\cdot\|)\in (2\pi i)E^{2}_{X,\R}$ is the first
    Chern form of $L$.
  \end{ex}

\subsection{Dolbeault complexes}
\label{sec:dolbeault-complexes}

We recall from \cite{Burgos:CDB} the notion of Dolbeault complex.

\begin{df}\label{def:10}
A \emph{Dolbeault complex} $A=(A^{\ast}_{\mathbb{R}},d_{A})$ is 
a graded complex of real vector spaces, which is bounded from below 
and equipped with a bigrading on $A_{\mathbb{C}}=A_{\mathbb{R}}
\otimes_{\mathbb{R}}{\mathbb{C}}$, i.e.,
\begin{displaymath}
A^{n}_{\mathbb{C}}=\bigoplus_{p+q=n}A^{p,q},
\end{displaymath}  
satisfying the following properties:
\begin{enumerate}
\item[(i)]
The differential $d_{A}$ can be decomposed as the sum $d_{A}=
\partial+\bar{\partial}$ of operators $\partial$ of type $(1,0)$, 
resp. $\bar{\partial}$ of type $(0,1)$.
\item[(ii)] 
It satisfies the symmetry property $\overline{A^{p,q}}=A^{q,p}$,
where $\overline{\phantom{M}}$ denotes the complex conjugation induced
by the real structure of $A_{\C}$. In other words, if $\omega \in
A_{\R}$ and $\alpha \in \C$ then $\overline{\omega \otimes \alpha }=\omega
\otimes \overline {\alpha} $. 
\end{enumerate}
 \end{df}

\begin{notation}
\label{def:13}
Given a Dolbeault complex $A=(A^{\ast}_{\mathbb{R}},d_{A})$, we 
will use the following notations. The Hodge filtration $F$ of $A$ 
is the decreasing filtration of $A_{\mathbb{C}}$ given by
\begin{displaymath}
F^{p}A^{n}\coloneqq F^{p}A^{n}_{\mathbb{C}}\coloneqq \bigoplus_{p'\geq p}A^{p',n-p'}.
\end{displaymath}
The filtration $\overline F$ of $A$ is the complex conjugate of $F$, 
i.e.,
\begin{displaymath}
\overline{F}^{p}A^{n}=\overline{F}^{p}A^{n}_{\mathbb{C}}=\overline
{F^{p}A^{n}_{\mathbb{C}}}.
\end{displaymath}
For an element $x\in A_{\mathbb{C}}$, we write $x^{i,j}$ for its 
component in $A^{i,j}$. For $k,k'\geq 0$, we define an operator 
$F^{k,k'}:A_{\mathbb{C}}\longrightarrow A_{\mathbb{C}}$ by the 
rule 
\begin{displaymath}
F^{k,k'}(x)\coloneqq \sum_{l\geq k,l'\geq k'}x^{l,l'}.
\end{displaymath}
We note that the operator $F^{k,k'}$ is the projection of $A^{\ast}_
{\mathbb{C}}$ onto the subspace $F^{k}A^{\ast}\cap\overline{F}^{k'}
A^{\ast}$. We will write $F^{k}=F^{k,-\infty}$. 

We denote by $A^{n}_{\mathbb{R}}(p)$ the subgroup $(2\pi i)^{p}\cdot
A^{n}_{\mathbb{R}}\subseteq A^{n}_{\mathbb{C}}$, and we define the 
operator
\begin{displaymath}
\pi_{p}:A_{\mathbb{C}}\longrightarrow A_{\mathbb{R}}(p)
\end{displaymath}
by setting $\pi_{p}(x)\coloneqq \frac{1}{2}(x+(-1)^{p}\bar{x})$.
\end{notation}

\begin{df}
A \emph{Dolbeault algebra} $A=(A^{\ast}_{\mathbb{R}},d_{A},\land)$
is a Dolbeault complex equipped with an associative and graded 
commutative product 
\begin{displaymath}
\land:A^{\ast}_{\mathbb{R}}\times A^{\ast}_{\mathbb{R}}\longrightarrow
A^{\ast}_{\mathbb{R}} 
\end{displaymath}
such that the induced multiplication on $A^{\ast}_{\mathbb{C}}$ is 
compatible with the bigrading, i.e.,
\begin{displaymath}  
A^{p,q}\land A^{p',q'}\subseteq A^{p+p',q+q'}.
\end{displaymath}
If $A$ is a Dolbeault algebra, the a \emph{Dolbeault module} over $A$
is a Dolbeault complex $B=(B^{\ast}_{\mathbb{R}},d_{B})$ together
with an $A$-module structure 
\begin{displaymath}
\land:A^{\ast}_{\mathbb{R}}\times B^{\ast}_{\mathbb{R}}\longrightarrow
B^{\ast}_{\mathbb{R}} 
\end{displaymath}
compatible with the bigrading,
\end{df}

% \subsection{Mixed Hodge complexes}
% \label{sec:mixed-hodge-compl}

% \begin{df}\label{def:1}
%   \begin{enumerate}
%   \item A \emph{real Hodge complex} $K$ of weight $w\in \Z$
%   is a triple
%   \begin{displaymath}
%     \big(K_{\R}^{\ast},(K_{\C}^{\ast},F),\alpha \big),
%   \end{displaymath}
% where $K_{\R}^{\ast}$ is a bounded below complex of real vector spaces
% such that $H^{n}(K_{\R}^{\ast})$ is finite dimensional,
% $(K_{\C}^{\ast},F)$ is a (bounded) complex of complex vector spaces
% together with a decreasing \emph{Hodge} filtration $F$ and $\alpha \colon
% K_{\R}^{\ast}\to K_{\C}^{\ast}$ is a quasi isomorphism, such that, if
% we denote also by $F$ the filtration induced in cohomology, then for 
% each $k\in \Z$ the triple
% \begin{displaymath}
%   \left(H^{k}(K_{\R}^{\ast}),\big(H^{k}(K_{\C}^{\ast}),H^{k}(F)\big),F\right)
% \end{displaymath}
% is a pure Hodge structure of weight $k+w$.
% \item A \emph{real mixed Hodge complex} is a triple
%     \begin{displaymath}
%     \big((K_{\R}^{\ast},W),(K_{\C}^{\ast},W,F),\alpha \big),
%   \end{displaymath}
% where $W$ is an increasing filtration in both, $K_{\R}^{\ast}$ and
% $K_{\C}^{\ast}$, called the \emph{weight} filtration, such that, for
% each $w\in \Z$ the triple 
% \begin{displaymath}
%   \big(\gr_{w}^{W}K_{\R}^{\ast},
%   (\gr_{w}^{W}K_{\C}^{\ast},\gr_{w}^{W}F),
%   \gr_{w}^{W} \alpha \big)
% \end{displaymath}
% is a Hodge complex of weight $w$.
%   \end{enumerate}
% \end{df}

\begin{ex}\label{ex:1} Let $X$ be a smooth variety over
  $\C$ of dimension $d$.
  \begin{enumerate}
  \item \label{item:1}   Let $E^{\ast}_{X,\R}$ denote the complex of smooth real
  valued  differential forms on $X^{\an}=X(\C)$ while $E^{\ast}_{X}$ denote the one
  of smooth complex valued differental forms. It is a Dolbeault
  complex with the usual bigrading
  \begin{displaymath}
    E^{n}_{X}=\bigoplus _{p+q=n}E^{p,q}_{X}.
  \end{displaymath}
  This Dolbeault complex will be denoted $E_{X}$. It is a Dolbeault
  algebra. When $X$ is projective, 
  $E^{\ast}_{X}$ coputes the cohomology of $X$ with its real structure
  and Hodge filtration.
\item \label{item:3} Let $D^{\ast}_{X}$ be the complex of currents on
  $X(\C)$ with the real subcomplex $D^{\ast}_{X,\R}$ as in section
  \ref{sec:conv-diff-forms}.  The complex of currents has also a
  bigrading and is another example of a Dolbeault complex that we
  denote $D_{X}$. It is a Dolbeault module over $E_{X}$.
  
\item \label{item:2}   Assume that $X$ is projective and let $Y\subset
  X$ be a simple normal crossing divisor on $X$. Let
  $E^{\ast}_{X,\R}(\log Y)$ denote the complex of real valued smooth
  differential forms on $U=X(\C)\setminus Y(\C)$ with logarithmic
  singularities along $Y$ as in \cite{Burgos:CDc}.  This complex is
  also a Dolbeault complex and computes the cohomology of $U$ with its
  real structure and its Hodge filtration. In fact it also has a
  weight filtration that allow us to compute the real mixed Hodge
  structure of the cohomology of $X$.
  Some times it is useful to have a complex
  that depends only on $U$ and not in a particular
  compactification. To this end we will write
  \begin{displaymath}
    E_{U,\log}=\lim_{\substack{\longrightarrow\\(X',Y')}}E_{X'}(\log Y'),
  \end{displaymath}
  were the limit is taken for all smooth compactifications
  $U\hookrightarrow X'$ with $Y'=X'\setminus U$  simple normal
  crossing divisor.
\end{enumerate}
\end{ex}

% The following facts are well known.

% \begin{prop}\label{prop:1} Let $K$ be a real mixed Hodge complex.
%   \begin{enumerate}
%   \item The cohomology of $K$ carries a mixed $\R$-Hodge structure.  
% \item The spectral sequence associated to the filtration $F$
%   degenerates at the $E_{1}$ term.
% \item The spectral sequence associated to the filtration  $W$
%   degenerates at the $E_{2}$ term.
%   \end{enumerate}
% \end{prop}

% \begin{df}\label{def:3} Given a filtered cochain complex $(K,W)$, with $W$ an
%   increasing filtration, then the decaled filtration is the
%   filtration $\widehat W$ given by
%   \begin{displaymath}
%     \widehat W _{w}K^{n}=\{x\in W_{w-n}K^{n}\mid dx \in W_{w-n-1}K^{n+1}\}. 
%   \end{displaymath} 
% \end{df}

% The basic property of the decaled filtration is the following.

% \begin{prop}\label{prop:2}
%   Let $(K^{\ast},W)$ be a filtered cochain complex such that the spectral
%   sequence associated to the filtration $W$ degenerates at the term
%   $E_{2}$. Then the spectral sequence associated to $\widehat W$
%   degenerates at the term $E_{1}$. Moreover, if $W$ and $\widehat W$
%   also denote the filtrations induced in cohomology by $W$ and
%   $\widehat W$, then
%   \begin{displaymath}
%     \widehat W_{w}H^{n}(K^{\ast})=W_{w-n}H^{n}(K^{\ast}).
%   \end{displaymath}
% \end{prop}

% \begin{rmk}\label{rem:7}
%   Recall that, if $\big((K_{\R}^{\ast},W),(K_{\C}^{\ast},W,F),\alpha
%   \big)$ is a real mixed Hodge complex, then the weight filtration in
%   cohomology is not the filtration induced by $W$ but the filtration
%   induced by $\widehat W$. 
% \end{rmk}

\subsection{Deligne cohomology}
\label{sec:absol-hodge-cohom}

We recall now the definition of Deligne cohomology.

\begin{df}\label{def:2} Let $A$ be a Dolbeault complex. For each $p\in
  \Z$, consider the obvious inclusions
  morphims $\iota _{p}\colon F^{p}A_{\C}^{\ast}\to
  A_{\C}^{\ast}$ and $\alpha _{p}\colon A_{\R}^{\ast}(p)\to
  A_{\C}^{\ast}$.  
  
The \emph{total real Deligne complex} of  $A$ twisted
by $p$ is the
simple complex associated to the map $(\iota _{p},-\alpha _{p})$:
\begin{displaymath}
  \fD_{t}(A,p)=s\left(F^{p}A_{\C}^{\ast}\oplus A_{\R}^{\ast}(p)
  \xrightarrow{(\iota _{p},-\alpha _{p})} A_{\C}^{\ast}\right).
\end{displaymath}
More explicitly, 
\begin{displaymath}
  \fD^{n}_{t}(A,p)=(2\pi i)^{p}A_{\R}^{n}\, \oplus\,
  F^{p}A_{\C}^{n}\, \oplus \,
  A_{\C}^{n-1},
\end{displaymath}
with differential
\begin{displaymath}
  d(r,f,\omega )=(dr,df,\iota_{p} (f)-\alpha_{p} (r)-d\omega ).
\end{displaymath}
The \emph{Deligne cohomology} of $A$ twisted by $p$ is the
cohomology of the complex $\fD_{t}(A,p)$,
\begin{displaymath}
  H^{n}_{\fD}(A,p)=H^{n}(\fD_{t}(A,p)).
\end{displaymath}
\end{df}

By its construction as the simple of a diagram, Deligne cohomology
fits in a long exact sequence.
\begin{prop}\label{prop:5}
  Let $A$ be a Dolbeault complex. Then there is a long exact sequence
  \begin{multline*}
    \dots \to H^{n-1}(A^{\ast}_{\R}(p))\oplus
    H^{n-1}(F^{p}A^{\ast}_{\C})\to H^{n-1}(A^{\ast}_{\C}) \to\\
    H^{n}_{\fD}(A,p) \to 
    H^{n}(A^{\ast}_{\R}(p))\oplus
    H^{n}(F^{p}A^{\ast}_{\C})\to H^{n}(A_{\C}^{\ast})\to\dots
  \end{multline*}
\end{prop}
\begin{proof}
The result follows from the standard long exact sequence associated to
the simple of a morphism. 
\end{proof}

\begin{rmk}\label{rem:10} If the spectral sequence associated to the
  Hodge filtration $F$ degenerates at the term $E_{1}$ as happens with
  the Dolbeault complexes $E_{X}$ and $D_{X}$ for $X$ projective, the
  above exact sequence can be written as
  \begin{multline*}
    \dots \to H^{n-1}(A^{\ast}_{\R}(p))\oplus
    F^{p}H^{n-1}(A^{\ast}_{\C})\to H^{n-1}(A^{\ast}_{\C}) \to\\
    H^{n}_{\fD}(A,p) \to 
    H^{n}(A^{\ast}_{\R}(p))\oplus
    F^{p}H^{n}(A^{\ast}_{\C})\to H^{n}(A_{\C}^{\ast})\to\dots
  \end{multline*}
\end{rmk}

\begin{ex}\label{exm:2} Let $X$ be a smooth projective variety over
  $\C$.
  \begin{enumerate}
  \item \label{item:6} The complex $\fD_{t}(E_{X},p)$ will be denoted
    as $\fD_{t}(X,p)$  and its cohomology by
    $H^{\ast}_{\fD}(X,\R(p))$. It is called the 
    Deligne cohomology of $X$. The complex $\fD_{t}(D_{X},p)$ will
    be denoted  $\fD_{t,D}(X,p)$. This complex is quasi-isomorphic to
    $\fD_{t}(X,p)$, thus it also computes the Deligne cohomology of $X$.
  \item  \label{item:7} Let $Y$ be a simple normal
  crossing divisor of $X$ and $U=X\setminus Y$. Then the cohomology
  of the complex $\fD_{t}(E_{X}(\log Y),p)$ only depends on $U$ and not on
  the compactification $X$. We will denote it as
  $H^{\ast}_{\fD}(U,\R(p))$.  This cohomology is called
  Deligne-Beilinson cohomology of $U$. In fact,
  \begin{displaymath}
    H^{\ast}_{\fD}(U,\R(p))=H^{\ast}(\fD_{t}(E_{U,\log},p)),
  \end{displaymath}
  where $E_{U,\log}$ was introduced in Example
  \ref{ex:1}~\eqref{item:2}.
  We will write $\fD_{t,\log}(U,p)\coloneqq \fD_{t}(E_{U,\log},p)$. 
\item \label{item:8} Let now $Z\subset X$ be a proper closed subset and write
  $V=X\setminus Z$. Then there is a morphism of Dolbeault complexes 
  $E_{X}\to E_{V,\log}$ that induces morphisms $\fD_{t}(X,p)\to
  \fD_{t,\log}(V,p)$. We denote
  \begin{displaymath}
    \fD_{t,Z}(X,p)=s(\fD_{t}(X,p)\to \fD_{t,\log}(V,p)),
  \end{displaymath}
  the simple complex associated to these morphisms. We define the Deligne-Beilinson
  cohomology of $X$ with support on $Z$ as
  \begin{displaymath}
    H^{\ast}_{\fD,Z}(X,\R(p))=H^{\ast}(\fD_{t,Z}(X,p)).
  \end{displaymath}
\item \label{item:9} Let $\caZ^{p}$ denote the directed set of algebraic closed
  subsets of $X$ of codimension at least $p$ ordered
  by inclusion. Then we will write
  \begin{displaymath}
    \fD_{t,\caZ^{p}}(X,p)=\lim_{\substack{\longrightarrow\\Z\in
        \caZ^{p}}}s(\fD_{t}(X,p)\to \fD_{t,\log}(X\setminus Z,p))
  \end{displaymath}
  and
  \begin{displaymath}
    H^{\ast}_{\fD,\caZ^{p}}(X,\R(p))=H^{\ast}(\fD_{t,\caZ^{p}}(X,p)). 
  \end{displaymath}
  \end{enumerate}
\end{ex}

\begin{prop}\label{prop:4}
  Let $X$ be a smooth complex variety.
  \begin{enumerate}
  \item \label{item:4} Let $Z\subset X$ be an irreducible
  subvariety of codimension $p$. Then
  \begin{displaymath}
    H^{n}_{\fD,Z}(X,\R(p))=
    \begin{cases}
      0,& \text{ if }n < 2p,\\
      \R\cdot [Z],& \text{ if }n= 2p,
    \end{cases}
  \end{displaymath}
    where $[Z]$ is the fundamental class of $Z$.
  \item \label{item:5} Moreover
  \begin{displaymath}
    H^{n}_{\fD,\caZ^{p}}(X,\R(p))=
    \begin{cases}
      0,& \text{ if }n < 2p,\\
      Z^{p}(X)\otimes \R, &\text{ if }n= 2p,
    \end{cases}
  \end{displaymath}
  where $Z^{p}(X)$ is the group of codimension $p$ cycles on $X$.
  \end{enumerate}
\end{prop}
\begin{proof}
  By purity, the cohomology of a smooth variety with support on an irreducible
  subvariety of codimension $p$ satisfies
  \begin{displaymath}
    H^{n}_{Z}(X,\R)=
    \begin{cases}
      0,& \text{ if }n< 2p,\\
      \R\cdot[Z],& \text{ if }n= 2p,\\
    \end{cases}
  \end{displaymath}
  Moreover, the fundamental class satisfies
  \begin{displaymath}
    [Z]\in  F^{p}H^{2p}_{Z}(X), \quad
    [Z]\in  H^{2p}_{Z}(X,(2\pi i)^{p}\R).
  \end{displaymath}
  Therefore statement \eqref{item:4} follows from the long exact
  sequence of Proposition \ref{prop:5}.

  Statement \eqref{item:5} follows from statement \eqref{item:4}.
\end{proof}

\begin{rmk}\label{rem:6} The Deligne-Beilinson cohomology groups in
  degree bigger that $2p$ can be seen as pathological (see
  \cite{Beilinson:naHc}). It is 
  convenient to truncate the total Deligne complex in degree $2p$.
  Let $\tau _{\leq 2p}$ be the canonical truncation given by
  \eqref{eq:30}. 

We will use the notation
\begin{displaymath}
  \tfD_{t}(X,p)=\tau _{\leq 2p}\fD_{t}(X,p),\quad  \tfD_{t,Z}(X,p)=\tau _{\leq
      2p}\fD_{t,Z}(X,p).
\end{displaymath}
\end{rmk}

\begin{ex}\label{exm:3} Assume that $X $ is projective. Thus we can
  use the complex $\fD_{t,D}(X,p)$ to compute the Deligne
  cohomology of $X$. Let $Z$ be a closed subvariety of $X$ of
  codimension $p$. Recall that, following the notation \eqref{eq:8},
  the current $\delta _{Z}$ already incorporates a twist. Therefore it
  belongs at the same 
  time to $(2\pi i)^{p}D^{2p}_{X,\R}$ and to $
  F^{p}D^{2p}_{X}$ and it represents  
  the image of the fundamental class $[Z]$ in
  $H^{2p}(X,\C)$. The triple 
  \begin{equation}\label{eq:75}
    (\delta _{Z},\delta _{Z},0)\in \fD_{t,D}(X,p).
  \end{equation}
  is closed and represents a class in $H_{\fD}^{2p}(X,p)$ that is
  called the fundamental class of $Z$ and is denoted as $[Z]$.
  
  By abuse of language, the triple \eqref{eq:75} will also be denoted
  $\delta _{Z}$.  
\end{ex}

\subsection{The Thom-Whitney simple for Deligne cohomology }  
\label{sec:thom-whitney-simple}

Assume that $A$ is a Dolbeault algebra. This is the case for the Hodge complexes 
\eqref{item:1} and \eqref{item:2} of
Example \ref{ex:1}. Then the complex $\fD_{t}(A,p)$ has
several product structures \cite{Beilinson:naHc},
\cite{delignebeilinson},  one for each $\beta \in \R$. All of them are
homotopically eqivalent. The one for $\beta =1/2$ is graded commutative while
the ones with $\beta =0,1$ are associative, but none of them is graded
commutative and associative at the same time. To have a graded commutative and
associative algebra we use the Thom-Whitney simple as in
\cite[\S 6]{BurgosWang:hBC}.

  Denote by
  $L_{\R}=(L_{\R}^{\ast},d)$ the algebraic de Rham complex of
  $\A^{1}_{\R}$, that is,
  \begin{displaymath}
    L^{0}_{\R}=\R[\varepsilon ],\quad L^{1}_{\R}=\R[\varepsilon ]d\varepsilon, 
  \end{displaymath}
  where $\varepsilon $ is an indeterminate.

\begin{df} Let $A$ be a Dolbeault complex and $p$ an integer.  Then the
  Thom-Whitney Deligne complex of $A$ twisted by $p$, $\fDTW(A,p)$ is
  the subcomplex of
  \begin{displaymath}
    (2\pi i)^{p} A_{\R}^{\ast}\,\oplus \,
      F^{p}A^{\ast}_{\C}\, \oplus\, L^{\ast}_{\R}\otimes A^{\ast}_{\C}
  \end{displaymath}
given by
  \begin{displaymath}
    \fDTW(A,p)=\left \{(r,f,\omega )
       \middle | \,
     \begin{gathered}
       \omega \mid_{\varepsilon =0}=\alpha (r), \\
       \omega \mid_{\varepsilon =1}=\iota(f)
     \end{gathered}
     \right\},       
  \end{displaymath}
\end{df}
where $\iota_{p} $ and $\alpha_{p} $ are as in Definition \ref{def:2}. 
Since in a Dolbeault complex, the maps $\alpha_{p} $ and $\iota_{p} $ are
injective, the information conveyed by $r$ and $f$ is
redundant. Therefore we will simplify the notation by writing
   \begin{equation}\label{eq:9}
     \fDTW(A,p)=\left\{\omega \in L^{\ast}_{\R}\otimes 
       A_{\C}^{\ast }\middle |
       \begin{gathered}
       \omega |_{\varepsilon =0}\in  A^{\ast}_{\R}(p),\\
     \omega |_{\varepsilon =1}\in  F^{p}A^{\ast}_{\C}.
   \end{gathered}
   \right\}
   \end{equation}

\begin{ex}\label{exm:1} Let $X$ be a smooth projective variety over
  $\C$. The Thom-Whitney simple $\fDTW(E_{X},p)$ will
    be denoted by $\fDTW(X,p)$. The
    Thom-Whitney simple $\fDTW(D_{X},p)$ will be denoted by
    $\fD_{\TW,D}(X,p)$.
  Then
   \begin{displaymath}
     \fDTW^{\bullet}(X,\ast)=\bigoplus_{p}\fDTW^{\bullet}(X,p)
   \end{displaymath}
    is a
    bigraded associative algebra. It is graded-commutative with
    respect to the first degree. That is, if $\omega \in
    \fDTW^{n}(X,p)$ and $\omega '\in \fDTW^{m}(X,q)$, then
    \begin{displaymath}
      \omega \cdot \omega '=(-1)^{nm} \omega' \cdot \omega.
    \end{displaymath}
    The Thom-Whitney simple $\fD_{\TW,D}(X,\ast)$ is a module over
    $\fDTW(X,\ast)$.

    Following the notation introduced in Remark \ref{rem:6} we will
    use the notation
    \begin{equation}\label{eq:31}
      \tfDTW^{\bullet}(X,p)=\tau _{\leq
        2p}\fDTW^{\bullet}(X,p),\quad
      \tfDTW^{\bullet}(X,\ast)=\bigoplus_{p}\tfDTW^{\bullet}(X,p). 
    \end{equation}
    Clearly $\tfDTW^{\bullet}(X,\ast)$ is still an associative,
    graded commutative algebra. 
  \end{ex}

\begin{ex}\label{exm:4} Let $X$ be a smooth projective variety over
  $\C$ and $Z$ a codimension $p$ subvariety. Then, using the shorthand
  \eqref{eq:9}, the class $[Z]$ is represented in $\fD_{\TW,D}^{2p}(X,p)$ by
  \begin{displaymath}
    \delta _{Z}=1\otimes \delta _{Z}+d\varepsilon \otimes 0.
  \end{displaymath}
\end{ex}

\subsection{The Deligne complex}
\label{sec:deligne-beil-cohom}

Following \cite{Deligne:dc}, in \cite{Burgos:CDB} a concise complex
that computes Deligne cohomology was introduced.  If $A $ is a
Dolbeault complex, the associated Deligne complex is denoted
  $\fD(A,p)$. It is given by
\begin{displaymath}
  \fD^{n}(A,p)=
  \begin{cases}\displaystyle
    A^{n-1}_{\R}(p-1)\cap
    \bigoplus_{\substack{p'+q'=n-1\\p'<p,\ q'<p}}A^{p',q'}_{\C},&\text{
      if }n < 2p,\\
    \displaystyle
    A^{n}_{\R}(p)\cap
    \bigoplus_{\substack{p'+q'=n\\p'\ge p,\ q'\ge p}}A^{p',q'}_{\C},&\text{
      if }n \ge 2p.
  \end{cases}
\end{displaymath}
with differential $d$ given, for $x\in \fD^{n}(A,p)$ by
\begin{displaymath}
  dx=
  \begin{cases}
    -\pi (dx),&\text{ if }n<2p-1,\\
    -2\partial\bar \partial x,&\text{ if }n=2p-1,\\
    d(x),&\text{ if }n>2p-1,
  \end{cases}
\end{displaymath}
where, for $n<2p-1$, $\pi\colon A^n(\C)\rightarrow \fD^{n}(A,p)$ is the 
projection $\pi =\pi _{p-1}\circ F^{p-1,p-1}$. There are homotopy
equivalences
\begin{displaymath}
      \xymatrix{ 
      \DB_{t}(A,p)\ar@<5pt>[r]^{H } & \DB(A,p)\ar@<5pt>[l]^{G }
    }       
\end{displaymath}
given, for $(r,f,\omega )\in \fD_{t}^{n}(A,p)$, by 
 \begin{displaymath}
    H(r,f,\omega)=
    \begin{cases}
      \pi(\omega), &\text{if }n\le 2p-1,\\
      F^{p,p}r+2\pi _{p}(\partial
      \omega^{p-1,n-p+1}),
      &\text{if }n\geq 2p,
    \end{cases}
  \end{displaymath}
and, for $x\in \fD^{n}(A,p)$, by
\begin{displaymath}
  G (x)=
  \begin{cases}
    (\partial x^{p-1,n-p}-\bar \partial x^{n-p,p-1},2\partial
    x^{p-1,n-p},x),&\text{ if }n\le 2p-1,\\
    (x,x,0),&\text{ if }n\ge 2p.
  \end{cases}
\end{displaymath}
If $A$ is a Dolbeault algebra, then $\fD(A,\ast)$ has a graded
commutative product given by
\begin{equation}
  \label{eq:47}
  x\cdot y=H(G (x) \star_{1/2} G(y)).
\end{equation}
This product is only associative up to homotopy. See
  \cite[\S 2]{Burgos:CDB} for more details.

The construction above can be applied to the Dolbeault complexes 
  $E_{X}$, 
  $E_{X}(\log Y)$, $E_{U,\log}$ and $D_{X}$. We will use the notation
  \begin{displaymath}
    \DB(X,p)=\DB(E_{X},p),\quad
    \DB_{\log}(U,p)=\DB(E_{U,\log},p),\quad
    \DB_{D}(X,p)=\DB(D_{X},p).
  \end{displaymath}
  
\begin{ex}\label{exm:6}
  Summarizing, let $A$ be a Dolbeault complex, Then we have at our 
disposal the following diagram of complexes and morphisms
  \begin{equation}\label{eq:43}
    \begin{gathered}
    \xymatrix{ 
      \DB_{\TW}(A,p) \ar@<5pt>[r]^{I} &
      \DB_{t}(A,p)\ar@<5pt>[r]^{H } \ar@<5pt>[l]^{E}&
      \DB(A,p)\ar@<5pt>[l]^G    } ,      
    \end{gathered}
  \end{equation}
  where the arrows are homotopy equivalences. The leftmost 
  complex has the advantage that, when $A$ is a Dolbeault algebra,  it
  is an associative and 
  graded commutative algebra. On the middle complex, we have several
  product structures, but none is at the same time graded commutative
  and associative.
  The rightmost 
  complex is the smallest one and gives a more concise description of
  Deligne cohomology but again has the disadvantage that the product
  \eqref{eq:47} is only associative up to homotopy.

  In particular, if $X$ be a smooth projective variety over $\C$, we
  can specialize diagram \eqref{eq:43} to the case $A=E^{\ast}_{X}$ to
  obtain a diagram 
  \begin{equation}\label{eq:33}
    \begin{gathered}
    \xymatrix{
      \DB_{\TW}(X,p) \ar@<5pt>[r] &
      \DB_{t}(X,p)\ar@<5pt>[r] \ar@<5pt>[l]& \DB(X,p)\ar@<5pt>[l]
    }.     
    \end{gathered}
  \end{equation}
\end{ex}

\begin{ex}\label{exm:7}
  There are several variants of the diagram in Example
  \ref{exm:6}. First we can use 
  currents instead of differential forms, in this case we obtain the
  diagram
  \begin{equation}\label{eq:34}
    \begin{gathered}
    \xymatrix{ 
      \DB_{\TW,D}(X,p) \ar@<5pt>[r] &
      \DB_{t,D}(X,p)\ar@<5pt>[r] \ar@<5pt>[l]& \DB_{D}(X,p)\ar@<5pt>[l]
    }.
    \end{gathered}
  \end{equation}
  The complexes of this diagram are covariant with respect to proper
  morphisms.

  Similarly, if $U$ is quasi-projective, we can use differential forms
  with logarithmic singularities at infinity as in Example \ref{exm:2}
  \ref{item:7} to obtain a diagram
  \begin{equation}\label{eq:35}
    \begin{gathered}
    \xymatrix{   \DB_{\TW,\log}(U,p) \ar@<5pt>[r] &
      \DB_{t,\log}(U,p)\ar@<5pt>[r] \ar@<5pt>[l]& \DB_{\log}(U,p)\ar@<5pt>[l]
    }
    \end{gathered}
  \end{equation}
  of complexes that compute the Deligne-Beilinson cohomology of $U$.
  
   We can also define complexes that
  compute Deligne-Beilinson cohomology with support in a subvariety or in
  a family of supports using examples \ref{exm:2}~\ref{item:8} and
  \ref{exm:2}~\ref{item:9}.
  \end{ex}

\subsection{An analytic lemma}
\label{sec:an-analytic-lemma}

We gather in this section some analytic formulas that will allow us
later to make formal computations with respect differential forms and
currents.

First we adapt the notion of basic Green form from
\cite[4.6]{Burgos:Gftp} to the Thom-Whitney complex. Note that this is
a variant of the notion of Green form of logarithmic type in
\cite[1.3.2]{GilletSoule:ait}.

\begin{df}\label{def:23}
  Let $X$ be a complex projective manifold, $D$ a simple normal
crossings divisor and $U=X\setminus D$.
  Let $Y$ be a codimension $p$ cycle on $U$. A \emph{Green form} for
  $Y$ is an element $g_{Y}\in 
  \DB_{\TW,\log}^{2p-1}(U\setminus |Y|,p)$ of the form
  \begin{displaymath}
    (\varepsilon + 1) \otimes \partial g +
    (\varepsilon - 1) \otimes \bar \partial g + d\varepsilon \otimes
    g
  \end{displaymath}
  such that, on $U$,
  \begin{displaymath}
    d [g_{Y}] + \delta _{Y}=[\omega _{Y}],
  \end{displaymath}
  where $\omega _{Y}\in \DB_{\TW,\log}^{2p}(U\setminus |Y|,p)$. 
  We say that the Green form $g_{Y}$ is a \emph{basic Green form} if
  there exists a resolution of singularities $(\widetilde U,E)$ of
  $(U,Y)$ such that, locally in any coordinate neighborhood with
  coordinates $(z_{1},\dots, z_n)$ where $E$ is given by $z_1\dots
  z_{k}=0$, then $g$ can be written as
  \begin{displaymath}
    g=\sum_{i=1}^{k}-\alpha _{i}\log |z_{i}| + \beta 
  \end{displaymath}
  where $\alpha _{i}$, $i=1,\dots,k$ and $\beta $ are smooth forms and
  $\alpha _{i}|_{\{z_i=0\}}$ is $ \partial$-closed and $\bar
  \partial$-closed. 
\end{df}

We also define the product of a form with logarithmic singularities and the current
integration along a cycle.

\begin{df}\label{def:22} Let $X,D, U$ as in Definition
  \ref{def:23}. Let $Y$ be a prime cycle on $U$ and $g\in
  \DB_{\TW,\log}^{n}(U,p)$. Let $\overline Y$ be the closure of $Y$ on
  $X$ and $\widetilde Y$ a resolution of singularities of $\overline
  Y$. Let $\eta\colon \widetilde Y \to X$ the induced map. If 
  $\eta^{\ast}g$ is locally integrable in $\widetilde Y$ then we define the product
  \begin{displaymath}
    \delta _{Y}\cdot g = \eta_{\ast}[\eta^{\ast} g].
  \end{displaymath}
  Note that, even if $Y$ is a cycle on $U$, the current $\delta
  _{Y}\cdot g $ is a current in the whole $X$. 
  If $g$ is also locally integrable on $X$, then this product is also
  denoted as $\delta _{Y}\cdot [g]$. We extend this product to
  arbitrary cycles by linearity. Since $\delta _{Y}$ is of even
  degree, we write
  \begin{displaymath}
    g_{1}\cdot \delta _{Y}\cdot g_{2}=\delta _{Y}\cdot g_{1}\cdot g_{2}.
  \end{displaymath}
\end{df}

We use a similar notation for inverse images.

\begin{df}\label{def:25} Let $X,D, U$ and  $g$ be as before. Let $f\colon Y\to X$
  be a morphism such that $f^{-1}(D)$ is a simple normal crossings
  divisor.  If $g$ and $f^{\ast}g$ are locally integrable, we write 
  \begin{displaymath}
    f^{\ast}[g]= [f^{\ast} g].
  \end{displaymath}
\end{df}

\begin{lem}\label{lemm:13} Let $X$, $D$ and $U$ be as before. 
  Let $Y_{0},\dots,Y_{r}$ be cycles on $U$ intersecting properly of
  codimension $p_{0},\dots,p_{r}$ respectively,
  $g_{i}\in \DB_{\TW,\log}^{n_{i}}(U\setminus |Y_{i}|,q_{i})$ and
  $1\le s \le r$. Assume that, for $1\le i< s$, the condition
  $n_{i}<2p_{i}-1$ hold, while for $s\le i\le r$, $g_{i}$ is a basic
  Green form for the cycle $Y_{i}$, satisfying the equation
  $d[g_{i}]+\delta _{Y_{i}}=\omega _{i}$. Assume furthermore that, for
  each irreducible component $D_{j}$ of $D$ there is an index $i(j)$
  such that $g_{i(j)}$ is smooth in a neighbourhood of $D_{j}$ and
  $g_{i(j)}|_{D_{j}}=0$. Then the following statements hold:
  \begin{enumerate}
  \item The form $g_{1}\cdots g_{r}$, and the forms $g_1\cdots dg_i\cdots g_r$ for $i=1,\cdots,r$ are locally integrable on $X$, and
    \begin{multline*}
      d[g_{1}\cdots g_{r}]=\sum_{i=1}^{s-1}(-1)^{n_{1}+\dots +
        n_{i-1}}[g_{1}\cdots dg_{i}\cdots
      g_{r}]\\
      +\sum _{i=s}^{r}(-1)^{n_{1}+\dots + n_{i-1}}([g_{1}\cdots \omega _{i}\cdots
      g_{r}] - g_{1}\cdots \delta _{Y_{i}}\cdots
      g_{r}).
    \end{multline*}
    \item Let $f\colon Y \to X$ be a morphism of projective complex
    manifolds such that $f^{-1}(D)$ is a simple normal crossings
    divisor. Write $V=Y\setminus f^{-1}(D)$ and $f_{0}\colon V\to
    U$ for the restriction of $f$ to $V$. Assume furthermore that 
    $\codim_{V}f^{-1}_{0}Y_{i}\ge \codim_{U}Y_{i}$, $i=1,\dots,r$ and that
    the subvarieties $f_{0}^{-1}Y_{i} $ intersect properly in $V$.
    In particular there are well defined cycles $f_{0}^{\ast}Y_{i}$. Then the forms
    \begin{displaymath}
      f^{\ast} (g_{1}\cdots g_{r})\quad\text{and}\quad
      f^{\ast}(g_{1}\cdots d g_{i}\cdots  g_{r}),\ i=1,\dots,r
    \end{displaymath}
 are locally integrable on $Y$ and
    \begin{multline*}
      df^{\ast} [g_{1}\cdots g_{r}]=\sum_{i=1}^{s-1}(-1)^{n_{1}+\dots
        + n_{i-1}}f^{\ast}[g_{1}\cdots dg_{i}\cdots 
      g_{r}]\\
      +\sum _{i=s}^{r}(-1)^{n_{1}+\dots + n_{i-1}}(f^{\ast}[g_{1}\cdots \omega _{i}\cdots
      g_{r}] - f^{\ast}g_{1}\cdots \delta _{f_0^{\ast}Y_{i}}\cdots
      f^{\ast}g_{r}).
    \end{multline*}
  \item Let $\widetilde Y_{0}$ be a resolution of singularities on the
    closure $\overline Y_{0}$ and $\eta_{0}\colon \widetilde Y_{0}\to
    X$ as in Definition \ref{def:22}. Then the forms
    $$\eta_{0}^{\ast}
    (g_{1}\cdots g_{r})\quad\text{and}\quad\eta_{0}^{\ast}
    (g_{1}\cdots dg_{i}\cdots g_{r}),\hspace{0.1cm}i=1,\dots,r$$
    are locally
    integrable on $\widetilde Y_{0}$ and 
    \begin{multline*}
      d\delta _{Y_{0}}\cdot [g_{1}\cdots g_{r}]=\sum_{i=1}^{s-1}(-1)^{n_{1}+\dots + n_{i-1}}\delta _{Y_{0}}\cdot [g_{1}\cdots dg_{i}\cdots
      g_{r}]\\
      +\sum _{i=s}^{k}(-1)^{n_{s}+\dots + n_{i-1}}(\delta _{Y_{0}}\cdot [g_{1}\cdots \omega _{i}\cdots
      g_{r}] - [g_{1}\cdots \delta _{Y_{0}\cdot Y_{i}}\cdots
      g_{r}]),
    \end{multline*}
    where the intersection $Y_{0}\cdot Y_{i}$ means the intersection of
    cycles in $U$.
  \end{enumerate}
\end{lem}
\begin{proof}
  To prove the lemma, first one unwraps the definition of the
  Thom-Whitney complex to obtain statements about usual differential
  forms. Then the proof that all the involved differential forms are locally
  integrable uses the techniques of \cite[Proposition
  3.3]{Burgos:Gftp}. Note that an arbitrary form with logarithmic
  singularities does not need  to be locally integrable. The key point
  in the proof of local integrability is that wherever one of the
  forms may have arbitrary logarithmic singularities (any component of $D$),  another 
  of the forms vanishes. Thus the
  vanishing condition of some of 
  the forms along the divisor $D$ is a necessary condition for the
  local integrability. The proof of the equations involving the
  differential use the techniques of the
  proof of \cite[2.1.4]{GilletSoule:ait},
  \cite[2.2.2]{GilletSoule:ait}. Notationally the proof is more
  involved because it deals with a product of an arbitrary number of
  functions while in \emph{loc. cit.} it involves at most three
  functions. Conceptually is simpler because here we are assuming that
  all the intersections are proper.  
\end{proof}

\section{The cycle class map}
\label{sec:regulator-map-cycles}

The aim of this section is to recall two explicit incarnations of the
cycle class map from higher Chow groups to Deligne
cohomology. The first is a cubical version of Goncharov cycle class
map defined in \cite{Goncharov:prAmc}, while the second is a minor
variant of a construction by Bloch in \cite{Bloch:acbc}, that was used
in \cite{BurgosFeliu:hacg}. After composing with the isomorphism
between $K$ theory and higher Chow groups, both maps induce
Beilinson's regulator.

\subsection{Wang forms and Goncharov regulator}
\label{sec:wang-forms-goncharov}

Consider the inclusion $\G_{m}(\C)\subset \P^{1}(\C)$, with absolute
coordinate $t$ and write
$D=\P^{1}(\C)\setminus \G_{m}(\C)=\{0,\infty\}$. Let
$\lambda \in (L_{\C}\otimes E_{\P^{1}}(\log D))^{1}$ be the element given by
\begin{equation}
  \label{eq:10}
  \lambda =\frac{-1}{2}\left((\varepsilon +1)\otimes \frac{dt}{t}+
  (\varepsilon -1)\otimes \frac{d\bar t}{\bar t}
  +d\varepsilon\otimes \log t\bar t\right).
\end{equation}
Then
\begin{align*}
  \lambda |_{\varepsilon =0}&=\frac{-1}{2}\left(\frac{dt}{t}-
    \frac{d\bar t}{\bar t}\right)\in (2\pi i)E^{1}_{\P^{1},\R}(\log
                              D),\\
  \lambda |_{\varepsilon =1}&=\frac{-dt}{t}\in F^{1}E^{1}_{\P^{1}}(\log
                              D).
\end{align*}
Therefore, following the shorthand \eqref{eq:9}, we deduce that
$$\lambda \in \DB^{1}_{\TW}(E_{\P^{1}}(\log D),1).$$
The form $\lambda $
satisfies the properties
\begin{align}
  \lambda \mid _{t=1}&=0\label{eq:11}\\
  d \lambda &=0 \label{eq:12}.
\end{align}
But we are more interested in the associated current. Write
\begin{displaymath}
  [\lambda ]= \frac{-1}{2}\left((\varepsilon +1)\otimes \left[
    \frac{dt}{t}\right ]+
  (\varepsilon -1)\otimes \left[ \frac{d\bar t}{\bar t}\right]
  +d\varepsilon\otimes [ \log t\bar t ]\right).
\end{displaymath}
One easily verifies that $[\lambda ]\in \DB_{\TW,D}^{1}(\P^{1},1)$
and, using equations \eqref{eq:77},
that it satisfies
\begin{equation}\label{eq:13}
  d[\lambda ]=-1\otimes \delta _{\Div t}=\delta _{\infty}-\delta _{0}.
\end{equation}
In conclusion $\lambda $ is a basic Green form for the cycle
$[0]-[\infty]$. 

We now consider the cocubical complex manifold $\P^{1}(\C)^{\cdot}$
and its open cocubical submanifold $\square ^{\cdot}$. Let $D_{n}\subset
\P^{1}(\C)^{n}$ be the normal crossing divisor given as the union of
all cofaces of $\P^{1}(\C)^{\cdot}$. That is
\begin{displaymath}
  D_{n}=\left\{(t_{1},\dots,t_{n})\in \P^{1}(\C)^{n}\, \middle |\,  \exists i,\
  t_{i}\in \{0,\infty\}\right\}
\end{displaymath}

On $(\P^{1}(\C))^{n}\setminus D_{n}$, for $n\ge 0$, we consider the forms
\begin{align*}
  W_{0}&=1\\
  W_{n}&=\pi _{1}^{\ast}\lambda \cdots \pi _{n}^{\ast}\lambda, \ n>0, 
\end{align*}
where $\pi _{i}\colon (\P^{1})^{n}\to \P^{1}$ is the projection onto
the $i$-th factor.
Clearly  $W_{n}\in \DB^{n}_{\TW}(E^{\ast}_{(\P^{1})^{n}}(\log
D_{n}),n)$ and satisfies
\begin{align}
  W_{n}|_{(\P^{1})^{n}\setminus \square^{n}}&=0,\label{eq:15}\\
  d W_{n}&=0,\label{eq:16}
\end{align}
while the associated current $[W_{n}]\in \DB_{\TW,D}((\P^{1})^{n},n)$
satisfies
\begin{equation}
  \label{eq:17}
  d [W_{n}]=\sum _{i=1}^{n}\sum _{j=0,1} (-1)^{i+j} (\delta
  ^{i}_{j})_{\ast}[W_{n-1}]. 
\end{equation}

This formula follows from the fact that $\lambda $ is a basic Green
form for the cycle $[0]-[\infty]$ and Lemma
\ref{lemm:13}. Alternatively, the currents
$[W_{n}]$ can be defined using the  
\emph{Cartesian product} (also called tensor product or direct
product) of currents
\begin{displaymath}
  [W_{n}]=[W_{1}]\times \dots \times [W_{1}].
\end{displaymath}
Then equation \eqref{eq:17} follows from \eqref{eq:13} and the fact
that the Cartesian product satisfies Leibnitz rule
\cite[4.1.8]{Federer:GMT}. 

Let $\fA$ be any of the complexes of diagram \eqref{eq:43}. Then we
will denote by $W_{n,\fA}$ the image of $W_{n}$ in the corresponding
complex of forms in \eqref{eq:33} and $[W_{n}]_{\fA}$ the image of
$[W_{n}]$ in the  
complex of currents in  \eqref{eq:34}. For instance
\begin{displaymath}
  [W_{n}]_{ \fD}\in \fD_{D}^{n} ((\P^{1})^{n},n)
\end{displaymath}
is the cubical version of the current originally used by
Goncharov. For further reference we note
that
\begin{equation}\label{eq:88}
  [W_{1}]_{\DB}=\left[-\frac{1}{2}\log t\bar t\right].
\end{equation}

All the different families of currents $[W_{n}]_{ \fA}$ satisfy the analogue
of equation \eqref{eq:17}. For shorthand, If there is no decoration,
we will  mean that Thom-Whitney version. That is
\begin{displaymath}
  [W_{n}]=[W_{n}]_{\fDTW}.
\end{displaymath}

We can now describe the cubical version of Goncharov's regulator for
projective varieties. 

\begin{df}\label{def:7}
  Let $X$ be a smooth projective complex variety of dimension $d$. We denote by
  $p_{1},p_{2}$ the two projections from $X\times (\P^{1})^{n}$ to the
  first and second factor. Let $Z\subset X\times \square ^{n}$ be a
  codimension $p$ subvariety that meets properly all the faces. We
  define the current $\caP(Z)\in
 \DB_{\TW,D}^{2p-n}(X,p) $ as 
  \begin{equation}\label{Regulator:Eqn}
    \caP(Z)=  (p_{1})_{\ast}\delta _{Z}\cdot [p_{2}^{\ast}W_{n}]. 
  \end{equation}
  This current is well defined thanks to Lemma \ref{lemm:13}. 
The map $\caP$ is extended by linearity to define the \emph{cubical
  Goncharov regulator map} $\caP\colon
Z^{p}(X,\ast)\to \DB_{\TW,D}^{2p-\ast}(X,p)$.

Definition \ref{def:7} can be adapted to any of the complexes of
currents of Diagram \eqref{eq:34}. We denote any such complex as
$\fA_D$ and $\caP_{\fA}(Z)$ will be the
image of $Z$ in any $\fA_D$. 
For instance 
\begin{displaymath}
  \caP_{\fD}(Z)\in\fD_{D}^{2p-n}(X,p) 
\end{displaymath}
is the cubical version of the original Goncharov regulator.
\end{df}

% \begin{rmk} For a pre-cycle $Z\in Z^{p}(X,n)$,
%   putting back all the implicit powers of $2\pi i$ and writing $f=p_{1}\circ
%   \iota \colon \widetilde Z \to X$ we obtain
%   \begin{displaymath}
%     \caP_{\fD}(Z)\in \fD_{D}^{2p-n}(X,p)\subset D^{2p-n-1}_{X}
%   \end{displaymath}
%   and, for $\eta\in E^{2d-2p+n+1}_{X}$,
%   \begin{displaymath}
%     \caP_{\fD}(Z)(\eta)=  \frac{1}{(2\pi i)^{d+n-p}}\int_{\widetilde
%       Z}\iota^{\ast}p_{2}^{\ast}[W_{n}]_{\fD}\land f^{\ast}(\eta).
%   \end{displaymath}
% \end{rmk}
\begin{prop}\label{prop:6}\
  
  \begin{enumerate}
  \item   The map $\caP\colon
    Z^{P}(X,\ast)\to \DB_{\TW,D}^{2p-\ast}(X,p)$ is a morphism of
    complexes. 
  \item If $\zeta \in \widetilde Z^{p}(X,\ast)$ is a
    degenerate pre-cycle then $\caP(\zeta)=0$.
  \item $\caP$ induces a morphism of complexes
    \begin{displaymath}
      \widetilde Z^{p}(X,\ast)\to
      \DB_{\TW,D}^{2p-\ast}(X,p) 
    \end{displaymath}
    that agrees with the restriction of
    $\caP$ to $NZ^{P}(X,\ast)$.
  \end{enumerate}
\end{prop}
\begin{proof}
    The first statement follows from equations Lemma \ref{lemm:13} and
  the fact that $\lambda $ is a basic Green form for the cycle
  $[0]-[\infty]$ that vanishes at the point $1$. Formally it follows
  from follows from equations \eqref{eq:15} and
  \eqref{eq:17}. The second statement is a direct check while the third
  statement is consequence of the other two. 
\end{proof}

\begin{rmk}\label{ReglRmk}
The regulator $\caP_{c}$ used in the paper \cite{BurgosFeliuTakeda} (in
section 7.2) agrees with $\mathcal{P}_{\fD}.$
\end{rmk}

\begin{rmk}
  The Goncharov regulator is very concrete and well suited for
  computations. Nevertheless it has some drawbacks. In particular, due
  to the use of currents, it is difficult to establish its
  contravariant functoriality and its multiplicativity.  In the paper
  \cite{BurgosFeliu:hacg}, E. Feliu and the first author have
  introduced another map that uses only differential forms and not
  currents, making it easier to establish the contravariance and
  multiplicative properties of the regulator. In the paper
  \cite{BurgosFeliuTakeda}, in joint work with Takeda they have used
  this new map to prove that the cubical version of Goncharov's
  regulator agrees after composing with the isomorphism with
  $K$-theory with Beilinson's regulator. The map of Burgos-Feliu
  is a minor modification  of a map introduced by Bloch
  \cite{Bloch:acbc}. 
\end{rmk}

Goncharov regulator is compatible with direct images with respect to
proper maps.

\begin{prop}\label{prop:3}
  Let $f\colon X\to Y$ be a morphism between smooth projective
  varieties over $\C$ of relative dimension $e$. Thus the map is
  proper and, for integers $p,n$, there are direct image maps
  \begin{displaymath}
    f_{\ast}\colon Z^{p}(X,n)_{0} \to Z^{p-e}(Y,n)_{0},\quad
    f_{\ast}\colon \fA_{D}^{2p-n}(X,p) \to \fA_{D}^{2p-n-2e}(X,p-e).
  \end{displaymath}
  Then the compatibility
  \begin{displaymath}
    \caP_{\fA}\circ f_{\ast}= f_{\ast}\circ \caP_{\fA}
  \end{displaymath}
  holds.
\end{prop}

\subsection{Differential forms and affine lines}
\label{diffformsaffine}
Let $X$ be a smooth complex projective variety and let 
$\fA$ denote any complex appearing in diagram
\eqref{eq:43}.
Then for every $n,p\geq 0$, let $\fA_{\log}^*(X\times
\square^n,p)$ denote the construction $\fA$ applied to the Dolbeault
complex  $E^{\ast}_{X\times
\square^n,\log}$. Let $\tfA_{\log}^*(X\times
\square^n,p)$ be its canonical truncation in degree $2p$. That is
\begin{displaymath}
  \tfA_{\log}^*(X\times\square^n,p)=\tau _{\leq 2p}\fA^*(E_{X\times
\square^n,\log},p).
\end{displaymath}
The structural maps of the cocubical
scheme $\square^{\cdot}$ induce a cubical structure
on $\tfA_{\log}^r(X\times \square^*,p)$ for every $r$ and $p$.

Following the convention of \S \ref{sec:cubic-coch-compl}, we consider
the $2$-iterated cochain complex 
\begin{displaymath}
  \tfA_{\A}^{r,-n}(X,p)\coloneqq \tfA_{\log}^r(X\times \square^n,p),
\end{displaymath}
with differentials $(d,\delta=\sum_{i=1}^n
(-1)^{i}(\delta_i^0-\delta_i^1))$. So $d$ is the differential in the
complex $\fA$, while $\delta $ denotes the differential in the
cubical complex.  
Let
\begin{displaymath}
  \tfA_{\A}^*(X,p)=s(\tfA_{\A}^{*,*}(X,p))
\end{displaymath}
be the associated simple complex. Its differential $d_s$ is given, for
every $\alpha\in 
\tfA_{\A}^{r,-n}(X,p)$, by $d_s(\alpha) = d(\alpha)
+(-1)^r \delta (\alpha)$.
Since we are using cubical structures, this complex does not compute the
right cohomology and we have to normalize it.

For every $r,n$, we write
\begin{displaymath}
  \tfA_{\A}^{r,-n}(X,p)_0=\tfA_{\log}^r(X\times \square^n,p)_0\coloneqq
N\tfA_{\log}^r(X\times \square^n,p).
\end{displaymath}
Hence $\tfA_{\A}^{\ast,\ast}(X,p)_0$ is the normalized $2$-iterated
complex and we denote by $\tfA^*_{\A}(X,p)_0$ the associated simple
complex. The next result is proved in \cite[Proposition
2.8]{BurgosFeliu:hacg} for Deligne-Beilinson cohomology. The proof
applies mutatis mutandis to any of the complexes of diagram
\eqref{eq:43}.  

\begin{prop}\label{affine3}\
  
  \begin{enumerate}
  \item \label{item:24} For every $n>0$ the complex $\tfA_{\A}^{*,-n}(X,p)_0$ is
    acyclic.
  \item \label{item:23} The natural morphism of complexes
  \begin{displaymath}
    \tfA^*(X,p)\coloneqq \tau _{\le 2p}\fA^*(X,p)= \tfA_{\A}^{*,0}(X,p)_0 \rightarrow
    \tfA_{\A}^{*}(X,p)_0 
  \end{displaymath}
  is a quasi-isomorphism. 
  \end{enumerate}
\end{prop}

\subsection{Cycle class map with differential forms}
\label{sec:cycle-class-map}

Let $X$ be a smooth projective variety over $\C$ as in the previous section.
Let $\caZ^p_{n,X}$ be the set of all 
Zariski closed subsets of $W\subset X\times \square^n$ of codimension
$\ge p$ such that, for every face $F$ of $\square^{n}$,
\begin{displaymath}
  \codim_{X\times F}(W\cap X\times F) \ge p.
\end{displaymath}
We order $\caZ^p_{n,X}$ by inclusion. It is a directed set and the
subsets of codimension exactly equal to $p$ is a cofinal subset. 
When there is no
source of confusion, we simply write $\caZ^p_n$ or even
$\caZ^p$ for $\caZ^p_{n,X}$. Consider the cubical abelian group 
\begin{displaymath}
\caH^p(X,*)\coloneqq \lim_{\substack{\longrightarrow\\Z\in \mathcal{Z}^p_*}}H^{2p}_{\fD,Z}(X\times
\square^*,\R(p)),
\end{displaymath}
with faces and degeneracies induced by those of $\square^{\ast}$. Let
$\caH^p(X,*)_0$ be the
associated normalized complex.

\begin{lem}\label{lemm:1} Let $X$ be a smooth projective complex
  variety. For every $p\geq 0$, there is a natural isomorphism of chain
  complexes,
  \begin{displaymath}
    Z^p(X,*)_{0}\otimes \R  \xrightarrow{\cong}  \caH^p(X,*)_{0}
  \end{displaymath}
sending $Z$ to $\cl(Z)=[Z]$.
\end{lem}
\begin{proof}
Follows from Proposition \ref{prop:4}.
\end{proof}

\begin{rmk}
Observe that the complex $\caH^{p}(X,*)_0$ has the same functorial
properties as $Z^p(X,*)_{0}\otimes \R$.
\end{rmk}
 
For $n\ge 0$ let $\fA^{*,-n}_{\A,\mathcal{Z}^p}(X,p)$ be the complex
\begin{equation}\label{eq:50}
  \fA^{*,-n}_{\A,\mathcal{Z}_{n}^p}(X,p)=
  \lim_{\substack{\longrightarrow\\Z\in \caZ^{p}_{n}}}s\left(
    \fA_{\A}^{\ast,-n}(X,p)\to \fA(E^{\ast}_{X\times
    \square^{n}\setminus Z, \log},p)
  \right)
\end{equation}
and write
\begin{equation}
  \label{eq:32}
  \tfA^{*,-n}_{\A,\mathcal{Z}_{n}^p}(X,p)=\tau _{\le
    2p}\fA^{*,-n}_{\A,\mathcal{Z}_{n}^p}(X,p).
\end{equation}
Note that we are truncating after taking the simple.
Varying $n$ the complexes \eqref{eq:32} form a cubical cochain complex
and its normalization 
is a $2$-iterated  cochain complex,
\begin{displaymath}
  \tfA_{\A,\caZ^p}^{r,-n}(X,p)_0
=N\tfA_{\A,\caZ^p}^{r,-n}(X,p).
\end{displaymath}
Following the convention of \S \ref{sec:cubic-coch-compl}, the
differential will be denoted by $(d,\delta )$ where $d$ is the 
differential coming from the complex direction and $\delta $ the
differential coming from the cubical direction.

We denote by $\tfA^{*}_{\A,\caZ^p}(X,p)_0$ the associated simple
complex and by $d_s$ its differential.

We shift from a cochain complex to a chain complex and denote by
$\tfA^{2p-*}_{\A,\caZ^p}(X,p)_0$ be the chain complex
whose $n$-graded piece is $\tfA^{2p-n}_{\A,\caZ^p}(X,p)_0$.

The fact that we are using the truncation implies that an element
\begin{math}
  \beta \in \tfA^{2p-n}_{\A,\mathcal{Z}^p}(X,p)_0
\end{math}
is of the form
\begin{displaymath}
\beta =  ((\omega_n,g_n),\dots,(\omega_0,g_0)),
\end{displaymath}
with $(\omega _{i},g_{i})\in \fA^{2p-n+i,-i}_{\A,\caZ^{p}_{i}}(X,p)$
and the top element
$(\omega _{n},g_{n})\in \fA^{2p,-n}_{\A,\caZ^{p}_{n}}(X,p)$
satisfies $d(\omega _{n},g_{n})=0$, thus is a cycle in the complex
$\fA^{\ast,-n}_{\A,\caZ^{p}_{n}}(X,p)$ .

\begin{prop}\label{difcubs}
For every $p\geq 0$, the family of morphisms
\begin{eqnarray*}
\tfA^{2p-n}_{\A,\mathcal{Z}^p}(X,p)_0 & \xrightarrow{\gamma'_1}&
\caH^{p}(X,n)_0 \\
((\omega_n,g_n),\dots,(\omega_0,g_0))  & \mapsto &
[(\omega_n,g_n)]
\end{eqnarray*}
defines a surjective quasi-isomorphism of chain complexes.
\end{prop}
\begin{proof}
  Except for the surjectivity, this statement for the complex $\fD$ is \cite[Proposition
  2.13]{BurgosFeliu:hacg}. The proof in \emph{loc. cit.} works for any
  of the complexes of Example \ref{exm:6} and shows the
  surjectivity. The key ingredients for the proof are
  Proposition \ref{prop:4} and Proposition \ref{cubchain2}. 
\end{proof}

We now consider the map of forgetting supports
\begin{displaymath}
  \begin{matrix}
    \rho\colon &  \tfA^{2p-n}_{\A,\mathcal{Z}^p}(X,p)_0 &
    \longrightarrow &  \tfA^{2p-n}_{\A}(X,p)_0\\
    & ((\omega_n,g_n),\dots,(\omega_0,g_0))  & \longmapsto &
    (\omega_n,\dots,\omega _{0}) 
  \end{matrix}
\end{displaymath}

\begin{thm}\label{Thm5.12}
  The map induced in cohomology by the zig-zag diagram
\begin{equation}
\xymatrix@C=1.5em{
& \caH^p(X,\ast )_0 & & \tfA^{2p-\ast }_{\mathbb{A}}(X,p)_0 \\
Z^p(X,\ast )_0 \ar[ur]^{\gamma _1} & & \tfA^{2p-\star
}_{\mathbb{A},\mathcal{Z}^p}(X,p)_0 \ar[ul]_{\gamma '_1}^{\sim}  \ar[ur]^{\rho} & 
}
\end{equation}
agrees with one induced by the map $\caP_{\fA}$ of Definition  \ref{def:7}
(Goncharov regulator).
\end{thm}

\begin{proof}
  This is essentially the content of \cite[Theorem
  7.7]{BurgosFeliuTakeda}.

  Since the tools used in the proof of Theorem \ref{Thm5.12} will be used
  later we summarize then here.  Consider the diagram of complexes
\begin{displaymath}
  \begin{gathered}
  \xymatrix@C=1.5em{
\caH^p(X,\ast )_0 & & \tfA^{2p-\ast }_{\mathbb{A}}(X,p)_0 \\
& \tfA^{2p-\star
}_{\mathbb{A},\mathcal{Z}^p}(X,p)_0 \ar[ul]_{\gamma '_1}^{\sim}  \ar[ur]^{\rho}
}  
  \end{gathered} 
\end{displaymath}
and denote
\begin{displaymath}
  \tfA^{2p-\ast}_{\A,\caH}(X,p)_{0}=s(\rho - \gamma '_{1})[-1]_{\ast}.
\end{displaymath}
Composing the map $\caP_{\fA}$ with the inverse of the isomorphism $\gamma _{1}$ of Lemma
\ref{lemm:1}, we obtain a morphism 
\begin{displaymath}
  \caP_{\fA}\colon \caH^p(X,\ast )_0 \longrightarrow \tfA_{D}^{2p-\ast}(X,p).
\end{displaymath}
%I think this is not needed
% There is a quasi-isomorphism
% \begin{displaymath}
%   \beta \colon \tAH^{2p-\ast}_{\A}(X,p)_{0}\xrightarrow{\sim
%   }\tAH^{2p-\ast}_{\A,\caH}(X,p)_{0} 
% \end{displaymath}
% given by $\beta (a)=(0,0,a)$.

There is a morphism that also represents the regulator map
\begin{displaymath}
  \rho \colon Z^p(X,\ast )_0 \longrightarrow \tfA^{2p-\ast}_{\A,\caH}(X,p)_{0}
\end{displaymath}
given by $\rho (Z)=(0,\gamma _{1}(Z),0)$.

Finally there is a morphism of complexes (\cite[Theorem 7.7]{BurgosFeliuTakeda}
) 
\begin{displaymath}
  \psi \colon \tfA^{2p-\ast}_{\A,\caH}(X,p)_{0} \longrightarrow
  \tfA^{2p-\ast}_{D}(X,p)
\end{displaymath}
given by
\begin{multline}\label{eq:28}
  \psi ((\omega _{n-1},g_{n-1}),\dots,(\omega _{0},g_{0}),Z,(\alpha
  _{n},\dots,\alpha _{0}))\\=
  \caP(Z)+\sum_{i=0}^{n}(p_{1})_{\ast}[\alpha _{i}\cdot W_{i,\fA}-g_{i}\cdot W_{i,\fA}].
\end{multline}
The fact that the forms appearing in \eqref{eq:28} are locally
integrable follows from Lemma \ref{lemm:13}. 

Putting together the above morphisms we obtain a commutative diagram,
from which the theorem is derived,
% I think the whole  diagram is not needed, only the next part.
% \begin{equation}\label{eq:29}
%   \xymatrix{
%     Z^{p}(X,\ast)_{0} \ar[r]^{\caP}\ar[d]_{\rho} &
%     \AH^{2p-\ast}_{\TW,D}(X,p)_{0}\ar@{=}[d]\\
%     \AH_{\A,\caH}^{2p-\ast}(X,p)\ar[r]^{\psi }&
%     \AH^{2p-\ast}_{\TW,D}(X,p)\\
%     \AH^{2p-\ast}_{\A}(X,p)_{0}\ar[u]^{\sim}_{\beta }\ar[r]_{I}&
%     \AH^{2p-\ast}_{\TW}(X,p).\ar[u]^{[\cdot]}
%   }
% \end{equation}

\begin{equation}\label{eq:29}
  \xymatrix{
    Z^{p}(X,\ast)_{0} \ar[r]^{\caP_{\fA}}\ar[d]_{\rho} &
    \tfA^{2p-\ast}_{D}(X,p)_{0}\ar@{=}[d]\\
    \tfA_{\A,\caH}^{2p-\ast}(X,p)\ar[r]^{\psi }&
    \tfA^{2p-\ast}_{D}(X,p)\hspace{0.1cm}.\\
  }
\end{equation}
\end{proof}

\begin{rmk}\label{rem:8}
  In \cite{BurgosFeliu:hacg}, it is the complex $\fD^{\ast}(X,\ast)$
  that is used. When discusing the product we will adopt the
  Thom-Whitney complex. 
\end{rmk}

\section{Green currents and  Green forms}
\label{sec:green-currents-green}

In this section we will introduce Green currents and Green forms of
logarithmic type for higher cycles and study their basic
properties. 

\medskip
\textbf{Notation}: The notations in this section will be the same that
has been used in \S \ref{sec:regulator-map-cycles}. Let $\fA$ be any
of the complexes that appear in Example \ref{exm:6} and Example
\ref{exm:7}. To fix ideas we
will be thinking mainly in 
the complex $\fD_{\TW}$ because it has better
multiplicative properties.  
We will use the
notation $\fA_{D}$ for the corresponding complex with currents
(Diagram \eqref{eq:34}). 

The
differential of any such complex (irrespective of the definition by
forms or
by currents) will be denoted by $d$. The elements of $\fA^*(X,\bullet)$
will be called forms while the elements of  $\fA^*_D(X,\bullet)$ will
be called currents.

We recall also the map
\begin{displaymath}
  \caP_{\fA}\colon Z^{p}(X,\ast)\to \fA_{D}^{2p-\ast}(X,p)
\end{displaymath}
given in Definition \ref{def:7}. We will also denote by $\caP_{\fA}$ the
induced morphism $
  \caP\colon Z^{p}(X,\ast)_0\to \fA_{D}^{2p-\ast}(X,p)$ (see Proposition
  \ref{prop:6}). 

  For any cochain complex $(A^{\ast},d_{A})$, we will denote
\begin{equation}
  \label{eq:85}
  \widetilde A^{n}=A^{n}/\im d_{A}.
\end{equation}

\subsection{Green currents}
\label{sec:green-currents}

Let $X$ be a smooth and projective variety of dimension $d$ defined
over $\C$.

\begin{df}\label{def:8}
  Let $Z\in Z^p(X,n)_{0}$ be a higher cycle (that is a pre-cycle $Z$
  such that  $\delta
  Z=0$). Then an \emph{$\fA$-Green 
    current} for $Z$ is a current
  \begin{displaymath}
    g_{Z}\in \fA ^{2p-n-1}_{D}(X,p)
  \end{displaymath}
  such that
  \begin{equation}
    \label{eq:21}
   \caP_{\fA}(Z)+dg_Z=[\omega_{Z}],  \text{
     for}\hspace{0.1cm}\omega_{Z} \in \fA^{2p-n}(X,p).
  \end{equation}
  In other words $\caP_{\fA}(Z)+dg_Z$ is the image of a smooth
  form. 
  A \emph{class of Green currents} is the class of a Green current in
  the quotient
  \begin{displaymath}
    \widetilde  {\fA} ^{2p-n-1}_{D}(X,p)=\fA ^{2p-n-1}_{D}(X,p)/\im d.
  \end{displaymath}
  If $g_{Z}$ is a Green current, its class will be denoted $\widetilde
  g_{Z}$.
  If $g_{Z}$ is a Green current for the cycle $Z$ and $\widetilde
  g_{Z}$ is its class, we denote
  $\omega (g_{Z})=\omega (\widetilde g_{Z})$ the form in
  $\fA^{2p-n}(X,p)$ such that
  \begin{displaymath}
    [\omega (g_{Z})]= \caP_{\fA}(Z)+dg_Z.
  \end{displaymath}
\end{df}

Note the parallel
between this definition and \cite[Definition
1.2.3]{GilletSoule:ait}. Observe also that equation \eqref{eq:21}
implies that $\omega _{Z}$ represents the cycle class of $Z$ in
$H_{\fA}^{2p-n}(X,\R(p))$.

\begin{rmk}\label{rem:2} From the Green current alone, it is not
  always possible to recover the cycle $Z$ and the smooth form $\omega
  _{Z}$. For instance in Remark \ref{rem:13} we will see that $0$ can
  be a Green current for a nonzero cycle. Therefore the cycle $Z$ is
  part of the data of the Green current and will always be either
  explicit or implicit.
\end{rmk}

\begin{rmk}\label{rem:1}
  If $Z\in Z^p(X,n)_{0}$ is a higher cycle, it is also a cycle
  in $X\times \square^{n} $. Therefore a Green current for $Z$ could
  also mean a current $g$ on $X\times \square^{n}$ satisfying
  $dg+\delta _{Z}$ smooth. To distinguish between both cases, the
  second will be called a Green current for $Z$ on $X\times
  \square^{n}$. This ambiguity will cause more problems for Green
  forms latter, and we will use the same method to distinguish both
  meanings of Green form.  
\end{rmk}

\begin{lem}\label{lemm:2} Let $Z\in Z^p(X,n)_{0}$ be a cycle. Then an
  $\fA$-Green current for $Z$ exists. Moreover if 
  $\widetilde g_{Z}$ and $\widetilde g'_{Z}$ are  two classes of
  $\fA$-Green currents for $Z$ then the difference 
  \begin{displaymath}
    \widetilde g_{Z}-\widetilde g'_{Z}\in \widetilde \fA^{2p-n-1}(X,p),
  \end{displaymath}
  is the class of a smooth form.
\end{lem}

\begin{proof}
  Since $\caP_{\fA}$ is a morphism of complexes, the condition $\delta Z=0$
  implies $d \caP_{\fA}(Z)=0$. Since the inclusion $[\cdot]\colon \fA^{\ast}(X,p)\to
  \fA_{D}^{\ast}(X,p)$ is a quasi isomorphism, the cohomology class
  of $\caP_{\fA}(Z)$ can be represented by en element of
  $\fA^{\ast}(X,p)$. In other words,  there exist elements
  $g_{Z}\in \fA_{D}^{2p-n-1}(X,p)$ and  $\omega _{Z}\in
  \fA^{2p-n}(X,p)$ such that $\caP_{\fA}(Z)+d g_{Z}=[\omega _{Z}]$,
  showing the existence of Green currents. 

  Let now $\widetilde g_{Z}$ and $\widetilde g'_{Z}$ be two classes of
  Green currents for $Z$. Write $\omega =\omega (\widetilde g_{Z})$
  and $\omega' =\omega (\widetilde g_{Z}')$ and let
  $g_{Z}$ and $g'_{Z}$ be two representatives of the given
  classes. Then 
  \begin{displaymath}
    d(g_{Z}-g'_{Z})=[\omega -\omega']. 
  \end{displaymath}
  Thus $[\omega -\omega ']$ is exact in the complex
  $\fA_{D}^{2p-n}(X,p)$, therefore $\omega -\omega '$ is exact in the complex
  $\fA^{2p-n}(X,p)$, and there exists $u_{1}\in \fA^{2p-n-1}(X,p)$ with
  \begin{displaymath}
    d(g_{Z}-g'_{Z})=d [u_{1}].
  \end{displaymath}
  The current $g_{Z}-g'_{Z} - [u_{1}]$ is closed, so its cohomology
  class can be represented by a form, and there are elements
  $u_{2}\in \fA^{2p-n-1}(X,p)$ and $v\in \fA_{D}^{2p-n-2}(X,p)$ with
  \begin{displaymath}
    g_{Z}-g'_{Z}- [u_{1}]=[u_{2}]+ d v
  \end{displaymath}
  proving the second statement.
\end{proof}

\subsection{Green forms}
\label{sec:green-forms}

As we will see latter, the definition of Green current is well suited to
define direct images. Nevertheless, as was 
the case for the classical arithmetic Chow groups, in order to define
inverse images and products we need to represent a Green current by a
differential form with singularities only along the cycle. The
difficulty here is that a higher cycle lives in $X\times \square^{n}$. In
order to go from $X\times \square^{n}$ to $X$ we need a ``staircase'' of
differential forms.

Let $Z\in Z^{p}(X,n)_{0}$ be a pre-cycle.   
We denote by $|Z|$ the support of $Z$ and, for each $k=0,\dots,n$ we
denote by $|Z|_{k}\subset X\times \square^{k}$ the codimension $p$
subset given as
\begin{displaymath}
  |Z|_{k}=\bigcup_{i_{1},\dots,i_{n-k}} (\delta
  ^{i_{1}}_{0})^{-1}\dots (\delta ^{i_{n-k}}_{0})^{-1}|Z|.
\end{displaymath}
In particular $|Z|=|Z|_{n}$. By abuse of notation, if $W\in
\caZ^{p}_{n}$ is a closed Zariski subset, we will denote
\begin{displaymath}
  |W|_{k}=\bigcup_{i_{1},\dots,i_{n-k}} (\delta
  ^{i_{1}}_{0})^{-1}\dots (\delta ^{i_{n-k}}_{0})^{-1}W.
\end{displaymath}

For each $k=0,\dots ,n$, let $H_{k}\subset X\times
  \square^{k}$ be the normal crossing divisor
  \begin{displaymath}
    H_{k}=\bigcup_{i=1,\dots,k} X\times \delta ^{i}_{1}(\square ^{k-1}).
  \end{displaymath}
  In particular $H_{0}=0$.

We denote  by 
  \begin{equation}\label{eq:6}
    \fA_{\log}^{\ast}(X\times \square^{k}\setminus |Z|_{k},p)_{0}\subset
    \fA_{\log}^{\ast}(X\times \square^{k}\setminus |Z|_{k},p)
  \end{equation}
  the subcomplex of elements $x$
  whose restriction
  $x|_{H_{k}}=0$. As in the case of a normalized complex, there is a
  differential
  \begin{displaymath}
    \delta \colon \fA_{\log}^{r}(X\times \square^{k}\setminus
    |Z|_{k},p)_{0}\to
    \fA_{\log}^{r}(X\times \square^{k-1}\setminus
    |Z|_{k-1},p)_{0}
  \end{displaymath}
 given by
 \begin{displaymath}
   \delta x= \sum_{i=1}^{k} (-1)^{i}(\delta _{0}^{i})^{\ast}x,
 \end{displaymath}
 that turns $\fA_{\log}^{\ast}(X\times \square^{\bullet}\setminus
 |Z|_{\bullet},p)_{0}$ into a 2-iterated complex.

  Analogously, we denote by
  \begin{equation}\label{eq:18}
    \fA_{|Z|_{k}}^{\ast}(X\times \square^{k},p)_{0}\subset
    \fA_{|Z|_{k}}^{\ast}(X\times \square^{k},p),
  \end{equation}
  the subcomplex of elements $x$ whose restriction
  $x|_{H_{k}}=0$.  Again it has a differential $\delta $ turning it
  into a 2-iterated complex. 

\begin{df}\label{GF}
Given a cycle $Z \in Z^p(X,n)_0$, an \emph{$\fA$-Green form} (or just a
Green form is $\fA$ is understood) 
for $Z$ is an $n$-tuple
\begin{displaymath}
  \fg_{Z}\coloneqq (g_n, g_{n-1},\cdots , g_0)\in \bigoplus^0_{k=n}\fA
^{2p-n+k-1}_{\log}(X\times \square ^{k}\setminus |Z|_{k},p)_0,
\end{displaymath}
Such that, if $n>0$,
\begin{enumerate}
\item\label{item:20} $\delta _{Z}+d[g_n]=0$, where $\delta _{Z}$ is viewed as an
    element in $\fA_{D}$ (see example \ref{exm:4}).
\item\label{item:21} $(-1)^{n-k+1}\delta g_k+dg_{k-1}=0,\quad k=2,\cdots ,n$.
\item\label{item:22} $(-1)^{n}\delta g_1+dg_0\in
  \fA^{2p-n}(X,p)$. In words, the form  $(-1)^{n}\delta g_1+dg_0$
  extends to a smooth form on the whole $X$. 
\end{enumerate}
While, if $n=0$ the previous conditions collapse to the condition
\begin{displaymath}
 \delta _{Z}+d[g_n]\in [\fA^{2p}(X,p)]. 
\end{displaymath}
The case $n=0$ will be
implicitly understood in the sequel and will not be treated separately.
As was the case for Green currents, we will use the notation
\begin{displaymath}
 \omega(\fg _{Z})\coloneqq (-1)^{n}\delta  g_1+dg_0\in \fA^{2p-n}(X,p).
\end{displaymath}

If $Z\in Z^p(X,n)_{00}$ is a cycle in the refined normalized complex,
then a \emph{refined Green form} 
is defined as a Green form satisfying the stronger condition
\begin{equation}\label{eq:56}
  \fg_Z\in \oplus^n_{k=0}\fA
^{2p-n+k-1}_{log}(X\times \square ^k\setminus |Z|'_{k},p)_{00}.
\end{equation}
A Green form or a refined Green form will be called \emph{basic} if
$g_{n}$ is a basic Green form for $Z$ on $X\times \square^{n}$ in the
sense of Definition \ref{def:23} (see Remark \ref{rem:1}).  
\end{df}

For defining inverse images and products it is important to have
control on where the singularities of a Green form are. This is the
reason we ask in the definition of a Green form that
\begin{displaymath}
g_{k}\in  \fA
^{2p-n+k-1}_{\log}(X\times \square ^{n}\setminus |Z|_{n},p)_0,
\end{displaymath}
but in intermediate steps, it is useful to relax this condition.
\begin{df}\label{def:4}
  Given a cycle $Z \in Z^p(X,n)_0$, a  \emph{loose $\fA$-Green form} (or just a
loose Green form is $\fA$ is understood) 
for $Z$ is an $n$-tuple
\begin{displaymath}
  \fg_{Z}\coloneqq (g_n, g_{n-1},\cdots , g_0)\in \bigoplus^0_{k=n}\fA
^{2p-n+k-1}_{\log}(X\times \square ^{k}\setminus \caZ^{p},p)_0,
\end{displaymath}
satisfying conditions \eqref{item:20}, \eqref{item:21} and
\eqref{item:22}. 
\end{df}

Let $\fg_{Z}=(g_n,\dots , g_0)$ be a loose Green form for $Z$. By the
definition of $\fA
^{2p-n+k-1}_{\log}(X\times \square ^{k}\setminus \caZ^{p},p)_0$ as
direct limit, for each $k\ge 0 $, we can choose a Zariski closed
subset $W_{k}\in \caZ^{p}_{k}$ such that
\begin{displaymath}
  g_{k}\in \fA
^{2p-n+k-1}_{\log}(X\times \square ^{k}\setminus W_{k},p)_0.
\end{displaymath}
Adding to $W_{n}$ sets of the form $\pi ^{\ast}W_{k}$, where $\pi
\colon X\times \square^{n}\to X\times \square ^{k}$ is the projection,
we can find a subset $W\in \caZ^{p}_{n}$ such that $W_{k}\subset
|W|_{k}$.
\begin{df}\label{def:5}
  We say that a loose Green form for $Z$ has \emph{singular support}
  on a subset $W\in \caZ^{p}_{n}$ with $|Z|\subset W$ if
\begin{displaymath}
  g_{k}\in \fA
^{2p-n+k-1}_{\log}(X\times \square ^{k}\setminus |W|_{k},p)_0,\  k=1,\dots,n.
\end{displaymath}
\end{df}

\begin{prop}\label{prop:15}
  Let $Z \in Z^p(X,n)_0$ be a cycle (that is, $\delta Z=0$). Then a
  loose Green form for $Z$ exists.
\end{prop}
\begin{proof}
  Consider the class $\cl(Z)\in \caH^{0}(X,n)$. Since $Z$ is a cycle, by Proposition
  \ref{difcubs}, there is an element
  \begin{displaymath}
    \beta = ((\omega_n,g_n),\dots,(\omega_0,g_0))\in \tfA^{2p-n}_{\A,\mathcal{Z}^p}(X,p)_0
  \end{displaymath}
  with $d_{s}\beta =0$ and
  \begin{equation}\label{eq:19}
    \cl(\omega_n,g_n)=\gamma _{1}'(\beta )=\cl(Z).
  \end{equation}
  By \cite[Theorem 4.4~(1)]{Burgos:Gftp} (see also \cite[Theorem
  5.9]{Burgos:CDB}), the equation \eqref{eq:19} implies that
  \begin{displaymath}
    \delta _{Y}+d [g_{n}] = [\omega _{n}].
  \end{displaymath}
  Consider now the element
  \begin{displaymath}
    \rho (\beta ) = (\omega_n,\dots,\omega _{0})\in \tfA^{2p-n}_{\A}(X,p)_0.
  \end{displaymath}
  Since $\beta $ is a cycle, $\rho (\beta )$ is a cycle. By
  Proposition \ref{affine3}~\eqref{item:23} there is an element
  \begin{displaymath}
    (\eta _{n+1},\cdots , \eta _0)\in \fA^{2p-n-1}_{\mathbb{A}}(X,p)_0
  \end{displaymath}
  such that
  \begin{displaymath}
    d_{s}(\eta _{n+1},\cdots , \eta _0)=(\omega_n,\dots,\omega _{0}) - (0,\dots,0,\omega ).
  \end{displaymath}
In particular $d\eta_{n+1}=0$.
 Since the complex $\tfA ^{\ast,-n }_{\A}(X,p)_0$ is
 acyclic for all $n>0$ (Proposition \ref{affine3}) we can find an element
 $\tilde{\eta }_{n+1}\in \fA ^{2p-1,-n-1 }_{\A}(X,p)_0$
 such that $d\tilde{\eta }_{n+1}=\eta 
 _{n+1}$. Then, one can verify that
 \begin{displaymath}
   \beta '\coloneqq \beta -d_{s}\left((\eta _{n+1},\tilde{\eta }_{n+1}), (\eta _n,0),\cdots ,
     (\eta _0,0)\right)=((0,g'_{n}),\dots,(0,g'_{1}),(\omega ,g'_{0})).
 \end{displaymath}
 Since $\tilde{\eta }_{n+1}$ is a smooth form,
 $\cl(0,g_{n}')=\cl(\omega _{n},g_{n})=\cl(Z)$. Therefore
 \begin{displaymath}
   \delta _{Z}+d[g'_{n}]=0.
 \end{displaymath}
 This equation, plus the condition that $d_{s}\beta '=0$ implies that
 $(g'_{n},\dots,g'_{0})$ is a loose Green form for $Z$. 
\end{proof}

\begin{prop}\label{prop:16}
  Let $Z\in Z^{p}(X,n)_{0}$ be a cycle and $\fg=(g _n,\dots,g_0)$ a
  loose Green form for $Z$. Assume that $\fg$ has singular support in
  $W\in \caZ^{p}_{n}$.  
    Then there are elements $u_{n},\dots,u_{0}$,with
    \begin{displaymath}
      u_{k}\in \fA^{2p-n+k-2}_{\log}(X\times \square^{k}\setminus
      |W|_{k},p)_{0} 
    \end{displaymath}
    such that
    \begin{align}
      g_{n}'&\coloneqq g_{n}+d u_{n}\in \fA^{2p-1}_{\log}(X\times \square^{n}\setminus
      |Z|_{n},p)_{0},\label{eq:39}\\
      g_{k}'&\coloneqq g_{k}+(-1)^{n-k-1}\delta u_{k+1}+d u_{k}
      \in \fA^{2p-n+k-1}_{\log}(X\times \square^{k}\setminus
      |Z|_{k},p)_{0}, \label{eq:40}
    \end{align}
    for $k=0,\dots,n-1.$ In consequence $\fg'=(g' _n,\dots,g'_0)$ is a
    Green form for $Z$.
\end{prop}
\begin{proof} We first note that the cohomology
  of the complex $\fA_{|Z|_{k}}^{\ast}(X\times \square^{k},p)_{0}$,
  fits in a long exact sequence
  \begin{displaymath}
    \to     H^{r}(\fA_{|Z|_{k}}^{\ast}(X\times \square^{k},p)_{0})
    \to H^{r}_{\DB,|Z|_{k}}(X\times \square^{k},p) \to
    H^{r}_{\DB,|Z|_{k}\cap H_{k}}(H_{k},p) \to \dots
  \end{displaymath}
  We now prove the existence of $u_{n}$. Since $Z\in Z^{p}(X,n)_{0}$
  the cohomology class of $Z$ with supports
  \begin{displaymath}
    [Z]\in H^{2p}_{\DB,|Z|}(X\times \square^{n},p)
  \end{displaymath}
  satisfies $[Z]\mid_{H_{n}}=0$. Therefore it lifts to a class in
  \begin{displaymath}
    H^{2p}(\fA_{|Z|}^{\ast}(X\times \square^{n},p)_{0}).
  \end{displaymath}
  This lifting is unique because $H^{2p-1}_{\fA,|Z|\cap
    H_{n}}(H_{n},p)=0$ by semipurity and the hypothesis that $Z$
  intersects properly all the faces.

  The pair $(0,g_{n})$ represents the class of $Z$ in the cohomology of
  the complex  $\fA_{|W|_{n}}^{\ast}(X\times
  \square^{n},p)_{0}$. Therefore there exists a pair
  \begin{displaymath}
    (\eta_{n},u_{n})\in \fA_{|W|_{n}}^{2p-1}(X\times
  \square^{n},p)_{0} 
  \end{displaymath}
  such that
  \begin{displaymath}
    (0,g_{n})-d(\eta_{n},u_{n})\in \fA_{|Z|}^{2p}(X\times
    \square^{n},p)_{0} .
  \end{displaymath}
  In particular equation \eqref{eq:39} is satisfied. Note that this
  proof work also when $n=0$, where $H_{n}=\emptyset$. 

  Assume now that $n>0$, that $0\le k \le n-1$ and that  the forms
  $u_{n},\dots, 
  u_{k+1}$ have been 
  found. We have to show the existence of $u_{k}$. Using the
  conditions satisfied by the $g_{k}$ and the $u_{k}$ one readily
  checks that
  \begin{equation}\label{eq:42}
    d(g_{k}+(-1)^{n-k+1}\delta u_{k+1})\in \fA^{2p-n+k}_{\log}(X\times
    \square^{k}\setminus 
      |Z|_{k},p)_{0}
  \end{equation}
  Using semi-purity of Deligne-Beilinson  cohomology and Proposition
  \ref{affine3} one can prove that
  the map
  \begin{displaymath}
    H^{r}(\fA^{\ast}_{\log}(X\times \square^{k}\setminus 
      |Z|_{k},p)_{0})\to
H^{r}(\fA^{\ast}_{\log}(X\times \square^{k}\setminus 
      |W|_{k},p)_{0})
  \end{displaymath}
  is injective for $r=2p-1$ and an isomorphism for $r<2p-1$. In
  particular, since $k\ne n-1$, it is injective for $r=2p-n+k$ and surjective for
  $r=2p-n+k-1$.  Thus the existence of $u_{k}$ follows from Lemma
  \ref{lemm:14} below. The last statement is clear form the definition
  of the $g'_{k}$.
\end{proof}

\begin{lem}\label{lemm:14}
  Let $A^{\ast}\to B^{\ast}$ be an injective morphism of complexes
  such that the map $H^{r}(A^{\ast})\to H^{r}(B^{\ast})$ is injective
  for $r=n$ and surjective for $r=n-1$. Let $x\in B^{n-1}$ such that
  $dx\in A^{n}$. Then there is an element $y\in B^{n-2}$ such that
  \begin{displaymath}
    x+dy\in A^{n-1}.
  \end{displaymath}
\end{lem}
\begin{proof}
  The element $dx\in A^{n}$ is closed in the complex $A^{\ast} $ and
  exact in the complex $B^{\ast}$. By the injectivity condition it is
  exact in $A^{\ast}$. So there is $z\in A^{n-1}$ with $dx=dz$. The
  element $x-z\in B^{n-1}$ is closed. By the surjectivity condition
  its cohomology class lies in the image of the cohomology of
  $A^{\ast}$. Thus there is an element $y\in B^{n-2}$ such that
  \begin{displaymath}
    x-z+dy\in A^{n-1}
  \end{displaymath}
proving the lemma.
\end{proof}

\begin{prop}\label{prop:18}
  Let $Z\in Z^{p}(X,n)_{0}$ be a cycle and $\fg=(g _n,\dots,g_0)$ a
  Green form for $Z$.   
    Then there is an element 
    \begin{displaymath}
      u\in \fA^{2p-2}_{\log}(X\times \square^{n}\setminus
      |Z|,p)_{0} 
    \end{displaymath}
    such that
    \begin{math}
      (g_{n}+d u,g_{n-1}+\delta u,g_{n-2},\dots,g_{0})
    \end{math}
    is a basic Green form for $Z$.  
\end{prop}
\begin{proof}
  By simplicity we assume that $Z$ is a prime cycle. The general case
  follows from this one by linearity.
  By \cite[Lemma 4.6]{Burgos:Gftp} a basic Green form $g'$  for $\overline
  Z$ on $X\times (\P^{1})^{n}$ exists. By \cite[Proposition
  5.5~(1)]{Burgos:CDB} there exist $u\in \fA^{2p-2}_{\log}(X\times
  \square^{n}\setminus |Z|,p)$ and $\alpha \in \fA^{2p-1}_{\log}(X\times
  \square^{n},p)$ such that, on $X\times \square^{n}$,
  \begin{displaymath}
    g_{n}+d u=g'+\alpha. 
  \end{displaymath}
  Since $g'+\alpha $ is a basic Green form for $Z$ on $X\times
  \square^{n}$, the proposition is proved.
\end{proof}

\begin{thm}\label{thm:11}
    Let $Z \in Z^p(X,n)_0$ be a cycle (that is, $\delta Z=0$). Then a
    basic Green form for $Z$ exists. 
\end{thm}
\begin{proof}
  Follows from propositions \ref{prop:15}, \ref{prop:16} and \ref{prop:18}.
\end{proof}

Given a
Green form $\fg_{Z}$, we next want to extract a Green 
current corresponding to it. 

\begin{prop}\label{GreenCurrent}
  Let $Z \in Z^p(X,n)_0$ be a cycle and
  $\fg_{Z}=(g_{n},\dots,g_{0})$ a loose Green form for $Z$.
  Then the current
    \begin{displaymath}
      [\fg_{Z}]\coloneqq \sum _{i=0}^n(p_{1})_{\ast}[g_{i}\cdot
      W_{i}]
    \end{displaymath}
    is a Green current for $Z$. 
\end{prop}
\begin{proof} For simplicity we
  compute on the Thom-Whitney simple. Strictly speaking, Lemma
  \ref{lemm:13} does not apply here because $\fg_{Z}$ is not a basic
  Green form. Nevertheless \cite[Proposition 3.3]{Burgos:Gftp} still
  implies that all the forms $g_{i}\cdot W_{i}$ are locally
  integrable.  Let $W$ be the singular support of $\fg_{Z}$. Arguing
  as 
  in Proposition \ref{prop:18}, there is an 
  element 
    \begin{displaymath}
      u\in \fA^{2p-2}_{\log}(X\times \square^{n}\setminus
      |W|_{n},p)_{0} 
    \end{displaymath}
    such that $g_{n}'\coloneqq g_{n}+d u $ is a basic Green form for
    $Z$ on $X\times \square^{n}$. Write $g_{n-1}'=g_{n-1}+\delta u$
    and $g'_{k}=g_{k}$ for $k\le 2$. Since, using Lemma \ref{lemm:13}
    and~\eqref{eq:17},
    \begin{displaymath}
      d(p_{1})_{\ast}[u\cdot W_{n}]=(p_{1})_{\ast}[d u\cdot
      W_{n}]+(p_{1})_{\ast}[\delta u\cdot W_{n-1}], 
    \end{displaymath}
    it is enough to prove the result for
    $\fg'_{Z}=(g'_{n},\dots,g'_{0})$. Thus from now on we assume that
    $g_{n}$ is a basic Green form for $Z$ on $X\times \square^{n}$. 

    By Lemma \ref{lemm:13} again we
    have the residue relation
    \begin{displaymath}
      d[g_n\cdot W_n]=-\delta_Z\cdot W_n-[\delta g_n\cdot W_{n-1}],
    \end{displaymath}
    and for $r<n$, the relations
    \begin{displaymath}
      d[g_r\cdot W_r]-[dg_r\cdot W_r]=(-1)^{n-r+1}[\delta g_r\cdot W_{r-1}].
    \end{displaymath}
    Using the above residue relations, and the step relations
    defining $\fg_Z$, we compute 
      \begin{align*}
    d[\fg_Z]&=\sum _{i=0}^{n}d(p_{1})_{\ast}[g_{i} \cdot W_{i}]\\
    &= \sum _{i=0}^{n}(p_{1})_{\ast}d[g_{i}\cdot W_{i}]\\
%     &=+ \sum
%       _{i=0}^{n}(-1)^{2p-n+i-1}\sum_{k=1}^{i}\sum_{j=0,1}(-1)^{i+j} 
%      (p_{1})_{\ast}[(\delta
%       _{j}^{k})^{\ast}g_{i} \cdot W_{i-1}]\\
       &=-(p_1)_{\ast}(\delta_Z\cdot
         W_n)+\underbrace{\sum^n_{i=1}(-1)^{n-i+1}(p_1)_{\ast}[\delta
         g_i\cdot W_{i-1}]}_{u}
         \\ 
       &\phantom{AA}+\underbrace{\sum^{n-1}_{i=0}(-1)^{n-i-1}(p_1)_{\ast}[\delta
         g_{i+1}\cdot W_i]}_{v}+\omega(\fg_{Z})
         \\ 
    &= \omega (\fg_{Z})-(p_{1})_{\ast}(\delta _{Z}\cdot W_{n})\\
        &=\omega (\fg_{Z})-\caP(Z),
      \end{align*}
      since $u$ and $v$ have opposite signs, they cancel. This proves
      the statement. 
    \end{proof}
    
\begin{prop}\label{prop:8}
  Let $Z \in Z^p(X,n)_0$ be a cycle, and
  $\widetilde g_{Z}$ a class of $\fA$-Green currents for $Z$. Then there
  exists a basic $\fA$-Green form 
  $\fg_{Z}$  such that
  \begin{displaymath}
    [\fg_{Z}]\in \widetilde g_{Z}.
  \end{displaymath}
\end{prop}
\begin{proof}
  By Theorem \ref{thm:11}, a basic Green form $\fg'_{Z}=(g_{n},\dots,g_{0})$
  for $Z$ exists. By
  Proposition \ref{GreenCurrent} the current $[\fg'_{Z}]$ is a Green
  current for $Z$. By Lemma \ref{lemm:2}, there is a current $u\in
  \fA^{2p-n-2 }_{D}(X,p)$ and a 
  form $v\in \fA^{2p-n-1 }(X,p)$ such that
  \begin{displaymath}
    g_{Z}- [\fg'_{Z}] = v + du.
  \end{displaymath}
  Writing $\fg_{Z}=(g_{n},\dots,g_{0}+v)$ we obtain the result.
\end{proof}

We next discuss when two Green forms give rise to the same Green
current.

\begin{prop}\label{prop:17} Let $Z\in Z^{p}(X,n)_{0}$ be a cycle and
  let $\fg=(g_{n},\dots,g_{0})$ and $\fg'=(g'_{n},\dots,g'_{0})$ be
  two loose Green forms for the cycle $Z$ with singular support
  contained in a subset $W\in \caZ^{p}_{n}$. Then
  \begin{equation}\label{eq:20}
    [\fg]^{\sim}=[\fg']^{\sim}
  \end{equation}
  if and only if there are elements $u_{n},\dots,u_{0}$,with
  \begin{displaymath}
    u_{k}\in \fA^{2p-n+k-2}_{\log}(X\times \square^{k}\setminus
    |W|_{k},p)_{0} 
  \end{displaymath}
  such that
  \begin{align*}
    g_{n}'-g_{n}&=d u_{n}\\
    g_{k}'-g_{k}&= (-1)^{n-k-1}\delta u_{k+1}+d u_{k},\quad
                  k=0,\dots,n-1.\\
  \end{align*}
  If $\fg$ and $\fg'$ are Green forms for $Z$ then we can choose
  \begin{displaymath}
    u_{k}\in \fA^{2p-n+k-2}_{\log}(X\times \square^{k}\setminus
    |Z|_{k},p)_{0}.
  \end{displaymath}
\end{prop}
\begin{proof}
  Assume that such elements exist. Write
  \begin{displaymath}
    u=\sum _{i=0}^n(p_{1})_{\ast}[u_{i}\cdot
      W_{i,\fA}].
    \end{displaymath}
    A computation similar to the one in the proof of Proposition
    \ref{GreenCurrent} shows that 
    \begin{displaymath}
      [\fg']-[\fg]=d u
    \end{displaymath}
    so $[\fg']$ and $[\fg]$ represent the same class of Green
    currents.

    Conversely, assume that equation \eqref{eq:20} is satisfied. Write
    $\fg''=\fg'-\fg$. Then $\fg''$ is a loose Green form  for the
    cycle zero ans singulr support contained in $W$. Moreover $\omega
    (\fg'')=0$. Applying Proposition 
    \ref{prop:16} to $\fg''$ we find elements $u'_{n},\dots ,u'_{0}$
    such that
    \begin{align*}
      \alpha _{n}&\coloneqq g_{n}-g'_{n}+d u'_{n}\in
                   \fA^{2p-1}_{\log}(X\times \square^{n},p)_{0},\\
      \alpha _{k}&\coloneqq g_{k}-g'_{k}+(-1)^{n-k-1}\delta u'_{k+1}+d u'_{k}
                   \in \fA^{2p-n+k-1}_{\log}(X\times \square^{k},p)_{0},
    \end{align*}
    for $k=0,\dots,n-1.$
    Therefore
    \begin{displaymath}
      \alpha \coloneqq (0,\alpha _{n},\dots,\alpha _{0})\in \tfA_{\A}^{2p-n-1}(X,p)_{0}.
    \end{displaymath}
    The properties of Green form imply that $\alpha $ is closed.
    By Proposition \ref{affine3}~\eqref{item:24} and induction over
    $n-k$, we can find elements 
    $u''_{n},\dots ,u''_{1}$ such that 
    \begin{align*}
      0&= \alpha _{n} +d u''_{n},\\
      0&=\alpha _{k}+(-1)^{n-k-1}\delta u''_{k+1}+d u''_{k},\quad k=1,\dots,n-1.
    \end{align*}
    So $\alpha $ is cohomologous to the element
    \begin{displaymath}
      (0,\dots,0,\alpha_{0}+(-1)^{n-1}\delta u''_{1})
    \end{displaymath}
    Now equation \eqref{eq:20} implies that $\alpha
    _{0}+(-1)^{n-1}\delta u''_{1}$ is a boundary in the complex
    $\fA^{\ast}(X,p)$, so there is a $u''_{0}$ such that
    \begin{displaymath}
      \alpha _{0}+(-1)^{n-1}\delta u''_{1} + du''_{0}=0.
    \end{displaymath}
    Writing $u_{k}=u_{k}'+u_{k}''$ we obtain the elements needed to
    prove the converse statement.

    The last statement is the particular case when $W=|Z|$.
\end{proof}

\subsection{Functorial properties of Green currents and Green forms} 
\label{sec:funct-prop-green}
We next study the functorial properties of Green currents and Green
forms. For the direct image it is better to work with Green currents,
while for inverse images it is necessary to work with Green forms.

As in the classical case, for the direct image to exists at the level
of Green currents, we need a smooth proper map.

\begin{prop}\label{prop:14} Let $f$ be a map between smooth
  projective complex varieties. Let $Z\in Z^{p}(X,n)$ be a cycle and
  $g_{Z}$ a Green current for $Z$. Then, if $f$ is smooth of $\omega
  (g_{Z})=0$, then the current $f_{\ast}g_{Z}$ is
  a Green current for the cycle $f_{\ast}Z$.
\end{prop}
\begin{proof}
  Assume that $f$ is smooth,
  then the compatibility of direct images with the differential of
  currents,  Proposition \ref{prop:3} and the fact that the direct image of
  a smooth form by a smooth map is again a smooth from we obtain
  \begin{displaymath}
    d f_{\ast }g_{Z}=f_{\ast}d g_{Z}=f_{\ast}[\omega
    (g_{Z})]-f_{\ast}\caP(Z)= [f_{\ast}\omega
    (g_{Z})]-\caP(f_{\ast}Z)
  \end{displaymath}
  proving the result in this case.
  If $f$ is not smooth but $\omega (g_{Z})=0$, then
  \begin{displaymath}
    d f_{\ast }g_{Z}=f_{\ast}d g_{Z}=-f_{\ast}\caP(Z)=-\caP(f_{\ast}Z)
  \end{displaymath}
  proving the result in the remaining case.
\end{proof}

We next study the inverse image. The inverse image is only defined for
classes of Green currents and not for Green currents. Let $f\colon
X\to Y$ be a morphism of smooth projective varieties 
  and $Z\in Z_{f}^{p}(X,n)_{0}$ be a cycle that is in good position
  with respect to $f$. Let $\widetilde g_{Z}$ be a class of Green
  currents for $Z$. Choose a basic Green form  $\fg_{Z}=(g_{n},\dots,g_{0})$
  for $Z$ such that $[\fg_{Z}]\in \widetilde g_{Z}$. We write
  \begin{displaymath}
    f^{\ast}\fg \coloneqq ((f\times \Id_{\square^{n}})^{\ast}g_{n},\dots, f^{\ast}g_{0}).
  \end{displaymath}
  For shorthand write $g_{k}'=(f\times
  \Id_{\square^{k}})^{\ast}g_{k}$.  By the functoriality of all the
  operations involved and Lemma \ref{lemm:13}, it is easy to check
  that $f^{\ast}\fg_{Z} $ 
  satisfies the conditions 
  \begin{enumerate}
  \item $\delta _{f^{\ast }Z}+d[g'_n]=0$, 
  \item $(-1)^{n-k+1}\delta g'_k+dg'_{k-1}=0,\quad k=2,\cdots ,n$.
  \item $(-1)^{n}\delta g'_1+dg'_0=f^{\ast}(\omega
    (\widetilde g_{Z})).$
  \end{enumerate}
  Since $Z$ is in good position with respect to $f$, $f^{-1}|Z|_{k}$
  has codimension $p$ and  $f^{\ast}\fg_{Z}$ a loose Green form for
  $f^{\ast}Z$. In general it is not a Green form because
  $|f^{\ast}Z|_{k}$ may not agree with $f^{-1}|Z|_{k}$.
  
\begin{prop} \label{prop:20} With the previous notation, the current
  $[f^{\ast}\fg_{Z}]$ is a Green current for 
  $f^{\ast}Z$. Moreover, $\fg'_{Z}$ is another choice of Green form
  for $Z$ with $[\fg_{Z}]^{\sim}=[\fg'_{Z}]^{\sim}$, then
   \begin{displaymath}
     [f^{\ast}\fg_{Z}]^{\sim}=[f^{\ast}\fg'_{Z}]^{\sim}.
   \end{displaymath}
   That is,
   the class of $[f^{\ast}\fg_{Z}]$ does not
   depend on the choice of $\fg_{Z}$. 
\end{prop}
\begin{proof}
   The first statement is proved using the same argument as in the
   proof of  Proposition \ref{GreenCurrent}. To prove the second
   statement Let $\fg$ and $\fg'$ be two Green forms for $Z$
   representing the same class of green currents. By Proposition
   \ref{prop:17}, there are elements $u_{k}\in
   \fA^{2p-n+k-2}_{\log}(X\times \square^{k}\setminus
   |Z|_{k},p)_{0}$, $k=0,\dots,n$
   satisfying the conditions in the proposition. Then the elements
   $f^{\ast}u_{k}$ satisfy the same conditions for the Green forms
   $f^{\ast}\fg_{Z}$ and $f^{\ast}\fg'_{Z}$, so they represent the
   same class of Green currents. 
\end{proof}
From the last proposition, and for cycles $Z\in Z^p_f(Y,n)_0$ we can define an inverse image for a class of Green currents as
$$f^\ast \widetilde{g}_Z:=[f^\ast \fg_Z]^\sim,$$
for any choice of green forms of logarithmic type $\fg_Z$ of $Z$, such that $[\fg_Z]\in \widetilde{g}_Z$.

\subsection{The $\ast$-product}
\label{sec:ast-product}
In this section we will give the  definition of the product of Green
currents. Since we want a graded-commutative and associative product
we will work in the Thom-Whitney complex. So in this section
$\fA=\fDTW$.  

 We
start by giving the definition of the $\ast$-product of Green forms
and Green currents.

% \red{For $Z\in Z^p(X,n)_0$ and $W\in Z^q(X,m)_0$, intersecting properly as per Definition \ref{def:19}, and a choice of Green form $\fg_W=\{g'_m,\cdots, g'_0\}$ for $W$, we define the following: Let $\overline{p^\ast_{12}(Z)}$ be the Zariski closure of $p^\ast_{12}(Z)$ on $X\times (\P^1)^{n+j}$, for each $0\leq j\leq m$, and let $\eta_j\colon \widetilde{p^\ast_{12}(Z)}\rightarrow X\times (\P^1)^{n+j}$ be the composition of the resolution of singularities of $\overline{p^\ast_{12}(Z)}$ with the inclusion of $\overline{p^\ast_{12}(Z)}$ in $X\times (\P^1)^{n+j}$. Then we define
% $$\delta_Z\cdot W_n\cdot g'_j\cdot W_j:=\eta_{j,\ast}[\eta^\ast_j(W_n\cdot g'_j\cdot W_j)].$$
% Now we give the following}
\begin{df}\label{def:16}
  Let $Z\in Z^{p}(X,n)_{0}$ and $W\in Z^{q}(X,m)_{0}$
  be cycles in the normalized complex, intersecting properly as
  in Definition \ref{def:19}. Let $g_{Z}$ be a Green current for the
  cycle $Z$ and let $\fg_{W}$ be a basic Green form for the cycle $W$ with
  components
  \begin{displaymath}
    \fg_{W}=(g'_{m},\dots,g'_{0}).
  \end{displaymath}
  Let $\omega _{Z}=\omega (g_{Z})$ and $\omega _{W}=\omega
  (\fg_{W})$. 
  We denote formally the product of currents
  \begin{displaymath}
    \mathcal{P}(Z)\cdot [\fg_W]\coloneqq
    \sum^m_{j=0}(p_{1})_*\left(\delta_Z\cdot W_n\cdot g'_j\cdot W_j\right),  
  \end{displaymath}
  where $\delta_Z\cdot W_n\cdot g'_j\cdot W_j$ is seen as a current in
  $X\times (\P^{1})^{n+m}$ as in Definition \ref{def:22}  and 
  $p_{1}\colon X\times (\P^{1}) ^{n+m}\to X$ is the projection. Then
  the \emph{$\ast$-product} of the Green current $g_Z$ and the Green
  form $\fg_W$ is the current defined as
  \begin{displaymath}
    g_Z\ast \fg_W\coloneqq (-1)^n\caP(Z)\cdot [\fg_W]+g_Z\cdot
    \omega_W
  \end{displaymath}
  \end{df}

\begin{prop}\label{prop:8-33} Let $Z$, $W$, $g_{Z}$ and $\fg_{W}$ be
  as in Definition \ref{def:16}. Then the current $g_Z\ast \fg_W$ satisfies
  \begin{displaymath}
    \mathcal{P}(Z\cdot W)+d(g_Z\ast \fg_W)=[\omega_Z\cdot \omega_W].
  \end{displaymath}
  Hence is a Green current for the product $Z\cdot W$.
\end{prop}
\begin{proof}
We compute, using Lemma \ref{lemm:13}, the fact that $Z$ is a cycle,
the properties defining Green forms and the
properties of the Wang forms \eqref{eq:17} 
\begin{align*}
  d(g_Z\ast \fg_W)
  &=(-1)^n\sum^m_{j=0}(p_{1})_*\left(d(\delta_Z\cdot W_n\cdot g'_j\cdot
    W_j)\right)+ d(g_Z\cdot \omega_W)\\
  &=(-1)^{n}\sum^m_{j=0}(p_{1})_*\left(\delta_{\delta Z}\cdot W_{n-1}\cdot g'_j\cdot
    W_j\right) -(p_{1})_*\left(\delta _{Z\cdot W}\cdot
    W_{n+m}\right)\\
  &\phantom{AA}+\sum_{j=0}^{m-1}(p_{1})_*\left(\delta_{Z}\cdot W_{n}\cdot d g'_j\cdot
    W_j\right)\\
    &\phantom{AA}+
      \sum_{j=1}^{m}(-1)^{m-j-1}(p_{1})_*\left(\delta_{Z}\cdot
      W_{n}\cdot \delta  g'_{j}\cdot 
      W_{j-1}\right)  \\
  &\phantom{AA}+[\omega_Z\cdot \omega_W]-\caP(Z)\cdot
    \omega_W\\
  &=-\caP(Z\cdot W)+(p_{1})_{\ast}(\delta _{Z}\cdot W_{n}\cdot \omega
    _{W})
    +[\omega_Z\cdot \omega_W]-\mathcal{P}(Z)\cdot
    \omega_W\\
  &=-\caP(Z\cdot W)
    +[\omega_Z\cdot \omega_W].
\end{align*}
\end{proof}

\begin{prop}\label{prop:007} If $g'_{Z}$ is another Green current for $Z$ and
  $\fg'_{W}$ is another Green form for $W$ with
  \begin{displaymath}
    \widetilde g_{Z}=\widetilde g'_{Z}, \quad
    [\fg_{W}]^{\sim}=[\fg'_{W}]^{\sim}, 
  \end{displaymath}
then
\begin{displaymath}
  (g_{Z}\ast \fg_{W})^{\sim}=(g'_{Z}\ast \fg'_{W})^{\sim}.
\end{displaymath}
\end{prop}
\begin{proof}
  Write $g_{Z}-g'_{Z}=dc$. Then
  \begin{displaymath}
    g_{Z}\ast \fg_{W}-g'_{Z}\ast \fg_{W}= (dc)\cdot \omega
    (\fg_{W})=d(c\cdot \omega (\fg_{W})). 
  \end{displaymath}
  Thus $(g_{Z}\ast \fg_{W})^{\sim}=(g'_{Z}\ast \fg_{W})^{\sim}$.
   Applying Proposition \ref{prop:17} to $\fg'_{W}$ and $\fg'_{W}$ we
   obtain a tuple  $u=u_{m},\dots, u_0$. Define
   \begin{displaymath}
     \caP(Z)\cdot [u]=\sum^m_{j=0}(p)_*\left(\delta_Z\cdot W_n\cdot
       u_j\cdot W_j\right) 
   \end{displaymath}
   then, using Lemma \ref{lemm:13}, one can verify that
   \begin{displaymath}
     g'_{Z}\ast \fg'_{W}-g'_{Z}\ast \fg_{W}=d (\caP(Z)\cdot [u]).
   \end{displaymath}
   Thus $(g'_{Z}\ast \fg'_{W})^{\sim}=(g'_{Z}\ast \fg_{W})^{\sim}$.
\end{proof}

In view of the previous proposition the following definition makes sense.

\begin{df}\label{def:6.21}
    Let $Z\in Z^{p}(X,n)_{0}$ and $W\in Z^{q}(X,m)_{0}$
  be cycles intersecting properly and $\widetilde g_{Z}$ and
  $\widetilde g_{W}$ classes of Green currents for $Z$ and $W$. Choose
  any representative $g_{Z}$ of $\widetilde g_{Z}$ and a Green form
  $\fg_{W}$ with  $[\fg_{W}]\in \widetilde g_{W}$. Then the
  $\ast$-product of $\widetilde g_{Z}$ and $\widetilde g_{W}$ is
  defined as
  \begin{displaymath}
    \widetilde g_{Z}\ast \widetilde g_{W}=
    \left (  (-1)^n\caP(Z)\cdot [\fg_W]+g_Z\cdot
    \omega(\fg_{W})\right)^{\sim}.
  \end{displaymath}
\end{df}

\begin{thm}\label{thm:12} Let $Z\in Z^{p}(X,n)_{0}$, $W\in
  Z^{q}(X,m)_{0}$  and $T\in Z^{r}(X,\ell)_{0}$ cycles such that $W$
  intersects properly $Z$ and $T$, $Z$ intersect properly $W\cdot T$
  and $T$ intersects properly $Z\cdot W$. Let $\widetilde g_{Z}$,
  $\widetilde g_{W}$ and $\widetilde g_{T}$ be classes of Green
  currents for $Z$, $W$ and $T$ respectively. Then
  \begin{enumerate}
  \item \label{item:28} $\widetilde g_{Z}\ast \widetilde g_{W}=
    (-1)^{nm}\widetilde g_{W}\ast \widetilde g_{Z}.$
  \item \label{item:29} $\widetilde g_{Z}\ast (\widetilde g_{W}\ast
    \widetilde g_{T})=(\widetilde g_{Z}\ast \widetilde g_{W})\ast
    \widetilde g_{T}.$ 
  \end{enumerate}
\end{thm}
\begin{proof}
  We start by proving the commutativity. Choose $\fg_{Z}=(g_{n},\dots,g_{0})$ and
  $\fg_{W}=(g'_{m},\dots,g'_{0})$ basic Green forms for $Z$ and $W$ with
  $g_Z\coloneqq [\fg_{Z}]\in \widetilde 
  g_{Z}$ and $g_{W}\coloneqq [\fg_{W}]\in \widetilde g_{W}$.  We define 
  \begin{displaymath}
    [\fg_Z]\cdot \mathcal{P}(W) \coloneqq
    \sum^n_{i=0}(p_{1})_*\left(g_i \cdot W_i\cdot \delta_W\cdot W_m\right),  
  \end{displaymath}
  and
  \begin{displaymath}
    \fg_Z\ast' g_W\coloneqq [\fg_Z]\cdot \caP(W)+(-1)^n\omega (g_Z)\cdot
    g_W.
  \end{displaymath}
  We also define the product of currents
  \begin{displaymath}
    [\fg_{Z}]\cdot [\fg_{W}]=
    \sum^n_{i=0}\sum _{j=0}^{m}(p_{1})_*\left(g_i \cdot
      W_i\cdot g'_{j}\cdot W_j\right)
\end{displaymath}
   
  Then, using Lemma \ref{lemm:13},
  \begin{multline*}
    d([\fg_{Z}]\cdot [\fg_{W}])=\\
    \omega _{Z}\cdot g_{W}-\caP(Z)\cdot [\fg_{W}]+(-1)^{n-1}g_{Z}\cdot
    \omega _{Z})+(-1)^{n} [\fg_{Z}]\cdot \caP(W)).
  \end{multline*}
  Therefore
  \begin{equation}\label{eq:36}
    (g_{Z}\ast \fg_{W})^{\sim}=(\fg_{Z}\ast' g_{W})^{\sim}.
  \end{equation}
  So we are left to compare $\fg_{Z}\ast' g_{W}$ and
  in $(-1)^{nm}g_{W}\ast \fg_{Z}$. But inspecting all 
  the integrals that appear in both terms, one realizes that they  are
  equal except for a reordering of 
  the variables of integration. Since the value of an integral does
  not depend on the name of the integration variables we deduce that
  \begin{displaymath}
    (g_{Z}\ast \fg_{W})^{\sim}=(-1)^{nm}(g_{W}\ast \fg_{Z})^{\sim}.
  \end{displaymath}

  We next prove the associativity. We choose basic Green forms for the
  cycles as follows
  \begin{alignat*}{3}
    \fg_{Z}&=(g_{n},\dots,g_{0}),\quad&g_{Z}&\coloneqq [\fg_{Z}]\in
    \widetilde g_{Z},\quad &\omega _{Z }&\coloneqq \omega (\widetilde g_{Z}),\\
    \fg_{W}&=(g'_{n},\dots,g'_{0}),\quad&g_{W}&\coloneqq [\fg_{W}]\in
    \widetilde g_{W},\quad &\omega _{W }&\coloneqq \omega (\widetilde g_{W}),\\
    \fg_{T}&=(g''_{n},\dots,g''_{0}),\quad&g_{T}&\coloneqq [\fg_{T}]\in
    \widetilde g_{T},\quad &\omega _{T}&\coloneqq \omega (\widetilde g_{T}).
  \end{alignat*}
We define the current
\begin{displaymath}
  [\fg_{Z}]\cdot \caP(W)\cdot [\fg_{T}]=
  \sum_{i=0}^{n}\sum_{k=0}^{\ell}g_{i}\cdot W_{i}\cdot \delta
  _{W}\cdot W_{m}\cdot g''_{k}\cdot W_{k}.
\end{displaymath}
Then, using again Lemma \ref{lemm:13},
\begin{align*}
  d\big( (-1)^{m}[\fg_{Z}]\cdot \caP(W)\cdot [\fg_{T}]
  &+[\fg_{Z}]\cdot [\fg_{W}]\cdot \omega _{T}
    \big) \\
  &=(-1)^{m}\left(
    \omega _{Z}\cdot \caP(W)\cdot [\fg_{T}]-\caP(Z\cdot W)\cdot [\fg_{T}]
    \right)\\
  &\phantom{AA}+(-1)^{n}\left(
     -[\fg_{Z}]\cdot \caP(W)\cdot\omega _{T}+[\fg_{Z}]\cdot
     \caP(W\cdot T)
     \right)\\
  &\phantom{AA}+\left(
    \omega _{Z}\cdot [\fg_W]\cdot \omega _{T}-\caP(Z)\cdot [\fg_W]\cdot \omega _{T}
     \right)\\
  &\phantom{AA}+(-1)^{n}\left(
    - [\fg_{Z}]\cdot \omega _{W}\cdot \omega _{T}+[\fg_{Z}]\cdot \caP(W)\cdot\omega _{T}
     \right)\\
  &=(-1)^n\big(\fg_{Z}\ast'(g_{W}\ast \fg_{T})-
  (g_{Z}\ast \fg_{W})\ast \fg_{T}\big).
\end{align*}
Applying equation \eqref{eq:36} we obtain the second statement of the theorem.\\
\end{proof}

The product of Green forms is compatible with inverse images and there
is a projection formula that relates the product, the inverse image
and the direct image. The next result can be verified formally thanks 
to Lemma \ref{lemm:13} and the corresponding properties of forms and
currents.

\begin{prop}\label{prop:19}
  Let $f\colon Y\to X$ be a morphism of complex projective manifolds,
  Let $Z\in Z^{p}(X,n)_{0}$ and  $W\in 
  Z^{q}(X,m)_{0}$ cycles in good position with respect to $f$ that
  intersect properly and such that $Z\cdot W$ is also in good position
  with respect to $f$ and $T\in Z^{r}(Y,\ell)_{0}$ a cycle that
  intersects  properly $f^{\ast}Z$. Let $\widetilde g_{Z}$,
  $\widetilde g_{W}$ and $\widetilde g_{T}$ be classes of Green
  currents for $Z$, $W$ and $T$ respectively. Then
  \begin{enumerate}
  \item $f^{\ast}(\widetilde g_{Z}\ast \widetilde g_{W})=
    f^{\ast}\widetilde g_{Z}\ast f^{\ast}\widetilde g_{W}$;
  \item $f_{\ast}(f^{\ast}\widetilde g_{Z}\ast \widetilde g_{T})=
    \widetilde g_{Z}\ast f_{\ast}\widetilde g_{T}$.
  \end{enumerate}
\end{prop}

\section{Higher arithmetic Chow groups}
\label{sec:high-arithm-chow}

In this section we present a definition of higher arithmetic Chow
groups for a smooth and projective variety $X$ of dimension $d$,
defined over an arithmetic field $F$.

Adapting Deligne-Soul\'e proposal for
higher arithmetic $K$ theory to higher arithmetic Chow groups, we can
define them as the relative homology of
the cycle class map between
Bloch cycle complex and a complex that computes Deligne-Beilinson
cohomology.
We will denote by $\widehat{\CH}^{\ast}(X,\ast)_{0}$
the groups obtained following this idea. They fit in a long exact sequence
\begin{multline*}
  \dots \to H^{2p-n-1}_{\caD}(X,\R(p))\to\\ \widehat{\CH}^{p}(X,n)_{0}\to
  {\CH}^{p}(X,n) \to  H^{2p-n}_{\caD}(X,\R(p))\to \dots
\end{multline*}
This method is followed by Goncharov \cite{Goncharov:prAmc} and
by Burgos and Feliu \cite{BurgosFeliu:hacg}.

Takeda
\cite{Takeda:haKt} has proposed a new definition of higher arithmetic
$K$-groups that map surjectively over the higher $K$
groups. Translating his idea to higher arithmetic Chow groups, one 
sought to find groups $\widehat{\CH}^{\ast}(X,\ast)$ that fit into an
exact sequence of the form
\begin{displaymath}
  {\CH}^{p}(X,n+1)
  \to \caD^{2p-n-1}(X,p)/\im d_{\caD}\to \widehat{\CH}^{p}(X,n)\to
  {\CH}^{p}(X,n) \to 0,
\end{displaymath}
where $\caD^{*}(X,*)$ is a complex that computes real
Deligne-Beilinson cohomology (or absolute Hodge cohomology). For each
$n,p$, the
group $\widehat{\CH}^{p}(X,n)_{0}$ would be recovered as a special subgroup
of $\widehat{\CH}^{p}(X,n)$. Clearly such a definition would depend on
the choice of the complex $\caD^{*}(X,*)$. 
A definition along these lines was
already proposed in Elisenda Feliu's thesis (\cite{Feliu:Thesis}, \S
3.10), where it was referred to as the \emph{modified higher arithmetic Chow
  groups}.

Below we give a definition of
such higher arithmetic Chow
groups which is analogous to the one given by Gillet and Soul\'e.

\subsection{Higher arithmetic Chow groups} Throughout this section, we
will work over an $\textit{arithmetic field}$. An arithmetic field is
a triple $(F, \Sigma , F_{\infty})$, where $F$ is a field, $\Sigma $
is a non-empty set of complex embeddings $F\hookrightarrow \C$ of $F$
and $F_{\infty}$ is a conjugate-linear $\C$-algebra automorphism of
$\C^{\Sigma }$ that leaves the image of $F$ under the diagonal immersion
invariant. Primarily, we will use the notation $F$, and the triple
will be understood. Typical examples of an arithmetic field are 
number fields, $\R$ and $\C$. Let $\fA$ be any
of the complexes that appear in Example \ref{exm:6} and 
let $X$ be a smooth and projective
variety of dimension 
$d$, defined over an arithmetic 
field $F$.

\medskip
\textbf{Real varieties.} Since $X$ is defined over an arithmetic field $F$,
we consider the associated complex space
\begin{displaymath}
  X_{\C}\coloneqq \coprod_{\sigma}X_{\sigma\in \Sigma },\quad \sigma\colon
F\hookrightarrow \C.
\end{displaymath}
Then $F_{\infty}$ induces an  antilinear involution on $X_{\C}$ also
denoted $F_{\infty}$. We will denote by
$\sigma_{F_{\infty}}$, the involution of $\fA^n(X_{\C}, p)$ given,  for any
element $\eta \in \fA^n(X_{\C}, p)$, by
\begin{displaymath}
  \sigma_{F_{\infty}}(\eta)=\overline{F_{\infty}^{\ast}\eta}.
\end{displaymath}
Then, the real Deligne-Beilinson cohomology of $X$ is defined as
$$H^n_{\DB}(X, \R(p))\coloneqq H^n_{\DB}(X_{\C}, \R(p))^{\sigma_{F_{\infty}}},$$
where the superscript $\sigma_{F_{\infty}}$ means the fixed part under
$\sigma_{F_{\infty}}$. The real Deligne cohomology of $X$ is the cohomology
of the real complex
$$\fA^n(X,p)\coloneqq \fA^n(X_{\C},p)^{\sigma_{F_{\infty}}},$$
i.e. there is an isomorphism
$$H^n_{\DB}(X, \R(p))\cong H^n(\fA^{*}(X,p), d).$$
We will consider this setting for the rest of the article.

\begin{df}\label{HCDef1}
We define the group of \emph{higher arithmetic cycles},
denoted
$\widehat{Z}^p(X,n,\fA)$ as the subgroup of $Z^p(X,n)_{0}\oplus
\widetilde{\fA} ^{2p-n-1}_{D}(X,p)$ consisting of pairs $(Z,\widetilde
g_Z)$ such that
$Z$ is a cycle and $\widetilde g_Z$ is a class of  $\fA$-Green current
for $Z$. We will frequently omit the simbol $\widetilde{\phantom g}$
over a Green current in an arithmetic cycle as it is understood. 

Let $\widehat{Z}^p_{\rat}(X,n,\fA)$ be the subgroup of
$\widehat{Z}^p(X,n,\fA)$ generated by elements of the form
\begin{alignat*}{2}
  &(\delta Z, -\widetilde{\caP}_{\fA}(Z)),& \text{ for } Z&\in Z^p(X,n+1)_{0}.
\end{alignat*}
We define the \emph{higher
arithmetic Chow groups} to be the quotient
\begin{displaymath}
  \widehat{\CH}^p(X,n,\fA)\coloneqq
\widehat{Z}^p(X,n,\fA)/\widehat{Z}^p_{\rat}(X,n,\fA).
\end{displaymath}
Again when $\fA$ is understood  we omit it from the
notation. 
\end{df}

\begin{rmk}\label{rem:9}
  Similarly, we can define $\widehat{Z}^p(X,n,\fA)_{00}$ and
  $\widehat{Z}^p_{\rat}(X,n,\fA)_{00}$ using cycles in the refined
  normalized complex  and obtain
  \begin{displaymath}
  \widehat{\CH}^p(X,n,\fA)=
\widehat{Z}^p(X,n,\fA)_{00}/\widehat{Z}^p_{\rat}(X,n,\fA)_{00}.    
  \end{displaymath}

\end{rmk}

An arithmetic field $F$ is a particular case of an arithmetic ring. Therefore we have at our disposal the arithmetic Chow groups
$\widehat{\CH}^{p}(X)$ defined in \cite{GilletSoule:ait}. We can
recover these groups using the Deligne complex $\DB$. We recall that
$d$ is the dimension of $X$.

\begin{prop}\label{case0}
  For each $p\ge 0$, the assignment
  \begin{equation}
    \label{eq:89}
    (Z,g_{Z})\mapsto (Z, 2(2\pi
  i)^{d-p+1}g_{Z})
  \end{equation}
defines an isomorphism 
\begin{displaymath}
  \widehat{\CH}^p(X,0, \DB)\xrightarrow{\cong} \widehat{\CH}^p(X).
\end{displaymath}
\end{prop}
\begin{proof}
  One has to take care
  of the different normalization used here and in
  \cite{GilletSoule:ait}. Let 
  $Z\in Z^{p}(X)$ be a cycle of codimension $p$ and denote by $\delta
  _{Z}^{\Top}$ the current integration along $Z$ with the topologist
  convention. Then
  \begin{equation}\label{eq:86}
    \delta _{Z}=\frac{1}{(2\pi i)^{d-p}} \delta
  _{Z}^{\Top}
  \end{equation}
  Recall  the real operator $d^{c}=i/(4\pi)(\bar \partial-\partial)
  $. The differential
  \begin{displaymath}
    d_{\DB}\colon \DB^{2p-1}(X,p)\to \DB^{2p}(X,p)
  \end{displaymath}
  is given by
  \begin{equation}\label{eq:87}
    d_{\DB}=-2\partial\bar \partial = 2(2\pi i)dd^{c}.
  \end{equation}
  The condition for $g$ to be a Green current for $Z$ in the sense of
  \emph{loc. cit.} is that
  \begin{displaymath}
    dd^{c}g + \delta ^{\Top}_{Z}\text{ is a smooth form}.
  \end{displaymath}
  Since we are in the case $n=0$, we have that $\caP(Z)=\delta
  _{Z}$. Therefore $g$ is a Green current for $Z$, in the sense of the
  present paper, if
  \begin{displaymath}
    -2\partial\bar \partial g_{Z}+\delta _{Z}\text{ is a smooth form}.
  \end{displaymath}
  Using \eqref{eq:86} and \eqref{eq:87} we deduce that $g$ is a Green
  current for $Z$ in the sense of this paper, if and only if $2(2\pi
  i)^{d-p+1}g$ is a Green current in the sense of 
  \emph{loc. cit.} Therefore the groups of arithmetic cycles are
  the same in both cases.

  Now we have to check that the notion of
  rational equivalence in both theories are equivalent. The main
  idea is to follow the proof of \cite[Proposition 1.6]{Fulton:IT}.

  Observe that, if $Z\in Z^{p}(X,1)$ is degenerate, then
  $\delta Z=0$ and $\caP(Z)=0$. Thus in order to discuss arithmetic rational
  equivalence at the level of $\widehat{Z}^p(X,0,\fA)$ we can use
  the group $Z^p(X,1)$ instead of $Z^p(X,1)_{0}$.  

  In \cite{GilletSoule:ait} the group of arithmetic cycles rationally
  equivalent to zero is generated by cycles of the form $(\Div
  f,(2\pi i)^{d-p+1}i_{\ast}[-\log|f|^{2}])$, where $i\colon W\to X$ is
  the inclusion of a codimension $p-1$ subvariety and $f\in
  K(W)^{\times}$ is a nonzero rational function. The factor $(2\pi
  i)^{d-p+1}$ comes from \eqref{eq:7}.

  Consider $\Gamma _{f}\subset W\times \P^{1}$, the graph of $f$. By
  sending it to $X\times \P^{1}$ and restricting to $X\times \square^{1}$
  we obtain a subvariety $Z_{f}\subset X\times \square^{1}$. Since it is
  defined as the graph of a nonzero rational function, it intersect
  properly the faces of $X\times \square^{1} $ and determines a
  pre-cycle  also denoted by $Z_{f}$.

  Now we compute, using \eqref{eq:88} and \eqref{Regulator:Eqn}, and
  the notation therein,
  \begin{displaymath}
    -\caP(Z_{f})=  \frac{1}{2}(p_{1})_{\ast}\iota_{\ast}[\iota
    ^{\ast}p_{2}^{\ast}\log t\bar t]=\frac{1}{2}i_{\ast}[\log|f|^{2}]. 
  \end{displaymath}
while
\begin{displaymath}
  \delta Z_{f}=Z_{f}\cap X\times \infty-Z_{f}\cap X\times 0=-\Div(f).
\end{displaymath}
In conclusion, the assignment \eqref{eq:89} sends $(-\delta Z_{f},\caP(Z_{f}))$
to the arithmetic cycle $(\Div f,(2\pi
i)^{d-p+1}i_{\ast}[-\log|f|^{2}])$ and every
arithmetic cycle rationally equivalent to zero in \emph{loc. cit.} is
also rationally equivalent to zero here.

Conversely, let $Z$ be an irreducible subvariety of $X\times
\square^{1}$ intersecting properly all the faces. The restriction of
the projection $p_{2}\colon X\times \square^{1}\to \square^{1}$
determines a nonzero rational function $f'$ on $Z$. Let $W\subset X$ be
the closure of the image of $Z$ by the first projection and let $p\colon Z\to W$ be
the induced map. If $\codim W > p-1$ then one verifies easily that
$\delta Z=0$ and $\caP(Z)=0$. If $\codim W = p-1$, write
$f=N_{K(X)/K(W)}f'$, where $N_{K(Z)/K(W)}$ is the norm of the field
extension $K(Z)/K(W)$. then, by the discussion before
\cite[Proposition 1.6]{Fulton:IT} 
\begin{displaymath}
  \delta Z = -\Div f.
\end{displaymath}
Let $i'\colon Z\to X\times \square^{1}$ and $i\colon W\to X$ denote
the inclusions. Then, using the definition of the map $\caP$ and of
the norm of a field extension,
\begin{displaymath}
  \caP(Z)=p_{\ast}\frac{1}{2}i'_{\ast}[-\log|f'|^{2}]=\frac{1}{2}i_{\ast}[-\log|f|^{2}]. 
\end{displaymath}
In consequence the notions of rational equivalence agree and we obtain
the result.
\end{proof}

\begin{notation}\label{def:12}
  Let $F$ be an arithmetic field and $X$ a smooth projective variety  of
  dimension $d$ defined over $F$. 
  Let $\fA$ be any of the complexes in Diagram \eqref{eq:43}. When $\fA$ has been fixed, we will often drop the suffix/prefix $\fA$ from the definition. For example, the higher arithmetic Chow groups for a fixed complex $\fA$ will simply be denoted by $\widehat{\CH}^{\ast}(X,\bullet)$, and the Goncharov regulator by $\mathcal{P}$. 
\end{notation}

\begin{rmk}\label{rem:4}
  Let
  \begin{displaymath}
    \xymatrix{
      \fB \ar@<5pt>[r]^{\varphi} &
      \fA\ar@<5pt>[l]^{\psi }}
  \end{displaymath}
  be two of the complexes of Example \ref{exm:6} with the
  corresponding morphisms. Then, for each $n,p$ there are induced
  morphisms 
  \begin{displaymath}
    \xymatrix{
      \widehat{\CH}^p(X,n,\fB) \ar@<5pt>[r]^{\varphi} &
      \widehat{\CH}^p(X,n,\fA)\ar@<5pt>[l]^{\psi }}.
  \end{displaymath}
\end{rmk}

\subsection{Exact sequences}
\label{sec:exact-sequences}
Like the usual arithmetic Chow groups, the higher arithmetic version comes equipped with the following natural maps:

\begin{alignat}{2}
  \widetilde{\fA}^{2p-n-1}&\xrightarrow{\amap}\widehat{\CH}^{p}(X,n,\fA),&\quad
  b&\mapsto [(0,[b])],\notag \\ 
  \widehat{\CH}^p(X,n,\fA)&\xrightarrow{\zeta }\CH^p(X,n),&\quad [(Z,g)]&\mapsto
 [Z],  \notag  \\
 \widehat{\CH}^p(X,n,\fA)&\xrightarrow{\omega}Z\fA^{2p-n},&\quad 
 [(Z,g_{Z})]&\mapsto \omega (g_{Z}).\notag 
\end{alignat}
%In terms of the Goncharov presentation, the map $\amap$ is given by
%\begin{displaymath}
%  b\mapsto \Psi ((0,[b]))=(0,db,-[b],0)\sim (0,0,0,-b)
%\end{displaymath}
%thus, it agrees with the map $\amap$ for modified homology in
%\eqref{eq:41}.
%
%By Corollary \ref{2Cor1}
%\begin{displaymath}
%  \ker \omega = H_{n}(s(\caP_{\fA})),
%\end{displaymath}
%where $\caP_{\fA}$ is the map of Definition \ref{def:7}.
We will denote
\begin{displaymath}
  \widehat{\CH}^{p}(X,n,\fA)^{0}=\ker (\omega) \subset \widehat{\CH}^{p}(X,n,\fA),
\end{displaymath}
and write
\begin{alignat}{2}
  \widehat{\CH}^p(X,n,\fA)^{0}&\xrightarrow{\iota} ,\widehat{\CH}^p(X,n,\fA),&\quad
  [(Z,g_{Z})]&\mapsto [(Z,g_{Z})],\notag \\ 
  Z\fA^{2p-n-1}&\xrightarrow{\bmap}\widehat{\CH}^p(X,n,\fA)^{0},&\quad
  b&\mapsto [(0,b)].\notag 
\end{alignat}
\begin{rmk}\label{rmk:10}
In a later sequel, we will show that for $n>0$ the groups $\widehat{\CH}^p(X,n,\fA)^0$ are isomorphic to the higher arithmetic Chow groups defined in \cite{BurgosFeliu:hacg}. Hence they don't depend on the complex $\fA$.
\end{rmk}
Let $\rho\colon \CH^p(X,n)\rightarrow H^{2p-n}_{\fD}(X,\R(p))$ denote
the cubical Goncharov regulator. Recall that it is represented at the level of complexes
by the map 
$\caP$, and agrees, after composing
with the isomorphism between $K$-theory and higher Chow groups with
Beilinson's regulator.

Denote also by $\rho\colon \CH^p(X,n)\rightarrow
\widetilde{\fA}^{2p-n}$
the composition
\begin{displaymath}
  \CH^p(X,n)\xrightarrow{\caP}  H^{2p-n}_{\fD}(X,\R(p))
  \hookrightarrow  \widetilde{\fA}^{2p-n}. 
\end{displaymath}
This map is given by $[Z]\mapsto \widetilde{\omega}(g_Z)$, for a choice
of Green current $g_Z$ for $Z$, because $\caP(Z)$ is cohomologous to
$\omega (g_{Z})$.  

Analogous to the exact sequences that arise for arithmetic Chow
groups, we have the following exact sequences:
\begin{thm}\label{thm:3} The following sequences are exact
  \begin{enumerate}
  \item\label{item:27}
    \begin{displaymath}
    \CH^{p}(X,n+1)\xrightarrow{\rho } \widetilde
    {\fA}^{2p-n-1}\xrightarrow{\amap}
    \widehat{\CH}^{p}(X,n,\fA) \xrightarrow{\zeta}
    \CH^{p}(X,n)\to 0.
  \end{displaymath}
\item\label{item:128}
  \begin{multline*}
    \CH^{p}(X,n+1)\xrightarrow{\rho } 
    H_{\DB}^{2p-n-1}(X,\R(p)) \xrightarrow{\amap}
    \widehat{\CH}^{p}(X,n,\fA) \xrightarrow{\zeta\oplus \omega }\\
    \CH^{p}(X,n)\oplus Z\fA^{2p-n}\to H_{\DB}^{2p-n}(X,\R(p))\to 0. 
  \end{multline*}
  \item\label{item:129}
  \begin{displaymath}
  \CH^{p}(X,n+1)\xrightarrow{\rho } \widetilde
    {\fA}^{2p-n-1}\xrightarrow{\amap}
    \widehat{\CH}^{p}(X,n,\fA)^0 \xrightarrow{\zeta}
    \CH^{p}(X,n)^0\to 0,
    \end{displaymath}
  where for (\ref{item:129}), $\CH^p(X,n)^0=\text{ker}(\rho)$.

    %\item 
%  \begin{displaymath}
%    0\to
%    \widehat{\CH}^{p}(X,n)^{0}\xrightarrow{\iota}
%    \widehat{\CH}^{p}(X,n) \xrightarrow{\zeta\oplus \omega }\\
%    \CH^{p}(X,n)\oplus Z\fA^{2p-n}
%  \end{displaymath}
%\item 
%  \begin{multline*}
%    \dots \to H^{2p-n-1}_{\fD}(X,\R(p))\to  \widehat{\CH}^{p}(X,n)^0
%    \to\\ \CH^{p}(X,n) \xrightarrow{\rho } H^{2p-n}_{\fD}(X,\R(p)) \to \dots
%  \end{multline*}
  \end{enumerate}
\end{thm}
\begin{proof}
Since the ideas involved in all the three sequences are similar in nature, we show the proof for (\ref{item:27}), leaving the rest to the reader. We first show the exactness at $\widehat{\CH}^p(X,n, \fA)$: The inclusion $\text{im}(\amap)\subset \text{ker}(\zeta)$ is obvious. So we prove the other inclusion. Notice that $\zeta([Z,g_Z])=0$ means the class $[Z]=0\in \CH^p(X,n)$. So $Z=\delta(Z')$ for some precycle $Z'\in Z^p(X,n+1)_0$. We have the following relations in $\widehat{\CH}^p(X,n,\fA)$:
\begin{displaymath}
[(Z,g_Z)]=[(Z,g_Z)]-[(\delta(Z'),-\mathcal{P}(Z'))]=[(0, g_Z+\mathcal{P}(Z'))].
\end{displaymath}
Let $g'_Z=g_Z+\mathcal{P}(Z')$, then $dg'_Z=d(g_Z+\mathcal{P}(Z'))=[\omega(g_Z)]$. Hence, from Lemma \ref{lemm:2}, $\tilde{g}'_Z=[\eta]^\sim$, for some smooth form $\eta \in \fA^{2p-n-1}$. Thus $[(Z,g_Z)]=\amap(\widetilde{\eta})$, and we are done. Next we show the exactness at $\widetilde{\fA}^{2p-n-1}$. For an element $[Z]\in \CH^p(X, n+1)$, we have $\amap(\rho([Z]))=[(0, [\omega(g_Z))]$. Now, since $Z$ is a cycle, at the level of higher arithmetic Chow groups, we have
$$[(0, \omega(g_Z))]=[(\delta(Z), \mathcal{P}(Z))]+[(0, dg_Z)]=[(\delta(-Z), -\mathcal{P}(-Z))],$$
which is zero. Thus $\text{im}(\rho)\subset \text{ker}(\amap)$. On the other hand, if $\amap(\widetilde{\eta})=0$ for any element $\widetilde{\eta}\in \widetilde{\fA}^{2p-n-1}$, then $(0, [\eta])=(\delta(Z'), -\mathcal{P}(Z'))+(0, du)$. Thus $\delta(Z')=0$, hence $[Z']\in \CH^p(X, n+1)$, and $du-\mathcal{P}(Z')=[\eta]$. This means that if we take $[-Z']\in \CH^p(X,n+1)$, then $\rho([-Z'])=\widetilde{\omega}(u)=\widetilde{\eta}$. This shows the inclusion $\text{ker}(\amap)\subset \text{im}(\rho)$.
\end{proof}
%\red{\begin{rmk}\label{case01} For $n>0$, the groups
%  $\widehat{\CH}^p(X,n)^0$ agree with the higher arithmetic Chow groups
%  described in \cite{BurgosFeliu:hacg} and \cite{BurgosFeliuTakeda}.
%  This is not the case for $n=0$, the group $\widehat{\CH}^p(X,0)^0$
%  agrees, under the isomorphism of Proposition \ref{case0}, with the
%  subroup
%  \begin{displaymath}
%    \widehat{\CH}^{p}(X)^{0}\coloneqq \ker \omega
%    \subset \widehat{\CH}^{p}(X),
%  \end{displaymath}
%while the higher arithmetic Chow group of codimension $p$
%  described  \cite{BurgosFeliu:hacg} agree, for $n=0$, with
%  $\widehat{\CH}^{p}(X)$. 
%\end{rmk}- We should exclude this remark now.}

\subsection{Direct and inverse images}
\label{sec:direct-images}
We fix an arithmetic field $F$ in this discussion.
\subsubsection{Direct images}
\begin{thm}\label{thm:5}
  Let $f\colon X\to Y$ be a morphism of smooth equidimensional
  projective varieties over $F$ of relative dimension $e$.
  Since $X$ and $Y$ are assumed to be projective, the morphism
  $f$ is proper. 
  \begin{enumerate}
  \item \label{item:33} Let $Z\in Z^{p}(X,n)_{0}$ be a cycle,
    $g_{Z}$ a 
    Green current for $Z$ and $\omega _{Z}=\omega (g_{Z})$.
    If either $\omega _{Z}=0$ or the map $f$ is smooth, then 
    $f_{\ast}g_{Z}$ is a Green current for 
    the cycle $f_{\ast}Z$.
  \item \label{item:14} There is a morphism
    \begin{displaymath}
      f_{\ast}\colon \widehat{\CH}^{p}(X,n)^{0}\to
      \widehat{\CH}^{p-e}(Y,n)^{0}, 
    \end{displaymath}
    given by $(Z,g_{Z})\mapsto (f_{\ast}Z,f_{\ast}g_{Z})$.
  \item \label{item:15} If $f$ is smooth, there is a morphism
    \begin{displaymath}
      f_{\ast}\colon \widehat{\CH}^{p}(X,n)\to
      \widehat{\CH}^{p-e}(Y,n),
    \end{displaymath}
    given by $(Z,g_{Z})\mapsto (f_{\ast}Z,f_{\ast}g_{Z})$.
  \item \label{item:16} The direct image of higher arithmetic Chow groups is
    compatible with the direct image of higher Chow groups, with the 
    direct image in Deligne-Beilinson cohomology and, when $f$ is
    smooth with the direct image of differential forms. In particular,
    all the exact sequences of Theorem \ref{thm:3} give rise to
    commutative diagrams. 
    \item \label{thm:4} The direct image is functorial, that is, if
  $g\colon Y\to Z$ is
  another morphism then $(f\circ g)_{\ast}=f_{\ast}\circ g_{\ast}$.
  \end{enumerate}
\end{thm}
\begin{proof}
%  The key point is the commutativity of the diagram
%  \begin{equation}\label{eq:48}
%    \begin{gathered}
%    \xymatrix{Z^{p}(X,\ast)_{0}\ar[r]^-{\caP}\ar[d]_{f_{\ast}}
%      &\fA^{2p-\ast}_{D}(X,p)\ar[d]^{f_{\ast}} \\ 
%      Z^{p}(Y,\ast)_{0}\ar[r]_-{\caP} &\fA^{2p-\ast}_{D}(Y,p).
%    }      
%    \end{gathered}
%  \end{equation}
%  In fact, writing $p_{1,X}\colon X\times
%  \square^n\rightarrow X$ and $p_2\colon X\times \square ^n\rightarrow
%  \square ^n$ for  the projections, and similarly $p_{1,Y}$ and
%  $p_{2,Y}$, then, for $Z\in Z^p(X,n)_0$, we have (using
%  \ref{Regulator:Eqn})
%  \begin{multline*}
%    f_*(\mathcal{P}(Z))=f_*(p_{1,X})_*\iota_{Z,*}[\iota^*_Zp^*_2W_n]=(f\circ
%    p_{1,X})_*\iota_{Z,*}[\iota^*_Zp^*_2W_n]\\
%    =(p_{1,Y})_*\iota_{f_*Z,*}[\iota^*_{f_*Z}p^*_2W_n]=\mathcal{P}(f_*(Z)).
%  \end{multline*}
Note that \eqref{item:33} is essentially the content of Proposition \ref{prop:14}. Again by the compatibility of direct images with the differential of currents and Proposition \ref{prop:3}, we
  see that the maps in \eqref{item:14} and \ref{item:15} are well
  defined. The use of $(f_\ast Z, f_\ast g_Z)$, instead of $(f_\ast Z, f_\ast \widetilde{g}_Z)$ is justified in this definition from the commutativity relation $f_\ast du=d(f_\ast u)$, for any current $u$.

  The compatibility with the other 
  direct images and the functoriality can be verified easily. 
\end{proof}

\subsubsection{Inverse images}
\label{sec:inverse-images}

The next result is the existence of inverse images. As in the previous
section the field $F$ and the complex $\fA$ are fixed.

Let $f\colon X\to Y$ be a morphism of smooth 
  projective varieties over $F$. The higher arithmetic
  Chow groups can be defined using cycles in good position with
  respect to $f$. Denote by $\widehat{Z}^p_{f}(Y,n,\fA)$ the subgroup
  of $\widehat{Z}^p(Y,n,\fA)$ consisting of pairs $(Z,\widetilde
  g_{Z})$ with $Z\in  Z^p_f(Y,n)_0$, and by
  $\widehat{Z}^p_{f,\rat}(Y,n,\fA)$ the subgroup of
  $\widehat{Z}^p_{\rat}(Y,n,\fA)$ consisting of elements of the form
\begin{alignat*}{2}
  &(\delta Z, -\widetilde{\caP}_{\fA}(Z)),& \text{ for } Z&\in Z_{f}^p(Y,n+1)_{0}.
\end{alignat*}

\begin{lem}\label{lemm:15}
  The natural map
  \begin{displaymath}
    \frac{\widehat{Z}^p_{f}(Y,n,\fA)}{\widehat{Z}^p_{f,\rat}(Y,n,\fA)}\to
     \widehat{\CH}^{p}(Y,n,\fA)
  \end{displaymath}
is an isomorphism.
\end{lem}
\begin{proof}
  Repeating the proof of Theorem \ref{thm:3}~\eqref{item:27} and using
  Corollary \ref{cor:3} one sees that there is a commutative diagram
  with exact rows
  \begin{displaymath}
    \xymatrix{
      \CH^{p}(X,n+1)\ar[r]\ar@{=}[d] 
     & \widetilde{\fA}^{2p-n-1}\ar[r]\ar@{=}[d]
    &
    \frac{\widehat{Z}^p_{f}(Y,n,\fA)}{\widehat{Z}^p_{f,\rat}(Y,n,\fA)}
    \ar[r]\ar[d] 
    & \CH^{p}(X,n)\ar[r]\ar@{=}[d] &0\\
      \CH^{p}(X,n+1)\ar[r] 
     & \widetilde{\fA}^{2p-n-1}\ar[r]
    & \widehat{\CH}^{p}(X,n,\fA) \ar[r]
    & \CH^{p}(X,n)\ar[r] &0
    }
  \end{displaymath}
  from which the lemma follows.
\end{proof}

We use the previous lemma to define inverse images.
Let $(Z,\widetilde
  g_{Z})\in \widehat{Z}^p_{f}(Y,n,\fA)$. 
  By Proposition
  \ref{prop:20}, the pair $(f^{\ast}Z,f^{\ast}\widetilde g_{Z})$ is an
  arithmetic cycle on $X$.

\begin{thm}\label{thm:6}
  Let $f\colon X\to Y$ be a morphism of smooth 
  projective varieties over $F$.
  \begin{enumerate}
  \item \label{item:17} The previous construction induces a morphism
    \begin{displaymath}
      f^{\ast}\colon \widehat{\CH}^{p}(Y,n)\to
      \widehat{\CH}^{p}(X,n)
    \end{displaymath}
    that sends $\widehat{\CH}^{p}(Y,n)^{0} $ to
    $\widehat{\CH}^{p}(X,n)^{0}$. 
  \item \label{item:26} The inverse image of higher arithmetic Chow groups is
    compatible with the inverse image of higher Chow groups, with the 
    inverse image in Deligne-Beilinson cohomology and with the inverse
    image of differential forms. In particular, 
    all the exact sequences of Theorem \ref{thm:3} give rise to
    commutative diagrams. 
    \item The inverse image is functorial, that is, if
  $g\colon Y\to Z$ is 
  another morphism then $(f\circ g)^{\ast}=g^{\ast}\circ f^{\ast}$.
  \end{enumerate}
\end{thm}
\begin{proof}
Let $\alpha\in \widehat{\CH}^p(Y,n)$. From
Lemma \ref{lemm:15}, we can write $\alpha=[(Z, \widetilde{g}_Z)]$, with $Z\in
Z^p_f(Y,n)_0$. Then $f^\ast(\alpha)=[(f^\ast Z, f^\ast
\widetilde{g}_Z)]$. In view of Lemma \ref{lemm:15}, 
to see that $f^\ast$ is well defined as a map between higher
arithmetic Chow groups, we need to show that, for any precycle $Z'\in
Z^p_f(Y, n+1)_0$, the relation
\begin{displaymath}
  f^{\ast} (-\widetilde \caP(Z'))=-\widetilde \caP(f^{\ast}Z')
\end{displaymath}
holds, where the inverse image in the left hand side is the inverse
image of classes of Green currents. To compute the inverse image on
the left side we need a basic Green form for $\delta(Z')$ representing
the class $-\widetilde{\mathcal{P}}(Z')$. 

Since $Z'\in Z^p_f(Y, n+1)_0$, its class in $\CH^{p}(Y\times
\square^{n+1})$ is zero, because,  $(\delta _{1}^{1})^{\ast}Z'=0$ and
$(\delta _{1}^{1})^{\ast}$ induces an isomorphism between $\CH^{p}(Y\times
\square^{n+1})$  and $\CH^{p}(Y\times
\square^{n})$. 
Therefore there exists a basic Green form  $g_{n+1}$ for $Z'$ on
$Y\times \square ^{n+1}$ (Definition \ref{def:23}) such that
$dg_{n+1}=0$. Then   
\begin{equation}\label{eq:37}
  \fg_{\delta(Z')}=(\delta g_{n+1},0,0,\cdots, 0)
\end{equation}
is a basic Green form for $\delta(Z')$. Hence, from Proposition
\ref{GreenCurrent} we know that the corresponding current is a Green
current for $\delta(Z')$. We compute, using Lemma \ref{lemm:13} 
%\begin{displaymath}
\begin{align*}
  d((p_{1})_{\ast}[g_{n+1}\cdot W_{n+1}])
  &=-(p_{1})_\ast(\delta_{Z'}\cdot W_{n+1})-(p_{1})_\ast[\delta
    g_{n+1}\cdot W_{n}]\\
  &=-\caP(Z')-[\fg_{\delta (Z')}],
\end{align*}
which implies that $\fg_{\delta (Z')}$ represents the class
$-\widetilde \caP(Z')$. So
$f^\ast(-\widetilde{\mathcal{P}}(Z'))=[f^\ast\fg_{\delta(Z')}]^\sim$
and $f^{\ast}$ is well defined.
The compatibility with other inverse images and the functoriality are left to the reader.
\end{proof}

\subsection{Product structure}
The intersection theory of higher Chow groups and the $\ast$-product
of Green currents of higher algebraic cycles described in the last
section, combines to give an intersection theory at the level of
higher arithmetic Chow groups. To this end we first use Corollary
\ref{cor:2} to see that the higher arithmetic Chow groups can be
represented using cycles intersecting properly a given cycle.

Let $X$ be a smooth projective variety over $F$ and $W\in Z^{q}(X,m)$
a precycle. Denote by $\widehat{Z}^p_{W}(X,n,\fA)$ the subgroup
  of $\widehat{Z}^p(X,n,\fA)$ consisting of pairs $(Z,\widetilde
  g_{Z})$ with $Z\in  Z^p_W(X,n)_0$, and by
  $\widehat{Z}^p_{W,\rat}(X,n,\fA)$ the subgroup of
  $\widehat{Z}^p_{\rat}(X,n,\fA)$ consisting of elements of the form
\begin{alignat*}{2}
  &(\delta Z, -\widetilde{\caP}_{\fA}(Z)),& \text{ for } Z&\in Z_{W}^p(Y,n+1)_{0}.
\end{alignat*}

The following result is proved as Lemma \ref{lemm:15}.
\begin{lem}\label{lemm:16}
  The natural map
  \begin{displaymath}
    \frac{\widehat{Z}^p_{W}(X,n,\fA)}{\widehat{Z}^p_{W,\rat}(X,n,\fA)}\to
     \widehat{\CH}^{p}(X,n,\fA)
  \end{displaymath}
is an isomorphism.
\end{lem}

We now fix the complex $\DB_{\TW}$ and denote
$\widehat{\CH}^p(X,n)=\widehat{\CH}^p(X,n,\DB_{\TW})$. 
\begin{thm}\label{thm:int}
There is a pairing
\begin{displaymath}
  \widehat{\CH}^p(X,n)\otimes \widehat{\CH}^q(X,m)\rightarrow
  \widehat{\CH}^{p+q}(X, n+m), 
\end{displaymath}
which is  associative and graded commutative with respect to  $n$, and
that satisfies the following properties:
\begin{enumerate}
\item \label{item:a}
The maps
\begin{displaymath}
  \zeta\colon \bigoplus_{p,n\geq 0}\widehat{\CH}^p(X,n)\rightarrow
  \bigoplus_{p,n\geq 0}\CH^p(X,n), 
\end{displaymath}
and 
\begin{displaymath}
  \omega \colon \bigoplus_{p,n\geq 0}\widehat{\CH}^p(X,n)\rightarrow
  \bigoplus_{p,n\geq 0}Z\DB_{\TW}^{2p-n}(X,p), 
\end{displaymath}
are multiplicative.
\vspace{0.5cm}
\item \label{item:b}
  For a morphism $f\colon X\rightarrow Y$, the pullback
  \begin{displaymath}
   f^\ast\colon \widehat{\CH}^p(Y,n)\rightarrow \widehat{\CH}^p(X,n) 
  \end{displaymath}
  is multiplicative, i.e., given $\alpha\in \widehat{\CH}^p(Y,n)$ and $\beta \in \widehat{\CH}^q(Y,m)$, we have
$$f^\ast(\alpha\cdot \beta)=f^\ast(\alpha)\cdot f^\ast(\beta).$$
Further if $f$ is smooth and proper, the pushforward morphism satisfies the projection formula: Let $\alpha \in \widehat{\CH}^p(Y,n)$ and $\beta\in \widehat{\CH}^q(X,m)$. Then
$$f_{\ast}(f^\ast(\alpha)\cdot \beta)=\alpha\cdot f_{\ast}(\beta).$$
\end{enumerate}
\end{thm}
\begin{proof}
Let $\alpha\in \widehat{\CH}^p(X,n)$ and $\beta\in
\widehat{\CH}^q(X,m)$ be two elements in the higher arithmetic Chow
groups. Choose a representative $(W, \widetilde{g}_W)$  of $\beta
$. From Lemma \ref{lemm:16}, $\alpha$ can be represented as $(Z,
\widetilde{g}_Z)$ with $Z$ intersecting $W$ properly as in Definition
\ref{def:19}. We define
\begin{displaymath}
  \alpha\cdot \beta:=[(Z\cdot W, \widetilde{g}_Z\ast \widetilde{g}_W)].
\end{displaymath}
We have to show that this product does not depend on the chosen
representatives.

\begin{lem}\label{lemm:17} Let $(W,\widetilde g_{W})$ as before.
Let $Z'\in Z_{W}^p(X,n+1)_0$ be a pre-cycle. Then
\begin{displaymath}
  (-1)^{nm}\widetilde g_{W}\ast (-\widetilde{\mathcal{P}}(Z'))= 
  -\widetilde{\mathcal{P}}(Z')\ast
\widetilde{g}_W=-\widetilde{\mathcal{P}}(Z'\cdot W),
\end{displaymath}
where the $\ast$-products are defined using that  $-\mathcal{P}(Z')$
is a Green current for $\delta Z'$. 
\end{lem}

\begin{proof}[Proof of the lemma]
By Theorem \ref{thm:12}~\eqref{item:28},
$\widetilde{P}(Z')\ast \widetilde{g}_W=(-1)^{nm}\widetilde{g}_W\ast
\widetilde{P}(Z')$, so we only need to compute $\widetilde{g}_W\ast
\widetilde{P}(Z')$. Let $\fg_{\delta (Z')}$ be as in \eqref{eq:37}.
Since $W$ and $\delta(Z')$ intersect properly, we compute using Lemma
\ref{lemm:13}, the fact that $W$ is a cycle and that $dg_{n+1}=0$,
\begin{align*} 
d(p_{1})_{\ast}(\delta _{W}\cdot W_m&\cdot g_{n+1}\cdot W_{n+1})\\
&=(-1)^{m+1}(p_{1})_{\ast}\Big((\delta _{W\cdot Z'}\cdot
W_{m+n})+(\delta _{W}\cdot W_m\cdot \delta g_{n+1}\cdot W_n)\Big)\\
&=(-1)^{m}( -\caP(W\cdot Z')-g_W\ast \fg_{\delta (Z')}). 
\end{align*}
This shows that
\begin{displaymath}
  \widetilde g_{W}\ast [\fg_{\delta (Z')}]^{\sim}=-\widetilde
  {\caP}(W\cdot Z').
\end{displaymath}
Since the integrals defining $\caP(W\cdot Z')$ and
$(-1)^{nm}\caP(Z'\cdot W)$ are the same except for the ordering of
the variables, we obtain the lemma.
\end{proof}

Thanks to Lemma \ref{lemm:17} the product is well defined. Since the
product of higher 
cycles intersecting properly is associative, by \eqref{item:29} of Theorem
\ref{thm:12} we deduce that, if $Z$, $W$ and $T$ are as in Theorem
\ref{thm:12}, 
\begin{multline*}
  \big((Z,\widetilde g_{Z})\cdot (W,\widetilde g_{W})\big)\cdot (T,\widetilde
  g_{T})= \big( (Z\cdot W)\cdot T, (\widetilde g_{Z}\ast \widetilde
  g_{W})\ast \widetilde g_{T}\big)\\=
\big( Z\cdot (W\cdot T), \widetilde g_{Z}\ast (\widetilde
  g_{W}\ast \widetilde g_{T})\big)=
  (Z,\widetilde g_{Z})\cdot \big((W,\widetilde g_{W})\cdot (T,\widetilde
  g_{T})\big).
\end{multline*}

We show the graded commutativity in more details. Let
$\alpha$ and $\beta$ be as above, with representatives $[(Z,g_Z)]$ and
$[(W, g_W)]$ respectively, such that $Z\in Z^{p}(X,n)_{00}$ and $W\in
Z^{q}(X,m)_{00}$ intersect properly. Let $H_{n,m}$, given in
\eqref{eq:45}, be the homotopy that proves the graded commutativity of
the product in the higher Chow groups.  Then, using (\ref{item:28}) of
Theorem \ref{thm:12} and  \eqref{eq:44},
\begin{align*}
\alpha\cdot \beta-(-1)^{nm}\beta\cdot \alpha&=[(Z\cdot W, g_Z\ast g_W)]-(-1)^{nm}[(W\cdot Z, g_W\ast g_Z)]\\
&=[(Z\cdot W-(-1)^{nm}W\cdot Z, g_Z\ast g_W-(-1)^{nm} g_W\ast g_Z)]\\
&=[(\delta(H_{n,m}(Z\cdot W)), dv)].
\end{align*}

If we show that $\caP(H_{n,m}(Z\cdot W))=0$, then
\begin{multline*}
  [(\delta(H_{n,m}(Z\otimes W)), dv)]\\=[(\delta(H_{n,m}(Z\cdot W)), -\caP(H_{n,m}(Z\cdot W))]+[(0, dv)]=0,
\end{multline*}
and the graded commutativity follows.
\begin{lem}\label{lemm:18} The equation
  \begin{displaymath}
    \caP(H_{n,m}(Z\cdot W))=0
  \end{displaymath}
  holds.
\end{lem}
\begin{proof}[Proof of the lemma.]
In view of the definition \eqref{eq:45} of $H_{n,m}$, it is enough to
show that given a pre-cycle $T\in Z^p(X, n)_0$, and $h^\ast_n$ the
morphism given by \eqref{eq:46}, we have $\mathcal{P}(h^\ast_n(T))=0$.

Fon $n$ fixed, denote by $p_1\colon X\times \square^{n}\rightarrow X$
 and $p_{1}'\colon X\times \square^{n+1}\rightarrow X$ the first
 projection. Then
 \begin{displaymath}
   \mathcal{P}(h^\ast_n(T))=(p)_\ast\left(\delta_{h^\ast_n(T)}\cdot W_{n+1}\right).
 \end{displaymath}
Since $p'_{1}=p_1\circ h_{n}$, we rewrite
\begin{displaymath}
  \mathcal{P}(h^\ast_n(T))=(p_1)_{\ast}h_{n\ast}\left(\delta_{h^\ast_n(T)}\cdot
    W_{n+1}\right). 
\end{displaymath}
 Let
\begin{enumerate}
\item
$T'=h^{-1}_n(T)\subset X\times \square^{n+1}$.
\item
$\widetilde{T}\rightarrow T'$ a resolution of singularities of $T'$.
\item
$\overline{T}$ a smooth compactification of $\widetilde{T}$,
\end{enumerate}
such that there is a map
$$\overline{h}_n\colon \overline{T}\rightarrow X\times (\P^1)^n,$$
extending $h_n$. We write $\eta_n\colon \widetilde{T}\rightarrow
X\times \square^{n+1}$. Then, by definition,
\begin{displaymath}
  h_{n\ast}\left(\delta_{h^\ast_n(T)}\cdot W_{n+1}\right) =\overline{h}_{n\ast}[\eta^\ast_nW_{n+1}].
\end{displaymath}
So it is enough to prove that 
$\overline{h}_{n\ast}[\eta^\ast_nW_{n+1}]=0$.  Let $\omega$ be an
arbitrary smooth test form in $X\times (\P^1)^n$. We have to show that
\begin{displaymath}
  \overline{h}_{n\ast}[\eta^\ast_nW_{n+1}](\omega)
  =\int_{\overline{T}}\eta^\ast_nW_{n+1}\cdot
  \overline{h}^\ast_n\omega=0.
\end{displaymath}
The right hand side is an improper integral that can be computed in an
open dense subset of $\overline{T}$.

There are open dense subsets $U_T\subset T$ and $W_T\subset
\widetilde{T}$, such that the induced map
\begin{displaymath}
  \overline{h}_n\colon W_T\rightarrow U_T,
\end{displaymath}
is a fibration.
Locally in analytic topology, one can write $U_T=\cup_iU_{i,T},$
such that
\begin{displaymath}
  \overline{h}^{-1}_n (U_{i,T})\cap W_{T}=U_{i,T}\times C.
\end{displaymath}
Here $C$ is the fibre, a smooth curve inside $\square^2$ corresponding
to the first and the last coordinate of $\square^{n+1}$. Recall that
\begin{displaymath}
  W_{n+1}\cdot \overline{h}^\ast_n \omega=\pi^\ast_1 \lambda \cdot \pi^\ast_2 \lambda\cdots \pi^\ast_{n+1}\lambda \cdot \overline{h}^\ast_n \omega,
\end{displaymath}
where $\pi_i\colon X\times \square^{n+1}\rightarrow \square$ is the
projection to the $i$-th coordinate, and
\begin{displaymath}
 \lambda =-\frac{1}{2}\left((\varepsilon +1)\otimes \frac{dt}{t}+
  (\varepsilon -1)\otimes \frac{d\bar t}{\bar t}
  +d\varepsilon\otimes \log t\bar t\right),
\end{displaymath}
were $t$ is the coordinate of $\square$.
From the definition of $h_n$, and the fact that it is locally a
fibration, $\{\pi^\ast_i\lambda\}^n_{i=2}$ and
$\overline{h}^\ast_n\omega$ do not depend on the coordinate of
$C$. Thus we can write
\begin{displaymath}
  \overline{h}_{n\ast}(\eta^\ast W_{n+1}\cdot \overline{h}^\ast_n
  \omega)_{b_{i,T}}=\pm (\pi^\ast_2\lambda \cdots\pi^\ast_n\lambda\cdot
  \omega)_{b_{i,T}}\int_{C}\pi^\ast_1\lambda\cdot
  \pi^\ast_{n+1}\lambda. 
\end{displaymath}
Let $\overline{C}\subset (\P^1)^2$ be the closure of $C$ inside
$(\P^1)^2$, and $p\colon \overline{C}\rightarrow \Spec(F)$ be the
structural morphism. Then, the integral on the right hand side is the
same as $\caP(C)=p_\ast[W_2|_C]$.
Thus Lemma \ref{lemm:18} follows from Lemma \ref{lemm:19} below. 
\end{proof}

\begin{lem}\label{lemm:19}
  Let $C\subset \square ^{2}$ be a curve that intersects properly all
  the faces of $\square^{2}$ and $\overline C\subset
  (\P^{1})^{2}$ its closure. Let $p\colon \overline C \to \Spec F$ be the
  structural morphism. Then
  \begin{displaymath}
    p_{\ast}[W_{2}\mid _{C}]=0.
  \end{displaymath}
\end{lem}
\begin{proof}
Let $x,y$ be the coordinates of $\square ^{2}$. Then, the part of
  type $(1,1)$ of $W_{2}$ is given by
  \begin{align*}
    (W_{2})^{(1,1)}&=\frac{1}{4}(\varepsilon ^{2}-1)\otimes
    \left(\frac{dx}{x}\land \frac{d\bar y}{\bar y}-
      \frac{dy}{y}\land \frac{d\bar x}{\bar x})\right)\\
    &=\frac{1}{4}(\varepsilon ^{2}-1)\otimes d\left(\log x\bar
      x\left(\frac{dy}{y}+ \frac{d\bar y}{\bar
          y}\right)\right). 
  \end{align*}
  The coordinates $x$ and $y$ restricted to $C$ give rational
  functions on $C$ that we denote also by $x$ and $y$. The fact that
  $C$ intersects properly the faces of
  $\square^{2}$ implies that $\Div x$ and $\Div y$ are
  disjoint. Denote by $f=\log x\bar x$. As is customary, if
  $\Div(y)=\sum n^{p}p$, then we write
  \begin{displaymath}
    f(\Div(y))=\sum n_{p}f(p).
  \end{displaymath}
  Since
  \begin{displaymath}
    \frac{1}{2\pi i}\int _{\overline C}d\left(f
      \frac{dy}{y}\right)=-f(\Div y),\qquad
    \frac{1}{2\pi i}\int _{\overline C}d\left(f
      \frac{d\bar y}{\bar y}\right)=f(\Div y),
  \end{displaymath}
  we obtain the lemma.
\end{proof}
  We have proved that the product is graded commutative. The
  compatibility of the product with the morphisms $\zeta $ and $\omega
  $ is a direct computation.

  We leave to the reader to check that the product is compatible with
  inverse images and the projection formula.
\end{proof}

\begin{rmk}
  Following Remark \ref{rem:4}, let $\fB$ be another of the complexes
  of Example \ref{exm:6} and let
  \begin{displaymath}
        \xymatrix{
      \widehat{\CH}^p(X,n,\fB) \ar@<5pt>[r]^-{\varphi} &
      \widehat{\CH}^p(X,n,\DB_{\TW})\ar@<5pt>[l]^-{\psi }}.
  \end{displaymath}
  be the morphisms of Remark \ref{rem:4}. We can define a product
  in $\widehat{\CH}^\ast(X,\ast,\fB)$ by the rule
  \begin{displaymath}
    \alpha \cdot \beta =\psi (\varphi(\alpha )\cdot \varphi(\beta )).
  \end{displaymath}
  This product is still graded commutative, but in general, it is not
  associative. Using $\fB=\DB$ and $n=0$ we recover the original
  arithmetic intersection product of \cite{GilletSoule:ait} through
  Proposition \ref{case0}. 
\end{rmk}

\subsection{Height pairing of cycles with trivial real
  regulator}\label{HigherChow:Product} In this subsection, we assume
that the arithmetic field $F$ is a number field. 

We denote by $Z^{\ast}(X,\bullet)^{0}_{0}$ the subgroup of cycles whose
image by the real regulator to Deligne cohomology is zero. That is
\begin{displaymath}
  Z^{\ast}(X,\bullet)^{0}_{0}=\left\{ Z\in Z^{\ast}(X,\bullet)_{0}
    \,\middle |\, \delta(Z) = 0,\ \caP (Z) \text{ is a boundary.}\right \}
\end{displaymath}
% As defined previously, let $\CH^*(X,\bullet)^0$ denote the subgroup of $\CH^*(X,\bullet)$
% consisting of cycles  whose real regulator image is zero. For an
% element $Z\in Z^{\ast}(X,\bullet)^{0}_{0}$ we will denote by $[Z]$ its
% class in  $\CH^*(X,\bullet)^0$. The map $Z^{\ast}(X,\bullet)^{0}_{0}\to
% \CH^*(X,\bullet)^0$ is surjective.
if $Z\in Z^{p}(X,n)^{0}_{0}$, since $\caP(Z)$ is a boundary, 
there exists a current $g_{Z}$
such that $\mathcal{P}(Z)+d(g_Z)=0$. This current is a Green current
for $Z$. Then $[(Z,g_{Z})]\in
\widehat{\CH}^{p}(X,n)^{0}$, since $\omega (Z,g_{Z})=0$.

For two elements $Z\in Z^{p}(X,n)^{0}_{0}$ and $W\in
Z^{q}(X,m)^{0}_{0}$. we
define
\begin{equation}\label{eq:90}
  Z\bullet W\coloneqq [(Z, g_Z)]\cdot[(W, g_W)],
\end{equation}
where $g_{Z}$ and $g_{W}$ are Green currents satisfying
\begin{equation}\label{eq:91}
  \mathcal{P}(Z)+d(g_Z)=0,\qquad \mathcal{P}(W)+d(g_W)=0.
\end{equation}
If $Z$ and $W$ intersect properly, then
\begin{displaymath}
   Z\bullet W = [(Z\cdot W, g_Z\ast g_W)].
\end{displaymath}
%where $\fg_W$ is a Green form of logarithmic type for $W$ such that $g_W=[\fg_W]+du$.\\

To see that this product is well defined we have to show that this
definition is independent of the choice of Green currents for the
cycles $Z$ and $W$. It boils down to the following lemma.
\begin{lem}\label{HigherChowP}
Let $\eta
\in Z\fD_{\TW}^{2p-n-1}(X,p)$ be a closed form, so we can see the
current $[\eta]$ as
a Green current for the cycle $0\in Z^{p}(X,n)_{00}$, and $[(W,g_{W})]\in
\widehat{\CH}^{q}(X,m)^{0}$, so $\omega (g_{W})=0$.
Then the relations
\begin{align}
  \label{eq:92}
  ([\eta]\ast g_{W})^\sim &=(g_{W}\ast [\eta])^\sim =0\\
  \amap(\eta)\cdot[(W,g_W)]&=[(W,g_W)]\cdot \amap(\eta)=0\label{eq:93}
\end{align}
hold true.
\end{lem}
\begin{proof} Since the cycle zero intersects properly any other
  cycle, the second equation
  follows from the first.

 We compute the product in one direction, choosing a Green form
 $\fg_W$ such that $[\fg_W]\in \widetilde{g}_W$. We obtain 
  \begin{displaymath}
    [\eta]\ast g_{W}
    =\caP(0)\cdot [\fg_{W}]+[\eta]\cdot \omega(g_W),
  \end{displaymath}
up to a boundary. By definition the formal product $\caP(0)\cdot [\fg_{W}]$ is zero. Moreover
  $\omega(g_W)=0$ concluding that $([\eta]\ast \fg_{W})^\sim=0$. On the other
  direction, since $([\eta]\ast g_W)^\sim=(-1)^{nm}(g_W\ast [\eta])^\sim$, we conclude that $(g_W\ast [\eta])^\sim=0$.
%  \begin{displaymath}
%    g_{W}\ast [\eta]=\caP(W)\cdot [\eta]=-(d g_{W})\cdot \eta =
%    -d(g_{W}\cdot \eta),
%  \end{displaymath}
%  as a Green current.
\end{proof}

We can now define an analogue of Beilinson's height pairing for higher
cycles. We take the following notational liberty: $\CH^{\ast}(\Spec(F),
\bullet)$ will be denoted simple by $\CH^{\ast}(F, \bullet)$ when there
is no cause for confusion. For the higher arithmetic Chow groups and
Deligne cohomology of $F$, we will follow the same convention. 
\begin{df}\label{def:17}
Assume that $X$ is equidimensional of dimension $d$. Let
$p,q,m,n\ge 0$  be integers satisfying the relation
$2(p+q-d-1)=n+m$. Let $\pi \colon X\to \Spec F$ be the structural
morphism. For
cycles  $Z\in Z^{p}(X,n)^{0}_{0}$ and $W\in Z^{q}(X,m)^{0}_{0}$, the height pairing
is defined as 
\begin{displaymath}
  \langle Z,W \rangle = \pi _{\ast}(Z\bullet W)\in
  \widehat{\CH}^{p+q-d}(F, n+m)^{0}.
\end{displaymath}
\end{df}
Assume that $Z$ and $W$ intersect properly. Then we can give a formula
for the class of the above height pairing in $\widehat{\CH}^{p+q-d}(F,
n+m)^{0}\otimes \Q$.
Since $n+m=2(p+q-d-1)$ is even, we know that the group $\CH^{p+q-d}(F,
n+m)$ is torsion.  We
have the exact sequence 
\begin{equation}\label{eq:1729}
  0\to 
  \frac{H^1_{\DB}(F, \R(p+q-d))}{(\text{im}(\rho_{\text{Be}}))}\to
  \widehat{\CH}^{p+q-d}(F, n+m)\to
  \CH^{p+q-d}(F, n+m)\to 0, 
\end{equation}
where $\rho_{\text{Be}}$ denotes Beilinson's regulator. The pairing
$\langle Z,W\rangle$ is given, at the level of higher arithmetic cycles,
by $[\left(\pi_{\ast}(Z\cdot W), \pi_{\ast}(g_Z\ast
  g_W)\right)]$. Since the cycle $\pi_{\ast}(Z\cdot W)\in
\CH^{p+q-d}(F, n+m)$ is torsion, some integral multiple of it is the
boundary of a pre-cycle. Formally,
let $N$ be the order of $\pi_{\ast}(Z\cdot W)$, then there is a  $T\in
Z^{p+q-d}(F, n+m+1)_0$ such that 
\begin{displaymath}
 N(\pi_{\ast}(Z\cdot W))=\delta(T).
\end{displaymath}
Consider the element
\begin{displaymath}
  N[\left(\pi_{\ast}(Z\cdot W), \pi_{\ast}(g_Z\ast
  g_W)\right)]-[\delta(T), \mathcal{P}(T)].
\end{displaymath}
This element represents the same class as $N\langle Z, W\rangle$, but
is of the form
\begin{displaymath}
  \amap([N\pi_{\ast}(g_Z\ast g_W)-\mathcal{P}(T)]).
\end{displaymath}

The class of the current $\left(N\pi_{\ast}(g_Z\ast
  g_W)-\mathcal{P}(T)\right)^\sim$ belongs to the group 
$\widetilde{\fD}^{1}_{\TW,D}(F,p+q-d)$. Since
$\fD_{\TW,D}^{2}(F,p+q-d)=0$, this current 
is closed and gives us a class
\begin{displaymath}
  [N\pi_{\ast}(g_Z\ast g_W)-\mathcal{P}(T)]\in H^1_{\DB}(F,
  \R(p+q-d)). 
\end{displaymath}
Thus, we the rational height pairing is given by
\begin{multline*}
  \langle Z, W\rangle_{\Q} =
\frac{1}{N}\left(\overline{N\pi_{\ast}(g_Z\ast
    \fg_W)-\mathcal{P}(T)}\right)\\
\in
\frac{H^1_{\DB}(F, \R(p+q-d))}{(\text{Im}(\rho_{\text{Be}}))}\otimes \Q=
  \widehat{\CH}^{p+q-d}(F, n+m)_{\Q}.
\end{multline*}
Here $\overline{\phantom{A} }$ denotes the class in $H^1_{\DB}(F,
\R(p+q-d))/\text{Im}(\rho_{\text{Be}})$.\\

The height pairing defined above has two components, the one coming from
intersection theory, $\pi_{\ast}(Z\cdot W)$, that we have been able to
write in term of $\mathcal{P}(T)$ and the one coming from the Green
currents $\pi_{\ast}(g_Z\ast 
g_W)$.  The fact that we have written the intersection
theoretical part as $\mathcal{P}(T)$ used that some higher
Chow groups of the ground field are torsion. Of course, if the ground
field is not a number field, this may not be possible. However, the
component $\pi_{\ast}(g_Z\ast 
g_W)$ can be defined for any arithmetic field. For instance when
$F=\C$,  since all that we have used to define it involves complex 
geometry. In fact the term $\pi_{\ast}(g_Z\ast g_W)$
resembles the Archimedean component of Beilinson's height pairing. By
the same reasons as before, the current $\pi_{\ast}(g_Z\ast g_W)$ is
closed. 

\begin{df}\label{arcdf} Let $Z\in Z^{p}(X,n)^{0}_{0}$ and $W\in
  Z^{q}(X,m)^{0}_{0}$ be cycles intersecting properly, then the
  \emph{Archimedean component of the higher height pairing} is defined
  as
  \begin{displaymath}
   \langle Z, W\rangle _{\infty }\coloneqq \left(\pi_{\ast}(g_Z\ast
   g_W)\right)^\sim \in H^1_{\DB}(F, p+q-d),  
  \end{displaymath}
  for any choice of Green current $g_{Z}$ for $Z$ and a Green current $g_{W}$ for $W$.
  \end{df}
%We clarify the choice of $\fg_W$ here- we can always choose a Green current $g_W$ satisfying \eqref{eq:91}, and since there exist a Green form of logarithmic type $\fg_W$ such that $g_W=[\fg_W]+du$, we get $d([\fg_W])+\mathcal{P}(W)=d(g_W-du)+\mathcal{P}(W)=0$.
\begin{prop}\label{commarc}
The Archimedean component $\langle Z,
W\rangle_{\infty}$ of the height pairing, does not depend on the choice of Green
currents $g_{Z}$ and $g_{W}$ satisfying $\omega (g_{Z})=0$ and
$\omega (g_{W})=0$.
\end{prop}
\begin{proof}
  This proposition is a direct consequence of equation \eqref{eq:92}
  in Lemma \ref{HigherChowP}, and the details are omitted.
\end{proof}

\begin{rmk}
  The higher height pairing $\langle Z,W \rangle$ only depends on the
  classes $[Z],\,[W]\in \CH^{\ast}(X,\bullet)^{0}$. Indeed, if
  $W'=W+\delta T$, then $g_{W'}=g_{W}-\caP(T)$ is a Green current for
  $W'$ satisfying the analogue of conditions \eqref{eq:91}. Thus
  \begin{displaymath}
    Z\bullet W'=[(Z,g_{Z})]\cdot [(W',g_{W'})]=[(Z,g_{Z})]\cdot
    [(W,g_{W})]=Z\bullet W. 
  \end{displaymath}
By contrast the Archimedean component of the height pairing $\langle
Z, W\rangle _{\infty }$ depends on the cycles $Z$ and $W$.
\end{rmk}

\begin{rmk}
The Archimedean part of Beilinson's height pairing has two other
incarnations. One via the product in real-Deligne cohomology, and
another via the theory of biextensions of mixed Hodge structures of
weight $-1$ (see \cite{Hain_ArcHeight} for details). It would be
interesting to study if such incarnations are possible for the
Archimedean component of the higher height pairing. This will be the subject of a
future project.  
\end{rmk}

\section{Example: the case of dimension zero}
\label{sec:examples}

% We begin this section with a discussion: Suppose that for a smooth projective variety $X$ of dimension $d$, over a number field $F$, we have a complex $\mathcal{D}^{\star }(X,\bullet )$ which computes the Absolute Hodge/Deligne-Beilinson cohomology of $X$ (in most cases given by the smooth differential forms on $X$ or its dual, the currents). Also, suppose that there is a morphism of complexes (in the derived category sense)
% $$\mathcal{P}\colon Z^p(X,\star)_0\rightarrow \mathcal{D}^{2p-\star }(X,p),$$
% which computes the Beilinson regulator. Then we can always define a higher arithmetic Chow group $\widehat{CH}^{p}_{\mathcal{D}}(X,\star)$, using the modified homology of the corresponding diagram. Depending on the structure of $\mathcal{D}^{\star }(X,\bullet )$, these higher arithmetic Chow groups will vary.

Let $F$ be a number field and $X$ a smooth projective variety over $F$.
For each of the complexes $\fA$ in diagram \eqref{eq:43} we have
defined higher arithmetic Chow groups $\widehat{\CH}^{p}(X,n,\fA)$
equipped with direct images, inverse images and products. In
the case of $\widehat{\CH}^{p}(X,n,\fDTW)$ the product is 
associative and graded commutative. For the other complexes,
the product fails to be associative in general. In this section we
specialize to the case $X=\Spec F$ and derive the first non trivial 
examples of a higher arithmetic intersection pairing. For simplicity in the
notation, when $X=\Spec F$ we will write
\begin{displaymath}
  H^{\ast}_{\fD}(F,\R(\ast))\coloneqq
  H^{\ast}_{\fD}(X,\R(\ast)),
\end{displaymath}
and similarly, $\CH^{p}(F,n)$ and $\widehat{\CH}^{p}(F,n)$. 

\subsection{Higher Arithmetic Chow groups of a number
  field}\label{dimension:Zero}

We now look for the groups
$\widehat{\CH}^p(F,n,\fA)$ tha may be non-zero. To simplify the
discussion, we will neglect 
torsion and work after tensoring with $\Q$.

Let $r_1$ and $r_2$ denote the number of real and (non-conjugate)
complex embeddings of $F$ respectively. 

We start by writing down the exact sequence of Theorem
\ref{thm:3}~\eqref{item:27} 
in case of $X=\Spec(F)$ and after tensoring with $\Q$:  
\begin{multline}\label{eqnF}
\CH^p(F,n+1)_{\Q}\xrightarrow{\rho _{Be}}\fA
^{2p-n-1}(F,p)/\im(d)\xrightarrow{\amap}\\
\widehat{\CH}^p(F,n,\fA)_{\Q}\xrightarrow{\zeta}\CH^p(F,n)_{\Q}\rightarrow
0.
\end{multline}

By Borel's Theorem (see for instance \cite[Theorem 9.9 ]{Burgos:bb}),
after applying the isomorphism between higher Chow groups and motivic
cohomology, we have that
\begin{equation}\label{eq:63}
    \CH^{p}(F,n)_{\Q}=
  \begin{cases}
    \Q,& \text{ if }p=n=0.\\
    F^{\times}\underset{\Z}\otimes \Q,&  \text{ if }p=n=1,\\
    \Q^{r_{1}+r_{2}},&  \text{ if }2p-n=1,\ p \text{ odd,}\\
    \Q^{r_{2}},&  \text{ if }2p-n=1,\ p \text{ even,} \\
    0,& \text{ otherwise.}
  \end{cases}
\end{equation}

We treat first the simplest case, when $\fA=\fD$. Note that
\begin{equation}\label{eq:64}
  \fD^{n}(F,p)=
  \begin{cases}
    \R^{r_{1}+r_{2}},&\text{ if }n=p=0,\\
    \R(p-1)^{r_{1}+r_{2}},&\text{ if }n=1, p\text{ odd,}\\
    \R(p-1)^{r_{2}},&\text{ if }n=1,  p\text{ even,}\\
    0,&\text{ otherwise.}
  \end{cases}
\end{equation}
In particular, $\fD^{1}(F,p)=H^{1}_{\fD}(F,\R(p))$.

Using the exact sequence \eqref{eqnF} together with \eqref{eq:63} and
\eqref{eq:64} we deduce that
\begin{equation}
  \label{eq:65}
  \widehat{\CH}^{p}(F,n,\fD)_{\Q}=
  \begin{cases}
    \CH^{0}(F)_{\Q},&\text{ if }p=n=0,\\
    \CH^{p}(F,n)_{\Q},&\text{ if }2p-n=1,\ p>0,\\
    H^{1}_{\fD}(F,\R(p))/\im(\rho _{Be,\Q}), &\text{ if }2p-n=2,\ p>0,\\
    0,&\text{ otherwise.}
  \end{cases}
\end{equation}

We now turn our attention to $\fA=\fD_{TW}$. This case is more
complicated because
\begin{displaymath}
  \fD^{n}_{TW}(F,p)\not =\{0\} \Longleftrightarrow p=n=0, \text{ or
  }p>0, n=0,1.
\end{displaymath}
In this case we obtain that $\widehat{\CH}^{p}(F,n,\fDTW)_{\Q}$ satisfies,
\begin{enumerate}
\item for $p=n=0$, $\widehat{\CH}^{0}(F,0,\fDTW)_{\Q}=\CH^{0}(F)_{\Q}=\Q$; 
\item for $p>0$, $n=2p-1$, there is a short exact sequence
  \begin{equation}\label{item:25}
    0\to \fD_{TW}^{0}(F,p)\to \widehat{\CH}^{p}(F,n,\fDTW)_{\Q}\to
    \CH^{p}(F,n)_{\Q}\to 0;
  \end{equation}
\item for $p>0$, $n=2p-2$;
  \begin{displaymath}
    \widehat{\CH}^{p}(F,2p-2,\fDTW)_{\Q}=H^{1}_{\fD}(F,\R(p))/\im(\rho _{Be,\Q});
  \end{displaymath}
\item for other values of $p$ and $n$, is zero.
\end{enumerate}

The groups $\widehat{\CH}^p(F, n)^0$ do not depend
on  the complex used to  compute the cohomology. They are
given by the kernel of the Beilinson regulator map.
\begin{enumerate}
\item for $p=n=0$,
  $\widehat{\CH}^{0}(F,0,\fDTW)_{\Q}=\CH^{0}(F)^{0}_{\Q}=\{0\}$; 
\item for $p>0$ and  $n=2p-2$,
\begin{displaymath}
  \widehat{\CH}^p(F, 2p-2)^0=\widehat{\CH}^p(F,
  2p-2)=H^{1}_{\fD}(F,\R(p))/\im(\rho _{Be,\Q}); 
\end{displaymath}
\item for $p>1$ and  $n=2p-1$, the groups $\widehat{\CH}^p(F, 2p-1)^0$ are torsion by
Borel's theorem; 
\item For $p=1$, the group $\widehat{\CH}^1(F,1)^0$ is
given by the kernel of the logarithm map
\begin{displaymath}
  \log|\cdot|\colon F^{\times}\rightarrow \R^{r_1+r_2},\quad x\mapsto
(n_{\sigma }\log|x|_{\sigma})_{\sigma : F\hookrightarrow \C},
\end{displaymath}
where $n_{\sigma } =1$ if $\sigma$ is a real embedding and $n_{\sigma } =2$ if
$\sigma$ is a complex embedding.
\end{enumerate}
   As an example, for $F=\Q(i)$, where
$i$ is a solution of the equation $x^2+1=0$, $\widehat{\CH}^1(F,
1)^0$ consists of all elements $\frac{a}{c}+i\frac{b}{c}\in \Q(i)$, such that $a^2+b^2=c^{2}$
(Pythagorean triplets).

\begin{rmk}\label{rem:13}
 Let $Z\in Z^{p}(F,2p-1)_{0}$. Since
$X=\Spec(F)$ has dimension zero, then
$\fD_{\TW,D}^{1}(F,p)= \fD_{\TW}^{1}(F,p)$. Therefore $\caP(Z)$ already belongs
to $\fD_{\TW}^{1}(F,p)$, and, in consequence, we can take $0$ as the
Green current for $Z$, with the result that  $\omega _{Z}=\caP(Z)$. We
note that, although it is the zero 
Green current it is not trivial, in the sense that the corresponding
Green form may be non-zero. We caution the reader also that this in not a splitting of
the exact sequence \eqref{item:25} because, if $Z_{1},\ Z_{2}\in
Z^{p}(F,2p-1)_{0}$ are two cycles with the same class in
$\CH^{p}(F,2p-1)$, then $(Z_{1},0)$ and $(Z_{2},0)$ may give different
classes in $\widehat{\CH}^{p}(F,2p-1,\fDTW)$.
\end{rmk}

\subsection{Examples of intersection products in dimension zero}
\label{sec:exampl-inters-prod}

For $X=\Spec F$, the diagonal map $X\to X\times X$
is an isomorphism and there is no difference between the external and
the intersection product. In particular all (admissible) pre-cycles
intersect properly. 

\begin{rmk}\label{rem:14}
  From the computations on the previous section we observe that the
  only non trivial products among higher arithmetic Chow groups in the
  case of number fields are the product that involve $\widehat{\CH}^{0}(F,0)$
  and the products of the form 
  \begin{multline*}
    \widehat{\CH}^{p}(F,2p-1,\fDTW)_{\Q}\otimes
    \widehat{\CH}^{q}(F,2q-1,\fDTW)_{\Q}\to\\
    \widehat{\CH}^{p+q}(F,2(p+q)-2,\fDTW)_{\Q}. 
  \end{multline*}

  If we start with two classes, one in $\widehat{\CH}^p(F, 2p-1)^0$
  and one  in $\widehat{\CH}^q(F, 2q-1)^0$ then the intersection product in 
  $\widehat{\CH}^{p+q}(F,2(p+q)-2,\fDTW)_{\Q}$ is always zero.
  Indeed, for $p>1$ or $q>1$,
  we know that at least one of the groups
   $\widehat {\CH}^p(F,2p-1)^0 $ or $\widehat{\CH}^q(F,2q-1)^0$ 
is torsion, and so will be the intersection product. While for $p=q=1$,
the groups $\widehat{\CH}^1(F,1)^0$ are not torsion. Since
$\widehat{\CH^1(F,1)^0}\subset F^{\times}$ is the kernel of the
logarithm map, we can consider $\{\alpha\},\{\beta\} \in
\CH^1(F,1)^0$, with $\alpha ,\,\beta \in F^{\times}$ and $\alpha$, $
\beta$ in the kernel of the regulator. Choose and 
embedding $\sigma\colon F\to \C $. After enlarging $F$ if needed, we
can assume that $F$ is stable under complex conjugation and let $h$
be the automorphism of
$F$ induced by complex conjugation. Since $|\sigma
(\alpha )|=|\sigma (\beta  )|=1$, then 
$$h^{\ast}\{\alpha \}=\{\alpha^{-1}\},\quad
h^{\ast}\{\beta\}=\{\beta^{-1}\}.$$ 
Identifying $\{\alpha^{-1}\}$ with $-\{\alpha\}$ and
$\{\beta^{-1}\}=-\{\beta\}$ in $CH^1(F,1)$ (which of course, amounts
to adding a boundary $\delta T$. But as we will see in Lemma \ref{lemm:5}, all
such boundaries have trivial image under $\caP$) we deduce
\begin{displaymath}
 h^{\ast}(\langle \{\alpha\},\{\beta\}\rangle )=\langle
 \{\alpha\},\{\beta\}\rangle  \in \widehat{\CH}^2(F, 2). 
\end{displaymath}
Since any element $\gamma\in H^1_{\DB}(F, \R(2))$ satisfies
$h^{\ast}(\gamma)=-\gamma$, for parity reasons, the product $\langle
\{\alpha\},\{\beta\}\rangle=0$. 
\end{rmk}

By Remark \ref{rem:14}, if we want to have a non-zero higher
arithmetic intersection product, we need to look at cycles whose image
by the regulator map is non-zero.

As observed in Remark \ref{rem:13}, the zero dimensional case is
special because for a cycle $\alpha \in Z^p(F,
2p-1)_{0}$ we can always choose the Green current $0$. Therefore we
define a pairing
\begin{multline*}
(\ ,\ )_{p,q}\colon   Z^{p}(F,2p-1)_0\otimes
    Z^{q}(F,2q-1)_0\to\\
    \widehat{\CH}^{p+q}(F,2(p+q)-2,\fDTW)_{\Q}=H^{1}_{\fD}(F,\R(p+q))/\im(\rho
    _{Be,\Q}).
\end{multline*}
given, for $\alpha \in Z^p(F,
2p-1)_0$ and $\beta  \in Z^q(F,
2q-1)_0$, by 
\begin{displaymath}
  (\alpha ,\beta )_{p,q}=[(\alpha ,0)]\cdot [(\beta ,0)].
\end{displaymath}

The next lemma
will be  useful in concrete computations.

\begin{lem}\label{0-0lemma}
  Let $\alpha \in Z^p(F,
2p-1)_0$ and $\beta  \in Z^q(F,
2q-1)_0$, then
  the product $[(\alpha ,0)]\cdot [(\beta ,0)]$
is given by $[(\alpha\cdot \beta, 0)]$.
\end{lem}
\begin{proof}
  Let $\fg_{\beta }=\{g'_{2q-1},\cdots ,g'_0\}$ be a Green form for $\beta $ such that the
  corresponding Green current satisfies $[\fg_{\beta }]=0$. In other words
 \begin{displaymath}
  \sum_{j=0}^{2q-1}(p_{j})_\ast(g'_j\cdot W_j)=0.
\end{displaymath} 
 The lemma
  amounts to showing that, writing $g_{\alpha }=0$ and $\omega _{\beta
  }=\caP(\beta )$, 
  \begin{displaymath}
    g_{\alpha } \ast \fg_{\beta }= (-1)^{2p-1}\mathcal{P}(\alpha)\cdot
    [\fg_{\beta}]+g_{\alpha}\cdot \omega_{\beta}=0.
  \end{displaymath}
  Clearly,  $g_{\alpha}\cdot \omega_{\beta}=0$. We now denote by
  \begin{displaymath}
    p_{k}\colon \square^{k}\to \Spec(F)
  \end{displaymath}
  the structural morphism.
  From the definition
 \begin{displaymath}
   \mathcal{P}(\alpha)\cdot [\fg_{\beta}]=
   \sum^{2q-1}_{j=0}(p_{2p+j-1})_{\ast}(p_{2p-1}^{\ast}(\delta_{\alpha}\cdot W_{2p-1})\cdot
   p_{j}^{\ast}(g'_j\cdot W_j)).
 \end{displaymath}
Since $\mathcal{P}(\alpha )$ is smooth, using the projection formula
for forms and currents, we can write
\begin{displaymath}
  \mathcal{P}(\alpha)\cdot [\fg_{\beta}]=\caP(\alpha )\cdot
  \left(\sum^{2q-1}_{j=0}(p_{j})_{\ast}(g'_j\cdot W_j)\right)=0.
\end{displaymath}
This concludes the proof.
\end{proof}

Fix $\alpha $ and $\beta $ as in the previous lemma. we want to devise
a method to compute $(\alpha ,\beta )_{p,q}$. Since the group 
$\CH^{p+q}(F,2(p+q)-2)$ is torsion, there
exist an integer $N>0$ and
a pre-cycle $T\in Z^{p+q}(F,2(p+q)-1)_0$, satisfying
$N(\alpha\cdot\beta)= \delta(T)$. Therefore, in the group
$\widehat{\CH}^{p+q}(F,2(p+q)-2)_{\Q}$, by Lemma \ref{0-0lemma} we
derive the equation  
\begin{displaymath}
  (\alpha ,0)\cdot (\beta ,0)=
  (\alpha \cdot \beta ,0)=\frac{1}{N}(\delta
  (T),0)=\frac{1}{N}(0,-\caP(T)). 
\end{displaymath}
Using the isomorphism 
\begin{displaymath}
  \widehat{\CH}^{p+q}(F,2(p+q)-2)_{\Q}\simeq
  \left[H^{1}_{\fD}(F,\R(p+q))/\im(\rho _{Be})\right]\otimes \Q, 
\end{displaymath}
we obtain a formula for the intersection product
\begin{equation} \label{eq:66}
  (\alpha ,\beta )_{p,q}=(\alpha ,0)\cdot (\beta ,0)=
  \frac{1}{N}(\overline{-\caP(T)}).
\end{equation}
Here $\overline{-\caP(T)}$ means the class of $-\caP(T)$ in
$\left[H^{1}_{\fD}(F,\R(p+q))/\im(\rho _{Be})\right]\otimes \Q$.\\

% In particular, if we pick $\alpha \in CH^p(F, 2p-1)^0$ and $\beta\in
% CH^q(F, 2q-1)^0$, the recipe above defines the height pairing
% \begin{displaymath}
%   \langle \alpha,\beta \rangle_{p+q-1}=\frac{1}{N}\overline{\mathcal{P}(T)},
% \end{displaymath}
% for any $T\in Z^{p+q}(F, 2(p+q)-1)_0$, satisfying
% $N(\alpha\cdot\beta)=\delta(T)$. Here $N$ is the order of the element
% $\alpha\cdot\beta$ in $CH^{p+q}(F, 2(p+q)-2)$.
The pre-cycle $T$ is equation \eqref{eq:66} can be computed in many
cases as a linear combination of pre-cycles that explain known
relations in higher Chow groups, like the Totaro pre-cycle that
explains the Steinberg relations. Among them, there are the pre-cycles
that explain the multilinearity relations. The next lemma shows that we
do  not need to worry about them.

\begin{lem}\label{lemm:5}
  Let $\alpha,\ \beta \in F^{\times}\setminus \{1\}$ and $Z\in
  Z^{p}(F,2p-1)_{0}$ a cycle. Let $T\in Z^{p+1}(F,2p+1)_0$ be a pre-cycle
  such that
  \begin{equation}\label{eq:67}
    \delta T=\{\alpha\beta \} \times Z - \{\alpha\} \times Z -\{\beta
    \} \times Z.  
  \end{equation}
  Then $\overline {\caP(T)}=0$.
\end{lem}
\begin{proof}
  Let $T'$ be another pre-cycle satisfying \eqref{eq:67}. Then $T-T'$
  is a cycle, and
  \begin{displaymath}
    \caP(T)-\caP(T')=\caP(T-T')\in \im(\rho _{Be,\Q}).
  \end{displaymath}
  Thus if we prove the lemma for a particular choice of pre-cycle, it
  is also true for any other pre-cycle satisfying \eqref{eq:67}.
  Consider the pre-cycle  $T=C\times Z\subset \square^{2p+1}$, where
  $C\subset \square^{2}$ is the curve parameterized by
  \begin{displaymath}
    t\mapsto \left(t, \frac{(t-\alpha )(t-\beta )}{(t-1)^2}\right).
  \end{displaymath}
  It is straigtforward to check that, being $Z$ a cycle in the
  normalized complex, then $T\in Z^{p+1}(F,2p+1)_{0}$ satisfies
  \eqref{eq:67}.
  Using the multiplicativity of the Thom-Whitney complex we deduce.
  \begin{displaymath}
    \caP(T)=\caP(C)\cdot \caP(Z).
  \end{displaymath}
Finally, $\caP(C)=0$ by Lemma \ref{lemm:19}, and we obtain this result. 
\end{proof}
\begin{rmk}\label{ambiguity}
The pairing $(\ ,\ )_{p,q}$ above is a pairing between cycles and does
not descend to a pairing of higher Chow groups. There is an obvious
problem: Suppose $Z\in Z^p(F, 2p-1)_0$ is a boundary, i.e., there
exists a (pre)-cycle $Z'\in Z^p(F, 2p)_0$, such that
$Z=\delta(Z')$. Now, by assigning the $0$ Green current to $Z$, the
height pairing for any cycle $W\in Z^q(F, 2q-1)_0$ is given by 
$$(Z, W)_{p,q}=\overline{-\mathcal{P}(Z'\cdot W)},$$
which for appropriate choices of $W$, may be non-trivial (even after
taking the class mod the image of the regulator).
\end{rmk}

%For $p,q > 1$, the groups $\widehat{CH}^p(F, 2p-1)^0$ are torsion, as mentioned in the discussions in section \ref{dimension:Zero}. Thus, the height pairing would be torsion (and rationally, zero). One can of course try to find a torsion element, which is non-zero. But for now, we fix $p=q=1$, and consider height pairing for two elements
%$$\widehat{\alpha}\coloneqq  [(\alpha, 0)], \widehat{\beta}\coloneqq [(\beta, 0)] \in \widehat{CH}^1(F, 1)^0\subset CH^1(F,1).$$
%We have
%$$\overline{\widehat{\alpha}}=-\widehat{\alpha},\quad \overline{\widehat{\beta}}=-\widehat{\beta}.$$
%Hence $\overline{\widehat{\alpha}\cdot \widehat{\beta}}=\widehat{\alpha}\cdot \widehat{\beta}$ in $CH^2(F, 2)$. Whereas, any element $\gamma\in H^1_{\DB}(F, \R(2))$ satisfies $\overline{\gamma}=-\gamma$. Thus, for parity reasons, the height pairing would still be trivial.\\
At the level of cycles, we can compute specific examples of this
intersection pairing $(\ ,\ )_{p,q}$. We do so in two
situations: the case $p=q=1$ and the case $p=1$, $q=2$.

\bigskip
\textbf{Case $p=q=1$.}
\vspace{0.5cm}
We want to compute examples of the pairing
\begin{displaymath}
  (\ ,\ )_{1,1}\colon Z^1(F, 1)_0\otimes Z^1(F,
1)_0\rightarrow \frac{H^1_{\DB}(F, \R(2))}{\im(\rho_{Be})}\otimes \Q.
\end{displaymath}
given by
\begin{displaymath}
  (\alpha ,\beta )_{1,1}=(\alpha ,0)\cdot(\beta ,0)
\end{displaymath}

We start with the case $\beta =1-\alpha $. As we will see later, this
is the main ingredient of the general case.

It is well known and easily verifiable that $\alpha \cdot (1-\alpha
)=\delta (C_{\alpha })$, where
\begin{displaymath}
  C_{\alpha}\coloneqq  \left(z, 1-\frac{\alpha}{z}, 1-z\right),
\end{displaymath}
is the Totaro pre-cycle. Thus the height pairing is given by
\begin{displaymath}
  (\alpha, 1-\alpha )_{1,1}=\overline{-\mathcal{P}(C_{\alpha})}.
\end{displaymath}

We can identify
\begin{displaymath}
H^1_{\DB}(F, \R(2))=\R(1)^{r_{2}},  
\end{displaymath}
where $r_{2}$ is the number of non-equivalent complex immersions of $F$. Let
$\caL_{2}$ be the Bloch-Wigner dilogarithm function.  By
\cite[Theorem 3.6]{Goncharov:cpr} (see also \S 2 and \S 3 of
\emph{loc. cit.})
\begin{equation}\label{eq:69}
  \caP(C_{\alpha })=-\frac{1}{2\pi}\left(i\caL_{2}(\sigma _{1}(\alpha )),\dots,
  i\caL_{2}(\sigma _{r_{2}}(\alpha )\right).
\end{equation}

We specialize to the case $F=\Q(i)$. As we will recall in the next
section, by  \cite{Petras:dilog}, we know that
$i\mathcal{L}_2(i)$ (up to factors of $\pi$) generates
$$\im (\rho_{Be,\Q}\colon \CH^2(F,
3)_{\Q}\rightarrow H^1_{\DB}(F, \R(2))_{\Q}).$$
It will be interesting to prove
that, for certain $\alpha\in \Q(i)^{\times }\setminus \{1\}$, the
value $\mathcal{L}_2(\alpha)$ is not a (rational) 
multiple of $\mathcal{L}_2(i)$. This would be an example of the
non-triviality of the defined intersection pairing.
By a conjecture of Milnor (stated in
\cite{Milnor:hypgeom}, end of page 21), such a $\alpha$ exists. 

\medskip
Now, let $\alpha, \beta$ in $Z^1(F,1)$. We want to reduce the
computation of $(\alpha ,\beta )_{1,1}$ to the previous case.
Using that
$K_{2}(F)=F^{\times} \otimes F^{\times}/\sim$, where $\sim$ is
generated by the Steinberg relations, is torsion  and that 
$\CH^{2}(F,2)_{\Q}=K_{2}(F)_{\Q}$, we know that there is an integer $N>0$ such
that $\alpha\times \beta \in Z^2(F,2)_{\Q}$ can be written as
\begin{displaymath}
  \alpha\times  \beta= \frac{1}{N}\sum _i\gamma_i\times (1-\gamma_i)+
  \sum_j\delta(T_j), 
\end{displaymath}
where the pre-cycles $T_j$ are of the form discussed in Lemma
\ref{lemm:5}. By Lemma \ref{lemm:5}, the height pairing $(\alpha, 
\beta)_{1,1}$ is given by
\begin{displaymath}
  (\alpha, \beta)_{1,1}= \sum_i\overline{-\caP(C_{\gamma_i})}. 
\end{displaymath}

To simplify the computation latter we need another lemma.

\begin{lem}\label{lemm:7} Let $\alpha,\ \beta \in F^{\times }\setminus
  1$.
  \begin{enumerate}
  \item For any pre-cycle $T\in Z^{2}(F,3)_{0}$ with
    \begin{equation}
      \label{eq:68}
      \delta T=\{\alpha \}\times \{-\alpha \},
    \end{equation}
    we have $\overline{\caP(T)}=0$.
  \item For any pre-cycle $T\in Z^{2}(F,3)_{0}$ with
    \begin{equation}
      \label{eq:70}
      \delta T=\{\alpha \}\times \{\beta  \}+\{\beta  \}\times \{\alpha \},
    \end{equation}
    we have $\overline{\caP(T)}=0$.
  \end{enumerate}
\end{lem}
\begin{proof}
  By the same argument as in the proof of Lemma \ref{lemm:5}, it is
  enough to prove the result for one single $T$ for each
  equation. For the first statement we use that
  \begin{displaymath}
    \alpha \otimes (-\alpha )=\alpha \otimes (1-\alpha )+\alpha
    ^{-1}\otimes 1-\alpha ^{-1},
  \end{displaymath}
  thanks to the identity $-\alpha =(1-\alpha)/(1-\alpha ^{-1}) $.
  Therefore
  \begin{displaymath}
    \{\alpha\} \times \{-\alpha \}=\delta C_{\alpha }+  \delta
    C_{\alpha ^{-1}} +\delta  T,  
  \end{displaymath}
  where $T$ is a linear combination of pre-cycles of the form
  considered in Lemma \ref{lemm:5}.  Using Lemma \ref{lemm:5},
  equation \eqref{eq:69} together with the fact that
  $\caL_{2}(u^{-1})=-\caL_{u}(u)$, we obtain the first statement.

  The second statement follows from the first statement  and Lemma
  \ref{lemm:5}, using the identity
  \begin{displaymath}
    \alpha \otimes \beta + \beta \otimes \alpha =
    \alpha \beta \otimes (-\alpha \beta )-\alpha \otimes (-\alpha )-
    \beta \otimes (-\beta ).
  \end{displaymath}
\end{proof}

We ilustrate this procedure with a prototypical example. Again we
consider the case $F=\Q(i)$ and let $\alpha =2+3i$ and $\beta
=1-2i$. We want to compute $(\alpha ,\beta )_{1,1}$. Using the identities
\begin{align*}
  1-(-1-i)^6&=(2+i)(2+3i),\\
  1-(2+3i)&=i(-1-i)(1-2i),
\end{align*}
one can check that, in $F^{\times }\otimes F^{\times}$,
the following relation holds true.
\begin{align*}
  12\, \left(\alpha \otimes \beta\right) = -&12\, \left( (-1-i)\otimes (1-(-1-i))\right)+\\
                             &2\, \left((-1-i)^{6}\otimes (1-(-1-i)^6)\right) +\\
                             &12\, \left((2+3i)\otimes (1-(2+3i))\right)+\dots,
\end{align*}
where $\dots$ contains terms of the form $u\otimes v+v\otimes u$. Using
lemmas \ref{lemm:5}, \ref{lemm:7}, and equation \eqref{eq:69}, we deduce that
\begin{displaymath}
  \frac{1}{i}(\alpha ,\beta )_{1,1}=-\frac{1}{2\pi}\caL_{2}(-1-i)
  +\frac{1}{12\pi}\caL_{2}((-1-i)^6)+\frac{1}{2\pi}\caL_{2}(2+3i),
\end{displaymath}
modulo of course, the image of the regulator.

% We illustrate this with a prototypical example: Let $a_1, a_2$ be elements in $Z^1(F,1)$. Then we can write
% $$a_1\cdot a_2= (a_1)+ (a_2)+ \delta \left(\frac{(t-a_1)(t-a_2)}{(t-1)^2}, t\right),$$
% identifying the multiplicative structure of $F^{\times}$ with the additive structure of $CH^1(F,1)$. Hence, if we look at the product $\alpha\cdot \beta$, a typical boundary $T_j$ will have the form
% $$\left(\alpha , \frac{(t-a_1)(t-a_2)}{(t-1)^2}, t\right),$$
% albeit the fact that there are always more than one way to write the boundary.\\

\bigskip
\textbf{Case $p=1, q=2$, $F=\Q(i)$}.
\bigskip
We consider the higher cycles
$$i\in Z^1(F,1)_0,\quad Z_{i}\in Z^2(F, 3)_0,$$
where
$$Z_i\coloneqq 4\left(z, 1-\frac{i}{z}, 1-z\right)- \left(z, \frac{(z-i)^4}{(z-1)^4}, 1-z\right).$$
By \cite{Petras:dilog}, the element $[Z_i]$ is a generator of the free
part of $\CH^2(F,3)$. Therefore, $\caP(Z_i)=i\mathcal{L}_2(i)$ generates
$$\im \left(\rho_{Be,\Q}\colon \CH^2(F,
3)_{\Q}\rightarrow H^1_{\DB}(F, \R(2))_{\Q}\right).$$

Further, the intersection $i\cdot Z_{i}$ is given
by 
$$i\cdot Z_i= \delta \left(\underbrace{4(C'_i-C''_i)-\Xi_{i}}_{\in Z^3(F, 5)_0}\right),$$
where
$$C'_i=\left(z_2, 1-\frac{i}{z_2},z_1, 1-\frac{z_2}{z_1},1-z_1\right),$$
$$C''_i=\left(z_1, 1-\frac{i}{z_2},z_2, 1-\frac{z_2}{z_1},1-z_1\right),$$
and
$$\Xi_i=\left(z_2, 1-\frac{i}{z_2}, z_1, \frac{(z_1-i)^4}{(z_1-1)^4}, 1-z_2\right).$$
The cycles $C'_i$, $C''_i$ are higher analogues of the Totaro cycle. They are closely related to Goncharov's programme of associating the Bloch group to the higher Chow group of points (see \cite{Goncharov:prAmc} for details).\\
Hence, the intersection pairing is given by
$$(i, Z_i)_{1,2}= \overline{-(4\mathcal{P}(C'_i-C''_i)- \mathcal{P}(\Xi_i))}.$$
Using Lemma \ref{lemm:19}, now for the curve $C=\left(z_1, \frac{(z_1-i)^4}{(z_1-1)^4}\right)$, and the multiplicativity of the Thom-Whitney complex, it is easy to show that $\mathcal{P}(\Xi_i)=0$.
%\begin{claim}\label{keyclaim}
%$\mathcal{P}(\Xi_i)=0$.
%\end{claim}
%\begin{proof}
%By definition,
%$$\mathcal{P}(\Xi_i)=\frac{1}{(2\pi i)^2}\int_{\Xi_i}W_5.$$
%Using the fact that $W_n$ is multiplicative, we write
%$$\int_{\Xi_i}W_5=\left(\int_{\Gamma_1}W_2\right)\cdot \left(\int_{\Gamma_2}W_3\right),$$
%where $\Gamma_1=\left(\frac{(z_1-i)^4}{(z_1-1)^4}, z_1\right)$ and $\Gamma_2=\left(1-\frac{i}{z_2}, 1-z_2, z_2\right)$.
%Let
%$$f=\frac{(z_1-i)^4}{(z_1-1)^4},\quad g=z_1.$$
%Then
%$$\int_{\Gamma_1}W_2= \int _{\P^1}\left((\epsilon+1)\frac{df}{f}+ (\epsilon -1)\frac{d\overline{f}}{\overline{f}}\right)\wedge \left((\epsilon+1)\frac{dg}{g}+ (\epsilon -1)\frac{d\overline{g}}{\overline{g}}\right),$$
%$$=(\epsilon+1)(\epsilon-1)\int_{\P^1}d\left(\log|f|^2\wedge \frac{d\overline{g}}{\overline{g}}+\log|f|^2\frac{dg}{g}\right)=0.$$
%The claim follows.
%\end{proof}
Thus we get the pairing
$$(i, Z_i)_{1,2}=\overline{-4\mathcal{P}(C'_i-C''_i)}=\overline{\frac{2}{\pi^2}(\mathcal{L}_3(i))},$$
where $\mathcal{L}_3$ denotes the (Bloch-Wigner) tri-logarithm. The last equality follows from Theorem 3.6 of \cite{Goncharov:cpr} once again.
% \begin{rmk}
% Again, the choice of $F=\Q(i)$ is more due to convenience, than
% anything else. The method above can be generalized for any number
% field $F$.
% \end{rmk}
\begin{rmk}
Based on Goncharov's programme (\cite{Goncharov:prAmc}) and Zagier's
conjecture \cite{Zagier:conjpoly}, one can ask whether the
intersection pairings $(\alpha
,\beta )_{p,q}$ can be written as a combination of polylogarithms.
\end{rmk}

\newcommand{\noopsort}[1]{} \newcommand{\printfirst}[2]{#1}
  \newcommand{\singleletter}[1]{#1} \newcommand{\switchargs}[2]{#2#1}
  \def\cprime{$'$}
\providecommand{\bysame}{\leavevmode\hbox to3em{\hrulefill}\thinspace}
\providecommand{\MR}{\relax\ifhmode\unskip\space\fi MR }
\providecommand{\MRhref}[2]{%
  \href{http://www.ams.org/mathscinet-getitem?mr=#1}{#2}
}
\providecommand{\href}[2]{#2}

\end{document}